\documentclass[10pt]{amsart}

\usepackage{latex_base}
\usepackage{macros}
\title[]{The tropical discriminant of a polynomial map on a plane}
\author{Boulos El Hilany}
\thanks{For this work, the author was supported by the DFG Walter Benjamin Programme: EL 1092/1-1}
\begin{document}
\maketitle
\begin{abstract} 
The discriminant of a polynomial map is central to problems from affine geometry and singularity theory.
Standard methods for characterizing it rely on elimination techniques that can often be ineffective. 
This paper concerns polynomial maps on the two-dimensional torus defined over a field of Puiseux series.
We present a combinatorial procedure for computing the tropical curve of the discriminant of maps determined by generic polynomials with given supports. 
Our results enable one to compute the Newton polytope of the discriminant of complex polynomial maps on the plane. 

\end{abstract}
 \markleft{}
 \markright{}
\section{Introduction}\label{sec:intro}
Understanding the topology of polynomial maps $f:X\to Y$ between two smooth affine varieties is essential to classification problems in algebraic geometry, global analysis and singularity theory. 
Although these have been extensively investigated in the last fifty years, numerous problems regarding polynomial maps remain unsolved (see e.g., \cite{Tib07}), such as classifying their topological types~\cite{jelonek2001topological}, describing the set of missing points $Y\setminus f(X)$~\cite{Jel99a,Hilany+2022}, and the famous Jacobian conjecture~\cite{dEss12}.
The \emph{bifurcation set} $B(f)$ of a map $f$ plays a key role in this context. 
 This is the smallest set $B$ at which the map
\[
f: X\setminus f^{-1}(B)\to Y\setminus B
\] is a locally trivial $\cC^{\infty}$ fibration. The lack of thorough methods, feasible for describing the bifurcation set, hinders progress in affine geometry. Also, there is no systematic mathematical and algorithmic framework for the study of $B(f)$ that exploits the structure and the geometry of the polynomials. In this paper, we develop a combinatorial approach for characterizing the set of critical values of polynomial maps on the plane in the above manner. 

Thom's result~\cite{Tho69} on the finiteness of $B(f)\subset\C $ for polynomial functions $\C^n\to\C$ was later generalised for polynomial maps $f=(f_1,\ldots,f_n):X\to\Bbbk^n$ defined over the field $\Bbbk$ of real or complex numbers. It is now known, due to the works of Wallace~\cite{wallace1971linear}, Varchenko~\cite{varchenko1972theorems} and Verdier~\cite{verdier1976stratifications}, that $B(f)$ is contained in a proper algebraic (or semi-algebraic in the real case) set. Several notions of regularity conditions on the fibers of polynomial maps $f:\Bbbk^m\to\Bbbk^n$ have been established and related to the $\cC^{\infty}$ property above, using sophisticated techniques from topology and singularity theory (c.f. Rabier~\cite{rabier1997ehresmann}, Gaffney~\cite{gaffney1999fibers}, Kurdyka, Orro and Simon~\cite{kurdyka2000semialgebraic}). Equivalences between all these conditions have been proven (c.f. Gaffney~\cite{gaffney1999fibers}, Dias, Ruas and Tib\u{a}r~\cite{DRT12}, and Jelonek~\cite{jelonek2003generalized,jelonek2004asymptotic}), thus establishing a versatile description for non-regularity of the fibers. The corresponding set of values is similarly shown to be a proper (semi-) algebraic set~\cite{kurdyka2000semialgebraic}. Jelonek and Kurdyka~\cite{jelonek2005quantitative} later provided an upper bound on the degree of the above set $B$, and Esterov~\cite{Est13} showed that the bifurcation set coincides with $B$ whenever a map $(\C\setminus 0)^m\to\C^n$ satisfies some genericity conditions in terms of Newton polyhedra. 



The bifurcation set can be either approximated~\cite{DTT2017}, or its equations can be computed if $f$ is a generically-finite map~\cite{Jel93,Sta02}. In fact, for any polynomial map $f$ over an algebraically closed field $K$ of characteristic zero, 
an important subset of $B(f)$ can be determined using effective methods; 
the locus in question is the set $D(f)$ of critical values of $f$, called its \emph{discriminant}. 
For polynomial maps $f:K^n\to K^n$, for instance, the closure of $D(f)$ is often an algebraic hypresurface~\cite{jelonek2005quantitative} whose equation generates the ideal
\begin{equation}\label{eq:disc_elimination}
\langle f_1-y_1,\ldots,f_n-y_n,\cj_f\rangle\cap K[y_1,\ldots,y_n],
\end{equation} where $\cj_f$ is the determinant of the Jacobian matrix of $f$ with respect to a fixed coordinate system in $K^n$. 
Standard methods from elimination theory are impractical for computing~\eqref{eq:disc_elimination} whenever the ideals are generated by polynomials with high degrees. The motivation behind this paper is to find a new procedure that describes invariants of the discriminant, such as its Newton polytope. 
\subsection{Our contribution}\label{subs:contribution}  
We consider polynomial maps over the field $\K$ of $1$-parametric complex Puiseux series $
c_0t^{r_0}+ c_1t^{r_1}+\cdots $ with ascending real exponents. It is equipped with a function $\val:\K\to\R\cup\{-\infty\}$ that satisfies $\val(0)=-\infty$ and takes 
a non-zero value to minus its lowest exponent of $t$. Then, the \emph{tropicalization} is the map $\Val:\K^n\to(\R\cup\{-\infty\})^n$ taking $\val$ coordinate-wise, and sending any algebraic hypersurface $\{P=0\}\subset(\K\setminus 0)^n$, for some polynomial $P$ defined over $\K$, onto a piece-wise linear polyhedral complex called \emph{tropical hypersurface} (see~\S\ref{sec:definitions-descriptions}). Thanks to the well-known correspondence theorem of Kapranov (c.f. Theorem~\ref{th:Kapranov}), any such tropical hypersurface can be recovered from the corner-locus of a piecewise-affine function $P^{\trop}:\R^n\to\R$, called a \emph{tropical polynomial}, that is determined entirely by $P$. Tropicalization preserves much of the data from the initial varieties in the form of polyhedral invariants. This made describing the topological invariants, singularities, Newton polytopes and other properties of algebraic varieties much easier. Tropical geometry thus became a powerful tool for tackling numerous problems in an abundance of mathematical disciplines (see e.g.,~\cite{Mik05,Kat09,BIMS15,MS15,joswig2021essentials}).

We introduce a purely combinatorial method, Theorem~\ref{thm:main}, for computing the tropicalization of the discriminant of a polynomial map on the plane for generic polynomials with fixed given supports. 
This constitutes a correspondence theorem, which, applied to our context, allows the use of simple methods for describing invariants of the discriminant of complex polynomial maps. One such invariant is the Newton polytope (see Theorem~\ref{th:Newton-polytope}).

\subsubsection{Correspondence theorem} Tropical polynomials can be constructed from classical ones by replacing the monomial term $cx^iy^j$ with $\val(c)+iu+ jv$, then taking the maximum over all such terms. The polynomials 
{\small
\begin{equation}\label{eq:main_example}
\begin{array}{ll}
 f_1:= & x- 3t^{1/2} x^2+ 4y -5xy + 6t^2y^2 ,\\
 f_2:= & 8t^{21/5}x^2 -7tx-9t^8x^3+10t^3y-11t^{3/2}xy+12t^6 x^2y -13t^4y^2 +14t^5 xy^2-15t^{9} y^3,
\end{array}
\end{equation}} 
for example, give rise to the polynomials
{\small
\begin{equation}\label{eq:main_example_trop}
\begin{array}{rl}
 F_1:= & \!\!\!\!\!\! \max (u, ~2u- 1/2,~ v,~ u+v,~ 2v-2 ) ,\\
 F_2:= & \!\!\!\!\!\!\max (u-1,~2u-21/5,~3u-8,~v-3,~u+v-3/2,~2u+v-6,~2v-4,~u+2v-5,~3v-9)\!\!\!\!\!\!
\end{array}
\end{equation}
} Accordingly, $F$ is the \emph{tropical polynomial map} corresponding to $f$, which will be denoted by $f^{\trop}$.

The \emph{support} of a polynomial is the set of exponent vectors of its monomials in $\N^n$ appearing with non-zero coefficients. Given a pair of finite subsets $A_1,A_2\subset \N^2$, we define the space $K[A_1,A_2]$ of polynomial pairs $(f_1,f_2)$ over a field $K$, where the support of $f_i$ is included in $A_i$. This space can be identified with $K^{|A_1|}\times K^{|A_2|}$ as each pair of polynomials distinguishes a tuple formed by a list of their coefficients. 


\begin{theorem}[Theorem~\ref{th:main-concise}]\label{thm:main}
Let $A_1$ and $A_2$ be two finite non-empty subsets of $\N^2\setminus \{(0,0)\}$. Then, there exists a Zariski open $\Omega\subset \K[A_1,A_2]$ consisting of pairs $f$ for which the tropicalization, $\Val(D(f))$, of the discriminant of $f$ can be computed using only the corresponding tropical polynomial map $F:=f^{\trop}$.
\end{theorem} 
Let us summarize the method from Theorem~\ref{thm:main}, whose details can be found in \S\ref{sec:proof-main-th}. The tropical polynomial map $F:\R^2\to\R^2$ induces a polyhedral-decomposition $\Xi$ of the real plane 
\[
\R^2=\bigsqcup_{\xi\in\Xi}\xi,
\] where each element in $\Xi$ is the relative interior of a polyhedron in $\R^2$, and the restriction $F|_{\xi}$ at each $\xi\subset \R^2$ is an affine map such that if $F|_{\xi}= F|_{\xi'}$, then $\xi =\xi'$ or $\overline{\xi}$ is a face of $\overline{\xi'}$ (c.f~\cite{BB13,grigoriev2022tropical}). We consider the family of all maps $f\in\K[A_1,A_2]$, whose tropicalizaltions $F:=f^{\trop}$ produce a decomposition $\Xi$ satisfying a transversality property as in Definition~\ref{def:Transv}. These form a Zariski open subset (Lemma~\ref{lem:gen-valued_transversal}). 
There are finitely-many possible combinatorial types of cells which can be obtained from tropical polynomial maps $F$ above (Definition~\ref{def:main}). We provide necessary and sufficient conditions for a type of cell to contain the valuation of a critical point of $f$ (Proposition~\ref{prop:t-crit_mixed-cells}). Any cell with this property is called \emph{critical}. We then show that the tropicalization $\tdf(\xi)$ of the set of images $f(z)$ of the critical points $z\in (\K\setminus 0)^2$ depends only on $F:= f^{\trop}$ and on the critical cell $\xi\in\Xi$ containing $\Val(z)$ (Proposition~\ref{prop:t-crit_mixed-cells}). Furthermore, the combinatorial type of critical cells $\xi$ determines $\tdf(\xi)$; 
it is either the image of $\xi$ under $F$, or a union of at most two rays emanating from $F(\xi)$ whenever the latter is a point. 
This gives rise to a function $\Phi_F:\Xi\to\pow(\R^2)$, whose image produces a piecewise-recovery of the set $\Val(D_f)$ according to a universal recipe (described in Definition~\ref{def:main}) that depends only on the tropical map $F$. The resulting method makes the computational time dependent on the sparsity of the polynomials involved rather than their degrees and the complexity of the coefficients (which are Puiseux series).




\begin{example}\label{ex:main}
The tropicalized map of~\eqref{eq:main_example} is the map~\eqref{eq:main_example_trop}. 
The tropical curve of the set $D(f)$ is represented in Figure~\ref{fig:critical_set-disriminant} in green. It can be recovered by computing $\cd_f$ using~\eqref{eq:disc_elimination}. Instead, we obtained $\Tdf$ using Theorem~\ref{thm:main} and Definition~\ref{def:main} by piecing-up images $\Phi_F(\xi)$: Each of the vertices $\beta,\gamma,\delta,\eta$ is mapped onto a pair of half-lines (one vertical and one horizontal) emanating from the vertices $b,c,d,e$, respectively. Here, we have $F(\{\beta,\gamma,\delta, \eta\}) = \{b,c,d,e\}$. We also have $\Phi_F(\{\kappa,\lambda\})$ consists of two vertical half-lines emanating from $F(\kappa)=k$ and $F(\lambda)=l$ respectively, and  $\Phi_F(\{\alpha\})$ is a horizontal half-line emanating from $F(\alpha)=a$. The bounded yellow cell $\sigma$ is mapped to the line segment $F(\sigma)$ joining $b$ to $e$. We also have $\Phi_F(]\alpha,\beta[) = F(]\alpha,\beta[) = ]a,b[$. Any one-dimensional cell $\xi$ of $\Xi$, not contained in the boundary of the yellow region, satisfies $\Phi_F(\xi) = F(\xi)$. All other non-zero-dimensional cells have an empty set as image under $\Phi_F$.
\end{example}

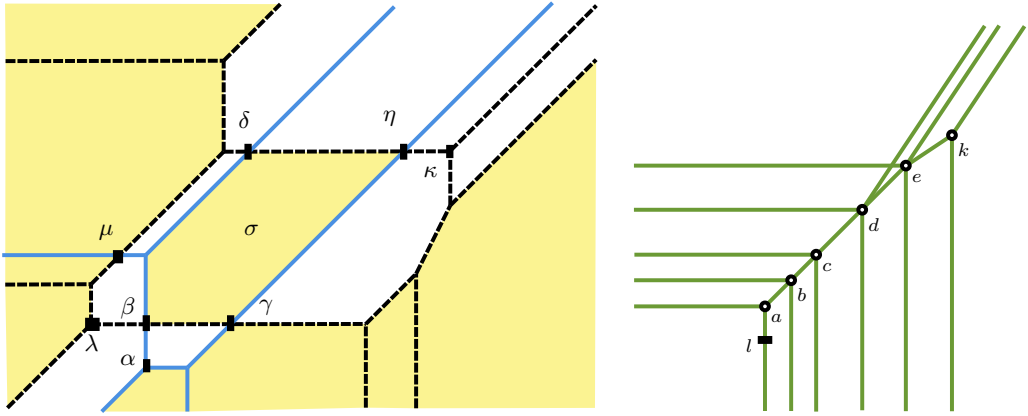
\begin{figure}[h]
\begin{tabular}{cc}
\tikzset{every picture/.style={line width=0.75pt}} 

\begin{tikzpicture}[x=.55pt,y=.55pt,yscale=-1,xscale=1]

\draw  [color={rgb, 255:red, 248; green, 231; blue, 28 }  ,draw opacity=0 ][fill={rgb, 255:red, 248; green, 231; blue, 28 }  ,fill opacity=0.48 ] (326.57,148.98) -- (426.9,48.65) -- (426.07,288.33) -- (303.07,288.58) -- (303.07,195.98) -- cycle ;
\draw  [color={rgb, 255:red, 248; green, 231; blue, 28 }  ,draw opacity=0 ][fill={rgb, 255:red, 248; green, 231; blue, 28 }  ,fill opacity=0.48 ] (268.5,230.55) -- (303.07,195.98) -- (303.07,288.58) -- (268.5,288.15) -- cycle ;
\draw  [color={rgb, 255:red, 248; green, 231; blue, 28 }  ,draw opacity=0 ][fill={rgb, 255:red, 248; green, 231; blue, 28 }  ,fill opacity=0.48 ] (145.76,260.62) -- (176.53,229.98) -- (268.5,230.55) -- (268.5,288.15) -- (145.76,290.84) -- cycle ;
\draw  [color={rgb, 255:red, 248; green, 231; blue, 28 }  ,draw opacity=0 ][fill={rgb, 255:red, 248; green, 231; blue, 28 }  ,fill opacity=0.48 ] (117.09,260.62) -- (145.76,260.62) -- (145.76,290.84) -- (86.87,290.84) -- cycle ;
\draw  [color={rgb, 255:red, 248; green, 231; blue, 28 }  ,draw opacity=0 ][fill={rgb, 255:red, 248; green, 231; blue, 28 }  ,fill opacity=0.48 ] (20.5,203.45) -- (79.5,203.2) -- (79.5,230.8) -- (22,288.3) -- cycle ;
\draw  [color={rgb, 255:red, 248; green, 231; blue, 28 }  ,draw opacity=0 ][fill={rgb, 255:red, 248; green, 231; blue, 28 }  ,fill opacity=0.48 ] (170.72,49.38) -- (170.72,111.98) -- (79.5,203.2) -- (20.5,203.45) -- (20.72,49.63) -- cycle ;
\draw  [color={rgb, 255:red, 248; green, 231; blue, 28 }  ,draw opacity=0 ][fill={rgb, 255:red, 248; green, 231; blue, 28 }  ,fill opacity=0.48 ] (209.7,10.4) -- (170.72,49.38) -- (20.72,49.63) -- (20.37,10.03) -- cycle ;
\draw  [color={rgb, 255:red, 248; green, 231; blue, 28 }  ,draw opacity=0 ][fill={rgb, 255:red, 248; green, 231; blue, 28 }  ,fill opacity=0.48 ] (187.53,112.48) -- (294.03,111.98) -- (176.53,229.98) -- (117.03,230.98) -- (117.09,183.18) -- cycle ;
\draw [color={rgb, 255:red, 0; green, 0; blue, 0 }  ,draw opacity=1 ][line width=1.5]  [dash pattern={on 3.75pt off 0.75pt}]  (22,288.3) -- (79.5,230.8) ;
\draw [color={rgb, 255:red, 74; green, 144; blue, 226 }  ,draw opacity=1 ][line width=1.5]    (117.09,260.62) -- (86.87,290.84) ;
\draw [color={rgb, 255:red, 0; green, 0; blue, 0 }  ,draw opacity=1 ][line width=1.5]  [dash pattern={on 3.75pt off 0.75pt}]  (79.5,230.8) -- (79.5,203.2) ;
\draw [color={rgb, 255:red, 0; green, 0; blue, 0 }  ,draw opacity=1 ][line width=1.5]  [dash pattern={on 3.75pt off 0.75pt}]  (20.5,203.45) -- (79.5,203.2) ;
\draw [color={rgb, 255:red, 0; green, 0; blue, 0 }  ,draw opacity=1 ][line width=1.5]  [dash pattern={on 3.75pt off 0.75pt}]  (79.5,203.2) -- (170.72,111.98) ;
\draw [color={rgb, 255:red, 0; green, 0; blue, 0 }  ,draw opacity=1 ][line width=1.5]  [dash pattern={on 3.75pt off 0.75pt}]  (79.5,230.8) -- (268.5,230.55) ;
\draw [color={rgb, 255:red, 0; green, 0; blue, 0 }  ,draw opacity=1 ][line width=1.5]  [dash pattern={on 3.75pt off 0.75pt}]  (303.07,195.98) -- (326.57,148.98) ;
\draw [color={rgb, 255:red, 0; green, 0; blue, 0 }  ,draw opacity=1 ][line width=1.5]  [dash pattern={on 3.75pt off 0.75pt}]  (268.5,230.55) -- (303.07,195.98) ;
\draw [color={rgb, 255:red, 0; green, 0; blue, 0 }  ,draw opacity=1 ][line width=1.5]  [dash pattern={on 3.75pt off 0.75pt}]  (326.57,148.98) -- (326.57,111.38) ;
\draw [color={rgb, 255:red, 0; green, 0; blue, 0 }  ,draw opacity=1 ][line width=1.5]  [dash pattern={on 3.75pt off 0.75pt}]  (170.72,111.98) -- (326.72,111.73) ;
\draw [color={rgb, 255:red, 0; green, 0; blue, 0 }  ,draw opacity=1 ][line width=1.5]  [dash pattern={on 3.75pt off 0.75pt}]  (268.5,288.15) -- (268.5,230.55) ;
\draw [color={rgb, 255:red, 0; green, 0; blue, 0 }  ,draw opacity=1 ][line width=1.5]  [dash pattern={on 3.75pt off 0.75pt}]  (303.07,288.58) -- (303.07,195.98) ;
\draw [color={rgb, 255:red, 0; green, 0; blue, 0 }  ,draw opacity=1 ][line width=1.5]  [dash pattern={on 3.75pt off 0.75pt}]  (326.57,148.98) -- (426.9,48.65) ;
\draw [color={rgb, 255:red, 0; green, 0; blue, 0 }  ,draw opacity=1 ][line width=1.5]  [dash pattern={on 3.75pt off 0.75pt}]  (170.72,49.38) -- (209.7,10.4) ;
\draw [color={rgb, 255:red, 0; green, 0; blue, 0 }  ,draw opacity=1 ][line width=1.5]  [dash pattern={on 3.75pt off 0.75pt}]  (170.72,111.98) -- (170.72,49.38) ;
\draw [color={rgb, 255:red, 0; green, 0; blue, 0 }  ,draw opacity=1 ][line width=1.5]  [dash pattern={on 3.75pt off 0.75pt}]  (20.72,49.63) -- (170.72,49.38) ;
\draw [color={rgb, 255:red, 0; green, 0; blue, 0 }  ,draw opacity=1 ][line width=1.5]  [dash pattern={on 3.75pt off 0.75pt}]  (326.57,111.38) -- (427.05,11.4) ;
\draw [color={rgb, 255:red, 74; green, 144; blue, 226 }  ,draw opacity=1 ][line width=1.5]    (117.09,183.18) -- (117.09,260.62) ;
\draw [color={rgb, 255:red, 74; green, 144; blue, 226 }  ,draw opacity=1 ][line width=1.5]    (117.09,183.18) -- (18.42,183.18) ;
\draw [color={rgb, 255:red, 74; green, 144; blue, 226 }  ,draw opacity=1 ][line width=1.5]    (117.09,183.18) -- (288.09,12.18) ;
\draw [color={rgb, 255:red, 74; green, 144; blue, 226 }  ,draw opacity=1 ][line width=1.5]    (145.76,260.62) -- (117.09,260.62) ;
\draw [color={rgb, 255:red, 74; green, 144; blue, 226 }  ,draw opacity=1 ][line width=1.5]    (145.76,260.62) -- (145.76,290.84) ;
\draw [color={rgb, 255:red, 74; green, 144; blue, 226 }  ,draw opacity=1 ][line width=1.5]    (145.76,260.62) -- (396.87,9.51) ;

\draw [color={rgb, 255:red, 0; green, 0; blue, 0 }  ,draw opacity=1 ][line width=6]    (185,112) -- (190,112) ;

\draw (185,82) node [anchor=north][inner sep=0.75pt]  [font=\small]  {$\delta$};
\draw (200,212) node [anchor=north][inner sep=0.75pt]  [font=\small]  {$\gamma$};
\draw (285,82) node [anchor=north][inner sep=0.75pt]  [font=\small]  {$\eta$};

\draw [color={rgb, 255:red, 0; green, 0; blue, 0 }  ,draw opacity=1 ][line width=6]    (292,112) -- (297,112) ;

\draw [color={rgb, 255:red, 0; green, 0; blue, 0 }  ,draw opacity=1 ][line width=6]    (173,230) -- (178,230) ;

\draw [color={rgb, 255:red, 0; green, 0; blue, 0 }  ,draw opacity=1 ][line width=6]    (115,230) -- (120,230) ;

\draw (80,236) node [anchor=north][inner sep=0.75pt]  [font=\small]  {$\lambda$};

\begin{scope}[xshift = -.05cm, yshift = -.1cm]
\draw [color={rgb, 255:red, 0; green, 0; blue, 0 }  ,draw opacity=1 ][line width=5]    (78,236) -- (88,236) ;
\end{scope}

\draw (105,250) node [anchor=north][inner sep=0.75pt]  [font=\small]  {$\alpha$};

\begin{scope}[xshift = .72cm, yshift = .45cm]
\draw [color={rgb, 255:red, 0; green, 0; blue, 0 }  ,draw opacity=1 ][line width=5]    (78,236) -- (83,236) ;
\end{scope}

\draw (105,210) node [anchor=north][inner sep=0.75pt]  [font=\small]  {$\beta$};
\draw (313,120) node [anchor=north][inner sep=0.75pt]  [font=\small]  {$\kappa$};

\begin{scope}[xshift = 4.75cm, yshift = -2.4cm]
\draw [color={rgb, 255:red, 0; green, 0; blue, 0 }  ,draw opacity=1 ][line width=5]    (78,236) -- (83,236) ;
\end{scope}

\draw (90,160) node [anchor=north][inner sep=0.75pt]  [font=\small]  {$\mu$};

\begin{scope}[xshift = .32cm, yshift = -1cm]
\draw [color={rgb, 255:red, 0; green, 0; blue, 0 }  ,draw opacity=1 ][line width=5]    (78,236) -- (85,236) ;
\end{scope}

\draw (190,160) node [anchor=north][inner sep=0.75pt]  [font=\small]  {$\sigma$};

\end{tikzpicture}

&

\tikzset{every picture/.style={line width=0.75pt}} 

\begin{tikzpicture}[x=0.55pt,y=0.55pt,yscale=-1,xscale=1]

\draw [color={rgb, 255:red, 107; green, 154; blue, 53 }  ,draw opacity=1 ][line width=1.5]    (0.33,209.69) -- (90.33,209.69) ;
\draw [color={rgb, 255:red, 107; green, 154; blue, 53 }  ,draw opacity=1 ][line width=1.5]    (90.33,209.69) -- (90.33,281.69) ;
\draw [color={rgb, 255:red, 107; green, 154; blue, 53 }  ,draw opacity=1 ][line width=1.5]    (90.33,209.69) -- (187.21,112.81) ;
\draw [color={rgb, 255:red, 107; green, 154; blue, 53 }  ,draw opacity=1 ][line width=1.5]    (0.33,191.7) -- (108.33,191.7) ;
\draw [color={rgb, 255:red, 107; green, 154; blue, 53 }  ,draw opacity=1 ][line width=1.5]    (108.33,191.7) -- (108.33,281.7) ;
\draw [color={rgb, 255:red, 107; green, 154; blue, 53 }  ,draw opacity=1 ][line width=1.5]    (0.53,174.1) -- (125.53,174.1) ;
\draw [color={rgb, 255:red, 107; green, 154; blue, 53 }  ,draw opacity=1 ][line width=1.5]    (125.53,174.1) -- (125.53,282.1) ;
\draw [color={rgb, 255:red, 107; green, 154; blue, 53 }  ,draw opacity=1 ][line width=1.5]    (0.21,112.81) -- (187.21,112.81) ;
\draw [color={rgb, 255:red, 107; green, 154; blue, 53 }  ,draw opacity=1 ][line width=1.5]    (187.21,113.81) -- (187.21,281.81) ;
\draw [color={rgb, 255:red, 107; green, 154; blue, 53 }  ,draw opacity=1 ][line width=1.5]    (0.21,143.19) -- (157.21,143.19) ;
\draw [color={rgb, 255:red, 107; green, 154; blue, 53 }  ,draw opacity=1 ][line width=1.5]    (157.21,143.19) -- (157.21,281.19) ;
\draw [color={rgb, 255:red, 107; green, 154; blue, 53 }  ,draw opacity=1 ][line width=1.5]    (157.1,143.28) -- (241.47,16.63) ;
\draw [color={rgb, 255:red, 107; green, 154; blue, 53 }  ,draw opacity=1 ][line width=1.5]    (187.22,112.99) -- (218.92,91.86) ;
\draw [color={rgb, 255:red, 107; green, 154; blue, 53 }  ,draw opacity=1 ][line width=1.5]    (218.92,91.86) -- (218.92,281.86) ;
\draw [color={rgb, 255:red, 107; green, 154; blue, 53 }  ,draw opacity=1 ][line width=1.5]    (218.92,91.86) -- (268.86,16.86) ;
\draw [color={rgb, 255:red, 107; green, 154; blue, 53 }  ,draw opacity=1 ][line width=1.5]    (187.21,112.81) -- (251.43,16.43) ;
\draw  [fill={rgb, 255:red, 255; green, 255; blue, 255 }  ,fill opacity=1 ][line width=1.5]  (87.68,209.69) .. controls (87.68,208.22) and (88.87,207.03) .. (90.33,207.03) .. controls (91.8,207.03) and (92.99,208.22) .. (92.99,209.69) .. controls (92.99,211.16) and (91.8,212.35) .. (90.33,212.35) .. controls (88.87,212.35) and (87.68,211.16) .. (87.68,209.69) -- cycle ;
\draw  [fill={rgb, 255:red, 255; green, 255; blue, 255 }  ,fill opacity=1 ][line width=1.5]  (105.68,191.7) .. controls (105.68,190.23) and (106.87,189.04) .. (108.33,189.04) .. controls (109.8,189.04) and (110.99,190.23) .. (110.99,191.7) .. controls (110.99,193.17) and (109.8,194.36) .. (108.33,194.36) .. controls (106.87,194.36) and (105.68,193.17) .. (105.68,191.7) -- cycle ;
\draw  [fill={rgb, 255:red, 255; green, 255; blue, 255 }  ,fill opacity=1 ][line width=1.5]  (122.88,174.1) .. controls (122.88,172.63) and (124.07,171.44) .. (125.53,171.44) .. controls (127,171.44) and (128.19,172.63) .. (128.19,174.1) .. controls (128.19,175.57) and (127,176.76) .. (125.53,176.76) .. controls (124.07,176.76) and (122.88,175.57) .. (122.88,174.1) -- cycle ;
\draw  [fill={rgb, 255:red, 255; green, 255; blue, 255 }  ,fill opacity=1 ][line width=1.5]  (154.55,143.19) .. controls (154.55,141.72) and (155.74,140.53) .. (157.21,140.53) .. controls (158.68,140.53) and (159.87,141.72) .. (159.87,143.19) .. controls (159.87,144.66) and (158.68,145.85) .. (157.21,145.85) .. controls (155.74,145.85) and (154.55,144.66) .. (154.55,143.19) -- cycle ;
\draw  [fill={rgb, 255:red, 255; green, 255; blue, 255 }  ,fill opacity=1 ][line width=1.5]  (184.55,112.81) .. controls (184.55,111.34) and (185.74,110.15) .. (187.21,110.15) .. controls (188.68,110.15) and (189.87,111.34) .. (189.87,112.81) .. controls (189.87,114.28) and (188.68,115.47) .. (187.21,115.47) .. controls (185.74,115.47) and (184.55,114.28) .. (184.55,112.81) -- cycle ;
\draw  [fill={rgb, 255:red, 255; green, 255; blue, 255 }  ,fill opacity=1 ][line width=1.5]  (216.27,91.86) .. controls (216.27,90.39) and (217.45,89.2) .. (218.92,89.2) .. controls (220.39,89.2) and (221.58,90.39) .. (221.58,91.86) .. controls (221.58,93.33) and (220.39,94.52) .. (218.92,94.52) .. controls (217.45,94.52) and (216.27,93.33) .. (216.27,91.86) -- cycle ;

\draw (110.33,195.1) node [anchor=north west][inner sep=0.75pt]  [font=\scriptsize]  {$b$};
\draw (127.53,177.5) node [anchor=north west][inner sep=0.75pt]  [font=\scriptsize]  {$c$};
\draw (159.21,146.59) node [anchor=north west][inner sep=0.75pt]  [font=\scriptsize]  {$d$};
\draw (189.22,116.39) node [anchor=north west][inner sep=0.75pt]  [font=\scriptsize]  {$e$};
\draw (220.92,95.26) node [anchor=north west][inner sep=0.75pt]  [font=\scriptsize]  {$k$};
\draw (92.33,213.09) node [anchor=north west][inner sep=0.75pt]  [font=\scriptsize]  {$a$};

\draw (75,230) node [anchor=north west][inner sep=0.75pt]  [font=\scriptsize]  {$l$};
\begin{scope}[xshift = 1.65cm, yshift = 4.5cm]
\draw [color={rgb, 255:red, 0; green, 0; blue, 0 }  ,draw opacity=1 ][line width=3]    (0,0) -- (10,0) ;
\end{scope}

\end{tikzpicture}
\end{tabular}
\caption{\textbf{(L)} The union of the tropical curves of the polynomials $f_1$ and $f_2$ appearing in Equation~\eqref{eq:main_example}, and forming a subdivision $\Xi$. \textbf{(R)} The tropical curve $\Tdf $
}\label{fig:critical_set-disriminant}
\end{figure}

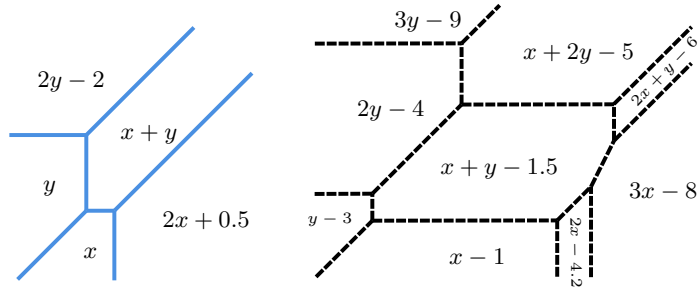
\begin{figure}[h]

\tikzset{every picture/.style={line width=0.75pt}} 

\begin{tikzpicture}[x=0.37pt,y=0.37pt,yscale=-1,xscale=1]

\draw [color={rgb, 255:red, 0; green, 0; blue, 0 }  ,draw opacity=1 ][line width=1.5]  [dash pattern={on 3.75pt off 0.75pt}]  (312,288.96) -- (369.5,231.46) ;
\draw [color={rgb, 255:red, 74; green, 144; blue, 226 }  ,draw opacity=1 ][line width=1.5]    (77.09,221.28) -- (6.84,291.53) ;
\draw [color={rgb, 255:red, 0; green, 0; blue, 0 }  ,draw opacity=1 ][line width=1.5]  [dash pattern={on 3.75pt off 0.75pt}]  (369.5,231.46) -- (369.5,203.86) ;
\draw [color={rgb, 255:red, 0; green, 0; blue, 0 }  ,draw opacity=1 ][line width=1.5]  [dash pattern={on 3.75pt off 0.75pt}]  (310.5,204.11) -- (369.5,203.86) ;
\draw [color={rgb, 255:red, 0; green, 0; blue, 0 }  ,draw opacity=1 ][line width=1.5]  [dash pattern={on 3.75pt off 0.75pt}]  (369.5,203.86) -- (460.72,112.65) ;
\draw [color={rgb, 255:red, 0; green, 0; blue, 0 }  ,draw opacity=1 ][line width=1.5]  [dash pattern={on 3.75pt off 0.75pt}]  (369.5,231.46) -- (558.5,231.21) ;
\draw [color={rgb, 255:red, 0; green, 0; blue, 0 }  ,draw opacity=1 ][line width=1.5]  [dash pattern={on 3.75pt off 0.75pt}]  (593.07,196.65) -- (616.57,149.65) ;
\draw [color={rgb, 255:red, 0; green, 0; blue, 0 }  ,draw opacity=1 ][line width=1.5]  [dash pattern={on 3.75pt off 0.75pt}]  (558.5,231.21) -- (593.07,196.65) ;
\draw [color={rgb, 255:red, 0; green, 0; blue, 0 }  ,draw opacity=1 ][line width=1.5]  [dash pattern={on 3.75pt off 0.75pt}]  (616.57,149.65) -- (616.57,112.05) ;
\draw [color={rgb, 255:red, 0; green, 0; blue, 0 }  ,draw opacity=1 ][line width=1.5]  [dash pattern={on 3.75pt off 0.75pt}]  (460.72,112.65) -- (616.72,112.4) ;
\draw [color={rgb, 255:red, 0; green, 0; blue, 0 }  ,draw opacity=1 ][line width=1.5]  [dash pattern={on 3.75pt off 0.75pt}]  (558.5,288.81) -- (558.5,231.21) ;
\draw [color={rgb, 255:red, 0; green, 0; blue, 0 }  ,draw opacity=1 ][line width=1.5]  [dash pattern={on 3.75pt off 0.75pt}]  (593.07,289.25) -- (593.07,196.65) ;
\draw [color={rgb, 255:red, 0; green, 0; blue, 0 }  ,draw opacity=1 ][line width=1.5]  [dash pattern={on 3.75pt off 0.75pt}]  (616.57,149.65) -- (695.78,70.43) ;
\draw [color={rgb, 255:red, 0; green, 0; blue, 0 }  ,draw opacity=1 ][line width=1.5]  [dash pattern={on 3.75pt off 0.75pt}]  (460.72,50.05) -- (499.7,11.06) ;
\draw [color={rgb, 255:red, 0; green, 0; blue, 0 }  ,draw opacity=1 ][line width=1.5]  [dash pattern={on 3.75pt off 0.75pt}]  (460.72,112.65) -- (460.72,50.05) ;
\draw [color={rgb, 255:red, 0; green, 0; blue, 0 }  ,draw opacity=1 ][line width=1.5]  [dash pattern={on 3.75pt off 0.75pt}]  (310.72,50.3) -- (460.72,50.05) ;
\draw [color={rgb, 255:red, 0; green, 0; blue, 0 }  ,draw opacity=1 ][line width=1.5]  [dash pattern={on 3.75pt off 0.75pt}]  (616.57,112.05) -- (695.93,33.18) ;
\draw [color={rgb, 255:red, 74; green, 144; blue, 226 }  ,draw opacity=1 ][line width=1.5]    (77.09,143.84) -- (77.09,221.28) ;
\draw [color={rgb, 255:red, 74; green, 144; blue, 226 }  ,draw opacity=1 ][line width=1.5]    (77.09,143.84) -- (-1.58,143.84) ;
\draw [color={rgb, 255:red, 74; green, 144; blue, 226 }  ,draw opacity=1 ][line width=1.5]    (77.09,143.84) -- (187.13,33.8) ;
\draw [color={rgb, 255:red, 74; green, 144; blue, 226 }  ,draw opacity=1 ][line width=1.5]    (105.76,221.28) -- (77.09,221.28) ;
\draw [color={rgb, 255:red, 74; green, 144; blue, 226 }  ,draw opacity=1 ][line width=1.5]    (105.76,221.28) -- (105.76,292.51) ;
\draw [color={rgb, 255:red, 74; green, 144; blue, 226 }  ,draw opacity=1 ][line width=1.5]    (105.76,221.28) -- (246.51,80.53) ;

\draw (71,256.4) node [anchor=north west][inner sep=0.75pt]    {\small $x$};
\draw (30,187.4) node [anchor=north west][inner sep=0.75pt]    {\small $y$};
\draw (109,130.4) node [anchor=north west][inner sep=0.75pt]    {\small $x+y$};
\draw (152,219.4) node [anchor=north west][inner sep=0.75pt]    {\small $2x+0.5$};
\draw (24,75.6) node [anchor=north west][inner sep=0.75pt]    {\small $2y-2$};
\draw (445,257.4) node [anchor=north west][inner sep=0.75pt]    {\small $x-1$};
\draw (300,220.4) node [anchor=north west][inner sep=0.75pt]    {\tiny $y-3$};
\draw (436,165.4) node [anchor=north west][inner sep=0.75pt]    {\small $x+y-1.5$};
\draw (585.57,224.37) node [anchor=north west][inner sep=0.75pt]  [rotate=-89.01]  {\tiny $2x-4.2$};
\draw (351,105.4) node [anchor=north west][inner sep=0.75pt]    {\small $2y-4$};
\draw (628.39,107.76) node [anchor=north west][inner sep=0.75pt]  [rotate=-313.75]  {\tiny $2x+y-6$};
\draw (520.5,47.6) node [anchor=north west][inner sep=0.75pt]    {{\small $x+2y-5$}};
\draw (629,189.8) node [anchor=north west][inner sep=0.75pt]    {{\small $3x-8$}};
\draw (389,15) node [anchor=north west][inner sep=0.75pt]    {\small $3y-9$};

\end{tikzpicture}
\caption{The tropical curves corresponding to the polynomials~\eqref{eq:main_example_trop}.}\label{fig:two_tropica_curves}
\end{figure}


\subsubsection{Newton polytope for complex maps}
The proof of the following result is constructive and gives rise to a recipe (see \S\ref{sec:Newton-pol}) for computing the Newton polytope of a complex polynomial map. Further possible applications of Theorem~\ref{thm:main} to problems from affine geometry are discussed in \S\ref{sec:future_work}.

\begin{theorem}\label{th:Newton-polytope}
Let $A_1$ and $A_2$ be two finite non-empty subsets of $\N^2\setminus \{(0,0)\}$. Then, there exists a Zariski open $\tilde{\Omega}\subset \C[A_1,A_2]$ consisting of pairs $f$ for which Theorem~\ref{thm:main} can be used to recover the Newton polytope of the discriminant $D(f)$ without computing $\cd_f$.
%
\end{theorem} 
All the proofs in this paper rely on classical results from tropical and toric geometries~\cite{Kat09,Rab12,OP13}, which are formulated in a higher-dimensional setting. Subsequently, albeit requiring much more elaborate analysis, our results are expected to extend for maps over a space of arbitrary dimension. 
\subsection{Related work} 
Due to their connection to ReLu neural networks~\cite{ZNL18}, tropical rational maps $\R^n\to\R^n$ have been the subject of close attention~\cite{grigoriev2022tropical}. Grigoriev and Radchenko have considered a tropical version of the Jacobian conjecture for tropical rational maps, and have shown that it holds true~\cite{grigoriev2022tropical}. In particular, they provided a sufficient condition for a tropical polynomial map to be an isomorphism. 

The \emph{tropical $A$-discriminant} $\cT\nabla_A$ is the subset of all tropical polynomials with a fixed support $A$, whose tropical hypersurface lifts to a singular hypersurface over $\K$~\cite{DES07}. These \emph{tropical singularities} (c.f.~\cite{markwig2012tropical}) 
can be effectively tested through an operation called \emph{Euler derivative}~\cite{DiTa12}. One can then use the
``Cayley trick'' to classify tropical non-transversal intersections between two tropical curves (see e.g., Definition~\ref{def:Transv} or~\cite{alonso2008tropical}) by expressing them as singular tropical hypersurfaces~\cite[\S5]{DiTa12}. The tropical discriminant of a polynomial map is a more subtle concept as the polynomials themselves are fixed, while we describe the valuation of their corresponding two constant terms that give rise to the above singularity. Accordingly, the discriminant of a map can be viewed as a one-dimensional subset of $\cT\nabla_A$.
%

\subsection{Organization of the paper} In \S\ref{sec:definitions-descriptions} we introduce notations and known results related to tropical geometry. 
In \S\ref{sec:class-trop-maps} we give an explicit description of the function $\Phi$ and the set $\Omega$ from Theorem~\ref{thm:main}. We also establish the necessary technical results and definitions that relate critical points and critical values of the map $f$ with their tropical analogues for $f^{\trop}$. These will be central to the proof of Theorem~\ref{thm:main} in \S\ref{sec:proof-main-th}. We prove Theorem~\ref{th:Newton-polytope} in \S\ref{sec:Newton-pol} by showing how to compute the Newton polytope of the discriminant of polynomial maps on the complex $2$-torus. \S\ref{sec:tropical_critical} and \S\ref{sec:supercr} are devoted for the proofs of the two main technical results, Propositions~\ref{prop:t-crit_mixed-cells} and~\ref{prop:super-critical-mixed-cells}, introduced in~\S\ref{sec:class-trop-maps}. In \S\ref{sec:future_work} we present several possible research directions aiming at applications of Theorem~\ref{thm:main} to open problems from affine geometry. The proof of Proposition~\ref{prop:t-crit_mixed-cells} is a case-by-case analysis, and its repetitive parts are left as an appendix.

\subsection*{Acknowledgments} The author is grateful to Timo de Wolff for his valuable remarks on the presentation of the manuscript.

\section{Some basics of tropical geometry}\label{sec:definitions-descriptions}

We state in this section some well-known facts about tropical geometry (see e.g.~\cite{BIMS15,richter2005first,MS15}, and the references therein). Some of the exposition and notations here are taken from~\cite{Br-deMe11,BB13,EH18}. We start by introducing in \S\ref{sub:base-field} the field $\K$ over which the polynomial maps will be defined. We will describe in \S\ref{sub:trop-polynomials_hypersurfaces} \emph{tropical hypersurfaces} and their relations to zero loci of polynomials over $\K$. Then, we define in \S\ref{subs:tropical_polynomials} \emph{tropical polynomials}, and describe the corresponding subdivisions obtained. 

\subsection{The base field}\label{sub:base-field}  A \emph{locally convergent generalized Puiseux series} is a formal series of the form 
\[
c(t)=\underset{r\in R}{\sum} c_rt^r,
\] where $R\subset\mathbb{R}$ is a well-ordered set, all $c_r\in\mathbb{C}$, and the series is convergent for $t>0$ small enough. We denote by $\K$ the set of all locally convergent generalized Puiseux series. It is an algebraically-closed field of characteristic 0~\cite{markwig2009field}, and can be equipped with the function 
\[
  \begin{array}{lccc}
    \displaystyle \val: & \mathbb{K} & \longrightarrow & \mathbb{R}\cup\{-\infty\} \\
    
    \displaystyle \ & 0 & \longmapsto  & -\infty \\
    
     \displaystyle \ & \underset{r\in R}{\sum} c_rt^r\neq 0 & \longmapsto  & -\min_R\{r\ |\ c_r\neq 0\}.
  \end{array}
\] We will call $\val$ the \emph{valuation function} or simply, \emph{valuation}. This extends to a map $\Val:\K^n\rightarrow(\mathbb{R}\cup\{-\infty\})^n$ by evaluating $\val$ coordinate-wise, i. e., $\Val(z_1,\ldots , z_n)=(\val(z_1),\ldots , \val(z_n))$.

\begin{remark}
In the standard literature (c.f.~\cite{efrat2008valuations}), the valuation of an element in $\K$ is defined as $-\val$. Our choice for it to be the opposite makes more geometrical the duality between subdivisions of Newton polytopes and tropical curves (c. f., \S\ref{subsub:subdivisions2}). 
\end{remark}

\subsection{Tropical hypersurfaces}\label{sub:trop-polynomials_hypersurfaces}

Let $f$ be a polynomial in $\K[z_1,\ldots,z_n]$.

\begin{notation}\label{not:supp}
We will express $f$ as a linear combination 
\[
\sum_{w\in \mathcal{W}} c_a z^a
\]
of monomials $z^a:=z_1^{a_1}\cdots z_n^{a_n}$, where $A$ is a finite subset of $\N^n$, and $c_a\in \K\setminus 0$. We call $A$ the \emph{support} of $f$. The \emph{Newton polytope}, $\cN(f)$, is the convex hull of $A$ in $\R^n$. The notation $\VK(f)$ refers to the zero set $\{z\in(\K\setminus 0)^n~|~f(z)=0\}$. The \emph{tropical hypersurface of $f$} is the subset in $\R^n$ defined as $\VT (f):=\Val(\VK (f)).$
\end{notation}
Consider the map
\[
  \begin{array}{lccc}
    \displaystyle \nu_f: & \mathbb{Z}^n & \longrightarrow & \mathbb{R}\cup\{-\infty\} \\
    
    \displaystyle \ & a & \longmapsto  &  \begin{cases}
      \val(c_a), & \text{if $ a\in A$,}  \\
      -\infty, & \text{otherwise.}
    \end{cases} \\
    
  \end{array}
\] Its \emph{Legendre transform} is a piecewise-linear convex function 
\[
  \begin{array}{lccc}
    \displaystyle \mathcal{L}(\nu_f): & \R^n & \longrightarrow & \mathbb{R} \\
    
    \displaystyle \ & x & \longmapsto  &\displaystyle \max_{a\in A}\{\langle x, a\rangle + \nu_f(a)\}, \\
    
  \end{array}
\] where $\langle ~,~\rangle:~\R^n\times\R^n\rightarrow\R$ is the standard Eucledian product. The set of points $x\in\R^n$ at which $\mathcal{L}(\nu_f)$ is not differentiable is called the \emph{corner locus} of $\mathcal{L}(\nu_f)$. We have the famous fundamental Theorem of Kapranov~\cite{Kap00},~\cite[Theorem 3.13]{MS15}.

\begin{theorem}[Kapranov]\label{th:Kapranov}
The tropical hypersurface $\VT (f)$ of a polynomial $f$ defined over $\K$ is the corner locus of its Legendre transform $\mathcal{L}(\nu_f)$.
\end{theorem}

This theorem gives rise to the following consequence for the polynomial $f$ above. 
\begin{notation}\label{not:restrictions}
For any $c\in \K$, we set $\overline{c}:=\alpha_{-\val(c)}$ (i. e., the the coefficient in $\C$ of the first term following the increasing order of the exponents of $t$ in $c$) if $c\neq 0$, and $\overline{c}:= 0$ otherwise.
We extend this notation to points $z\in \K^n$, by writing $\overline{z}$ in reference to $(\overline{z}_1,\ldots,\overline{z}_n)\in\R^n$. Given a polynomial $f$ with support $A$, we define the restriction $f_\delta$ to any subset $\delta\subset A$ as the polynomial $\sum_{a\in \delta}c_az^a$.
The notation $\overline{f}$ refers to the restriction $\sum\overline{c}_az^a$, where $a$ runs over all points in $A$ that maximize $\val (c_a)$. For example, if $f = t^5(2 + 3t) -5t^5z_1 + 16 t^7z_2 $, then $\overline{f} = 2 -5z_1$.
\end{notation} 

\begin{corollary} \label{cor:lower-order}
The polynomial $f$ has a solution $\tilde{z}\in(\K\setminus 0)^n$ with $\Val(\tilde{z}) = (0,\ldots,0)$ if and only if $\overline{f}$ has at least two non-zero monomial terms.
\end{corollary}

\begin{proof}
 An above solution $\tilde{z}$ exists if and only if $(0,\ldots,0)$ belongs to $\VT(f)$ (Theorem~\ref{th:Kapranov}). This holds true if and only if $\mathcal{L}(\nu_f)$ attains its maximum at two different terms simultaneously. This is equivalent to $\overline{f}$ having at least two non-zero monomial terms. 
\end{proof}

\subsection{Subdivisions from tropical polynomials}\label{subs:tropical_polynomials} For any two values $a,b\in\mathbb{T}:=\R\cup\{-\infty\}$, their \emph{tropical summation} $a\oplus b$ is defined as their maximum $\max (a,b)$, and their \emph{tropical multiplication} $a \otimes b$ is their usual sum $a+b$. This gives rise to a \emph{tropical semi-field} $(\mathbb{T},\oplus ,\otimes)$, where $\max(a,-\infty) = a$, and $ -\infty + a = -\infty$. A \emph{tropical polynomial} $F$ is defined over the semifield $(\mathbb{T},\oplus ,\otimes)$, which gives rise to a function 
\[
  \begin{array}{lccc}
    F: & \mathbb{T}^n & \longrightarrow & \mathbb{T} \\
    
 	 & x & \longmapsto  &\displaystyle \max_{a\in A}\{\langle x, a\rangle + \gamma_a\}, \\
    
  \end{array}
\] where $A$ is a finite set containing all $a\in\N^n$ for which $\gamma_a\in\R$. The set $A$ is called the \emph{support} of the tropical polynomial $F$, and the linear terms appearing in $F$ are called the \emph{tropical monomials}. 
The \emph{tropicalization} of the polynomial $f$ above is the tropical polynomial 
\[
f^{\trop}(x):=\max_{a\in A}\{\langle x, a\rangle +\val(c_a)\}.
\] This coincides with the piecewise-linear convex function $\mathcal{L}(\nu_f)$ defined above, and thus Theorem~\ref{th:Kapranov} asserts that $\VT(f)$ is the corner locus of $f^{\trop}$. Conversely, the corner locus of any tropical polynomial is a tropical hypersurface.

%
%

\subsubsection{Regular polyhedral subdivisions}\label{subsub:subdivisions}  

All polytopes in this paper are assumed to be convex.
\begin{definition}\label{def:polyh-subdiv}
A \emph{polyhedral subdivision} of a polytope $\Delta\subset\R^n$ is a set of polytopes $\{\Delta_i\}_{i\in I}$ satisfying $\cup_{i\in I}\Delta_i=\Delta$, and if $i,j\in I$, then $\Delta_i\cap\Delta_j$ is either empty or it is a common face of $\Delta_i$ and $\Delta_j$.
\end{definition}

\begin{definition}\label{def:Regular}
Let $\Delta$ be a polytope in $\R^n$ and let $\tau$ denote a polyhedral subdivision of $\Delta$ consisting of polytopes. We say that $\tau$ is \emph{regular} if there exists a continuous, convex, piecewise-linear function $\varphi:\Delta\rightarrow \R$ such that the polytopes of $\tau$ are exactly the domains of linearity of $\varphi$.
\end{definition} 

 Let $\Delta$ be an integer polytope in $\R^n$ and let $\varphi:~\Delta~\cap~\mathbb{Z}^n\rightarrow \R$ be a function. We denote by $\hat{\Delta}(\phi)$ the convex hull of the set 
\[
\{(a,~\varphi(a))\in\R^{n+1}~|~a\in\Delta\cap\mathbb{Z}^n\}.
\] Then the polyhedral subdivision of $\Delta$ induced by projecting the union of the lower faces of $\hat{\Delta}(\varphi)$ onto the first $n$ coordinates, is regular.


%
%
%
%
%


\subsubsection{Subdivisions and their duals}\label{subsub:subdivisions2}
Keeping with the same notation as in \S\ref{sub:trop-polynomials_hypersurfaces}, the tropical hypersurface $\VT(f)$ is an $(n-1)$-dimensional piecewise-linear complex which produces a \emph{polyhedral subdivision} $\Xi$ of $\R^n$. 
This is a finite collection of the relative interiors of polyhedra in $\R^n$, whose closures satisfy Definition~\ref{def:polyh-subdiv}.
Each element of $\Xi$ is called a \emph{cell}. The $n$-dimensional cells of $\Xi$, are the connected components of the complement of $\VT(f)$ in $\R^n$. 
All together, cells of dimension less than $n$ form the domains of linearity of $f^{\trop}$ at $\VT(f)$. 

The subdivision $\Xi$ induces a regular subdivision $\tau$ of the Newton polytope $\cN(f)$ of $f$ in the following way (see also~\cite[Section 3]{BB13}). Given a cell $\xi$ of $\VT(f)$ and a point $x$ in $\xi$, the set 
\[
\mathcal{I}_{\xi}:=\{a\in A~|~f^{\trop}(x)=\langle x,a\rangle + \val(c_a)\}
\] does not depend on $x$. All together the polyhedra $\delta(\xi)$, defined as the convex hull of $\mathcal{I}_{\xi}$ form a subdivision $\tau$ of $\cN(f)$ called the \emph{dual subdivision}, and the polyhedron $\delta(\xi)$ is called the \emph{dual} of $\xi$.  
An analogous description holds true if we consider two polynomials (c.f. Figure~\ref{fig:subdivisions})
 $f_1,f_2\in\mathbb{K}[z_1,z_2]$. 
The \emph{Minkowski sum} of any two subsets $A,B\subset\R^n$ is the coordinate-wise sum
\[
A+B:=\{a+b~|~a\in A,~b\in B\}.
\] Let $A_1,A_2\subset\mathbb{Z}^2$, $\Delta_1,\Delta_2\subset\R^2$, and $T_1,T_2\subset\R^2$ be their respective supports, Newton polytopes, and tropical curves, respectively. The union of $T_1\cup T_2$ defines a polyhedral subdivision $\Xi$ of $\R^2$. Any non-empty cell of $\Xi$ can be written as 
\[
\xi=\xi_1\cap\xi_2
\] with $\xi_i\in\Xi_i$, where $\Xi_i$ is the polyhedral subdivision of $\R^2$ produced by $T_i$ ($i=1,2$). Any cell $\xi\in\Xi$ can be uniquely written in this way. Similarly, the polyhedral subdivision induces a \emph{mixed dual subdivision} $\tau$  of the Minkowski sum $\Delta=\Delta_1 +\Delta_2$ in the following way. Any polytope $\delta\in\tau$ is equipped with a unique representation $\delta=\delta_1+\delta_2$ with $\delta_i\in\tau_i$, where each $\tau_i$ is the dual subdivision of $\Delta_i$ ($i=1,2$). The above duality-correspondence applied to the (tropical) product of the tropical polynomials gives rise to the following well-known fact (see e.g.,~\cite[\S3~\&~4]{BB13}).

\begin{proposition}\label{prop:mixed_subd}
There is a one-to-one duality correspondence between $\Xi$ and $\tau$, which reverses the inclusion relations, and such that if $\delta\in\tau$ corresponds to $\xi\in\Xi$, then

\begin{enumerate}

 \item $\xi=\xi_1\cap\xi_2$ with $\xi_i\in\Xi_i$ for $i=1,2$, then  $\delta = \delta_1 + \delta_2$ with $\delta_i\in\tau_i$ is the polytope dual to $\xi_i$.
 
 \item $\dim\xi + \dim\delta = 2$,
 
 \item the cell $\xi$ and the polytope $\delta$ span orthonogonal real affine spaces in $\R^2$,
 
 \item the cell $\xi$ is unbounded if and only if $\delta$ lies on a proper face of $\Delta$.

\end{enumerate}
\end{proposition}

\begin{figure}[h]
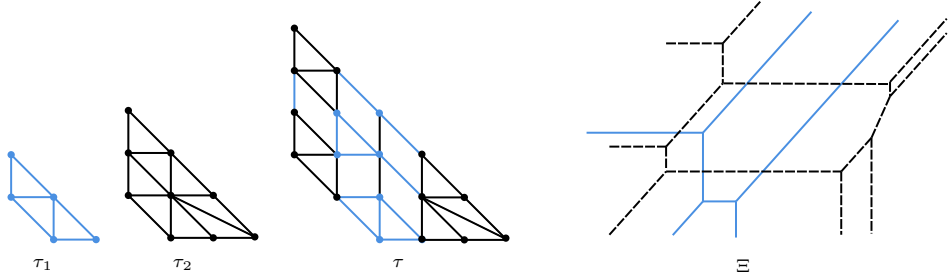

\center

\tikzset{every picture/.style={line width=0.75pt}} 


\caption{ The dual subdivisions $\tau_1$, $\tau_2$, and $\tau$ of $\cN(f_1)$, $\cN(f_2)$, and $\cN(f_1)+\cN(f_2)$, respectively, together with the subdivision $\Xi$ of $\R^2$ induced by the tropical curves $T_1$ and $T_2$. Here, $f_1$ and $f_2$ are the polynomials from~\eqref{eq:main_example}.}\label{fig:subdivisions}
\end{figure}

\begin{definition}[stable cells]\label{def:Transv} The cell $\xi\in\Xi$ is \emph{stable} if the equality $\dim(\delta)=\dim(\delta_1 )+\dim(\delta_2)$ holds, and for $i=1,2$, the point 
$(a,~-\nu_{f_i}(a))$ is either a vertex of the convex hull of 
\[
\{(a,~-\nu_{f_i}(a))\in\N^{2}\times \R~|~a\in\delta_i\},
\] or does not belong to it. We say that $\Xi$ is \emph{stable} if so are all of its cells. A point $x\in\R^2$ is a \emph{stable intersection point} of $T_1$ and $T_2$ if $x\in T_1\cap T_2$, and $x$ is a stable cell of $\Xi$. We say that the intersection $T_1\cap T_2$ is \emph{stable} if it consists of stable intersection points.
\end{definition}

\begin{notation}
For any $\xi\in\Xi$ as in Proposition~\ref{prop:mixed_subd}, we will use $\delta(\xi)$ to denote the polytope $\delta$, dual to $\xi$ and we use $\delta(\xi_i)$ ($i=1,2$) to denote the polytope $\delta_i$, dual to $\xi_i$.
\end{notation}

\section{Tropical polynomial maps on the plane}\label{sec:class-trop-maps} 

Let $A_1$ and $A_2$ be two finite subsets in $\N^2\setminus\{(0,0)\}$, and consider two polynomials $f_1,f_2\in\K[z_1,z_2]$, where $A_i$ is the support of $f_i$ ($i=1,2$). We say that the map $f:=(f_1,f_2):\TK\to\TK$ is \emph{dominant} if there exists a point $\tilde{z}\in \TK$, where the Jacobian matrix $\Jac_{\tilde{z}} f$, evaluated at $\tilde{z}$, is non-singular. In other words, the determinant of $\Jac_{z} f$ is a non-zero polynomial in $\K[z_1,z_2]$. In the rest of this section, we assume that $f$ is a dominant polynomial map. 

\begin{notation}\label{not:maps}
We use $A$ to denote the pair $(A_1,A_2)$ of supports of $f$, and $\K[A]$ to denote the space $\K[A_1,A_2]$. The map $f$ gives rise to a tropical polynomial map $F:=(F_1,F_2):\R^2\to\R^2$, where $F_1:=f^{\trop}_1$ and $F_2:=f^{\trop}_2$. Recall that $F_1$, and $F_2$ give rise to three polyhedral subdivisions of $\R^2$, denoted by $\Xi_1$, $\Xi_2$, and $\Xi$ (see \S\ref{subsub:subdivisions2}): $\Xi_1$, $\Xi_2$ are induced by the tropical curves $T_1:=\VT(f_1)$, $T_2:=\VT(f_2)$, and $\Xi$ is induced by $T_1\cup T_2$. We use $\tau_1$, $\tau_2$, and $\tau$ denote the subdivisions of $\cN(f_1)$, $\cN(f_2)$, and $\cN(f_1)+\cN(f_2)$, dual to $\Xi_1$, $\Xi_2$, and $\Xi$, respectively (see e.g. Figure~\ref{fig:subdivisions}). 
\end{notation}

In \S\ref{subsub:cell-types}, we classify the cells appearing in $\Xi$. The distinction on the type of each cell $\xi$, is made using the sizes and the mutual disposition of the two polyhedra $\delta(\xi_1)$ and $\delta(\xi_2)$ (see Definition~\ref{def:lat-diag}). This classification is necessary to introduce the function $\Phi_F$ in Definition~\ref{def:main}. 

The set $\Omega$ from Theorem~\ref{thm:main} is the intersection of three Zariski open subsets in $\K[A]$. The first one, introduced in \S\ref{sub:valuation-space}, encodes all maps inducing a stable subdivision $\Xi$ where $\tau_i$ ($i=1,2$) is a triangulation. The other two subsets of $\K[A]$ are introduced in \S\ref{sub:tropical_critical}; one is the largest subset containing all maps $f$, whose tropicalization of their critical locus depends only on $F$ (see Lemma~\ref{lem:crit-points-invar}), and the third one (introduced in Proposition~\ref{prop:t-crit_mixed-cells}) describes further conditions on the tropical curve of these critical loci. 

We conclude this section by introducing in \S\ref{subs:images_tropical} two technical results (Lemma~\ref{lem:val-F_F-val} and Proposition~\ref{prop:super-critical-mixed-cells}) on the tropical curve of the discriminant of $f$.  Lemma~\ref{lem:val-F_F-val}, and Propositions~\ref{prop:t-crit_mixed-cells} and~\ref{prop:super-critical-mixed-cells} are crucial for the proof of Theorem~\ref{thm:main} in \S\ref{sec:proof-main-th}. Their proofs are postponed to \S\ref{sec:tropical_critical}, \S\ref{sec:supercr} and \S\ref{sec:appendix_some_proofs}.

\subsection{Types of mixed cells and the tropical discriminant function}\label{subsub:cell-types} 
A cell of dimension $k$ is called a \emph{$k$-cell}. Recall that, for any cell $\xi$ in $\Xi$, there is a unique choice $(\xi_1,\xi_2)\in \Xi_1\times\Xi_2$, such that $\xi = \xi_1\cap\xi_2$. Let $\xi$ be a $2$-cell. Then each of $\xi_1$, and $\xi_2$ are $2$-cells as well, and thus the couple of their dual polytopes $(\delta(\xi_1),\delta(\xi_2))\in\tau_1\times\tau_2$ is a couple of points in $A_1\times A_2$.
\begin{definition}[Relevant, diagonal, and lateral cells]\label{def:lat-diag}
A $2$-cell $\xi\in\Xi$ is called \emph{relevant} if the set\\ $\{(0,0),~\delta(\xi_1),~\delta(\xi_2)\}$ belongs to a line $L$, and is called \emph{irrelevant} otherwise. We say that $\xi$ is \emph{lateral}  if it is relevant, and $L$ separates $\R^2$ into two components, where the closure of one of them contains $A_1\cup A_2$. A cell in $\Xi$ is called \emph{diagonal} if it is not lateral.
\end{definition}

Recall that any $\xi\in\Xi$ is the relative interior of a polytope in $\R^2$. In what follows, we use $\overline{\xi}$ to denote the Euclidean closure of $\xi$ in $\R^2$. Two adjacent cells $\xi,\xi'\in\Xi$ are said to be \emph{directly adjacent} if $\dim \overline{\xi}\cap \overline{\xi'} =1$.

\begin{remark}\label{rem:diagonal-lateral}
Unlike lateral cells in $\Xi$, diagonal cells can be relevant or irrelevant. Moreover, if a $1$-cell $\xi$ is directly adjacent to two lateral cells, then $\xi$ is a half-line.
\end{remark}

\begin{example}
To distinguish between relevant and irrelevant cells of the maps $f$ in Equation~\eqref{eq:main_example}, it is enough to consider the overlapping equations of the $2$-cells outside the union of two curves $T_1\cup T_2$ (see Figure~\ref{fig:two_tropica_curves}). The relevant and irrelevant cells are represented in Figure~\ref{fig:critical_set-disriminant} on the left.  
In this example, the only diagonal relevant cell is bounded and is yellow. All other yellow cells are lateral.
\end{example}

\begin{definition}[Cells essential to others]\label{def:essential}
Let $\xi$ be a lateral cell, let $\gamma$ be a $1$-cell directly adjacent to it, let $v_{\gamma}\in\Z^2$ be any primitive integer vector directing the segment $\delta(\gamma)$ and let $v_{\xi}\in\Z^2$ be a primitive integer with the same direction as $\delta(\xi)$. We say that $\gamma$ is \emph{essential to $\xi$} if $|\det (v_\gamma,v_{\xi})|\geq 2$.
\end{definition}

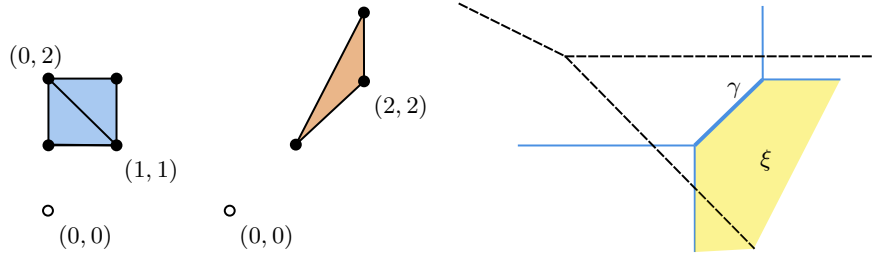
\begin{figure}

\tikzset{every picture/.style={line width=0.75pt}} 

\begin{tikzpicture}[x=1.25pt,y=1.25pt,yscale=-1,xscale=1]

\draw  [color={rgb, 255:red, 0; green, 0; blue, 0 }  ,draw opacity=0 ][fill={rgb, 255:red, 248; green, 231; blue, 28 }  ,fill opacity=0.48 ] (311.65,110.55) -- (335.16,110.55) -- (308.91,161.84) -- (291.31,143.84) -- (291.05,130.55) -- cycle ;
\draw  [color={rgb, 255:red, 0; green, 0; blue, 0 }  ,draw opacity=1 ][fill={rgb, 255:red, 74; green, 144; blue, 226 }  ,fill opacity=0.48 ][line width=0.75]  (116.1,130.5) -- (116.1,110.29) -- (95.5,110.29) -- (95.5,130.29) -- cycle ;
\draw  [fill={rgb, 255:red, 0; green, 0; blue, 0 }  ,fill opacity=1 ] (94,130.5) .. controls (94,129.67) and (94.67,129) .. (95.5,129) .. controls (96.33,129) and (97,129.67) .. (97,130.5) .. controls (97,131.33) and (96.33,132) .. (95.5,132) .. controls (94.67,132) and (94,131.33) .. (94,130.5) -- cycle ;
\draw  [fill={rgb, 255:red, 0; green, 0; blue, 0 }  ,fill opacity=1 ] (114.6,130.5) .. controls (114.6,129.67) and (115.27,129) .. (116.1,129) .. controls (116.93,129) and (117.6,129.67) .. (117.6,130.5) .. controls (117.6,131.33) and (116.93,132) .. (116.1,132) .. controls (115.27,132) and (114.6,131.33) .. (114.6,130.5) -- cycle ;
\draw  [fill={rgb, 255:red, 0; green, 0; blue, 0 }  ,fill opacity=1 ] (94,110.29) .. controls (94,109.46) and (94.67,108.79) .. (95.5,108.79) .. controls (96.33,108.79) and (97,109.46) .. (97,110.29) .. controls (97,111.12) and (96.33,111.79) .. (95.5,111.79) .. controls (94.67,111.79) and (94,111.12) .. (94,110.29) -- cycle ;
\draw  [fill={rgb, 255:red, 0; green, 0; blue, 0 }  ,fill opacity=1 ] (114.6,110.29) .. controls (114.6,109.46) and (115.27,108.79) .. (116.1,108.79) .. controls (116.93,108.79) and (117.6,109.46) .. (117.6,110.29) .. controls (117.6,111.12) and (116.93,111.79) .. (116.1,111.79) .. controls (115.27,111.79) and (114.6,111.12) .. (114.6,110.29) -- cycle ;
\draw    (95.5,130.5) -- (116.1,130.5) ;
\draw    (95.5,110.29) -- (116.1,130.5) ;
\draw  [color={rgb, 255:red, 0; green, 0; blue, 0 }  ,draw opacity=1 ][fill={rgb, 255:red, 212; green, 103; blue, 9 }  ,fill opacity=0.48 ][line width=0.75]  (191,111.12) -- (190.96,90.29) -- (170.36,130.29) -- cycle ;
\draw  [fill={rgb, 255:red, 0; green, 0; blue, 0 }  ,fill opacity=1 ] (189.5,111.12) .. controls (189.5,110.29) and (190.17,109.62) .. (191,109.62) .. controls (191.83,109.62) and (192.5,110.29) .. (192.5,111.12) .. controls (192.5,111.95) and (191.83,112.62) .. (191,112.62) .. controls (190.17,112.62) and (189.5,111.95) .. (189.5,111.12) -- cycle ;
\draw  [fill={rgb, 255:red, 0; green, 0; blue, 0 }  ,fill opacity=1 ] (189.46,90.29) .. controls (189.46,89.46) and (190.13,88.79) .. (190.96,88.79) .. controls (191.79,88.79) and (192.46,89.46) .. (192.46,90.29) .. controls (192.46,91.12) and (191.79,91.79) .. (190.96,91.79) .. controls (190.13,91.79) and (189.46,91.12) .. (189.46,90.29) -- cycle ;
\draw  [fill={rgb, 255:red, 0; green, 0; blue, 0 }  ,fill opacity=1 ] (168.86,130.29) .. controls (168.86,129.46) and (169.53,128.79) .. (170.36,128.79) .. controls (171.19,128.79) and (171.86,129.46) .. (171.86,130.29) .. controls (171.86,131.12) and (171.19,131.79) .. (170.36,131.79) .. controls (169.53,131.79) and (168.86,131.12) .. (168.86,130.29) -- cycle ;
\draw [color={rgb, 255:red, 74; green, 144; blue, 226 }  ,draw opacity=1 ]   (311.6,88.29) -- (311.65,110.55) ;
\draw [color={rgb, 255:red, 74; green, 144; blue, 226 }  ,draw opacity=1 ]   (311.65,110.55) -- (335.16,110.55) ;
\draw [color={rgb, 255:red, 74; green, 144; blue, 226 }  ,draw opacity=1 ][line width=1.5]    (291.05,130.55) -- (311.65,110.55) ;
\draw [color={rgb, 255:red, 74; green, 144; blue, 226 }  ,draw opacity=1 ]   (291.05,130.55) -- (291.1,162.82) ;
\draw [color={rgb, 255:red, 74; green, 144; blue, 226 }  ,draw opacity=1 ]   (237.54,130.55) -- (291.05,130.55) ;
\draw [color={rgb, 255:red, 0; green, 0; blue, 0 }  ,draw opacity=1 ] [dash pattern={on 3.75pt off 0.75pt}]  (251.65,103.55) -- (345.16,103.55) ;
\draw [color={rgb, 255:red, 0; green, 0; blue, 0 }  ,draw opacity=1 ] [dash pattern={on 3.75pt off 0.75pt}]  (219.64,87.52) -- (251.9,103.55) ;
\draw  [fill={rgb, 255:red, 255; green, 255; blue, 255 }  ,fill opacity=1 ] (148.86,150.29) .. controls (148.86,149.46) and (149.53,148.79) .. (150.36,148.79) .. controls (151.19,148.79) and (151.86,149.46) .. (151.86,150.29) .. controls (151.86,151.12) and (151.19,151.79) .. (150.36,151.79) .. controls (149.53,151.79) and (148.86,151.12) .. (148.86,150.29) -- cycle ;
\draw  [fill={rgb, 255:red, 255; green, 255; blue, 255 }  ,fill opacity=1 ] (93.86,150.29) .. controls (93.86,149.46) and (94.53,148.79) .. (95.36,148.79) .. controls (96.19,148.79) and (96.86,149.46) .. (96.86,150.29) .. controls (96.86,151.12) and (96.19,151.79) .. (95.36,151.79) .. controls (94.53,151.79) and (93.86,151.12) .. (93.86,150.29) -- cycle ;
\draw [color={rgb, 255:red, 0; green, 0; blue, 0 }  ,draw opacity=1 ] [dash pattern={on 3.75pt off 0.75pt}]  (251.9,103.55) -- (309.16,161.84) ;
\draw  [color={rgb, 255:red, 0; green, 0; blue, 0 }  ,draw opacity=0 ][fill={rgb, 255:red, 248; green, 231; blue, 28 }  ,fill opacity=0.48 ] (308.91,161.84) -- (291.1,162.82) -- (291.31,143.84) -- cycle ;


\draw (193,114.52) node [anchor=north west][inner sep=0.75pt]  [font=\small]  {$(2,2)$};
\draw (117.43,133.9) node [anchor=north west][inner sep=0.75pt]  [font=\small]  {$(1,1)$};
\draw (82.54,99.18) node [anchor=north west][inner sep=0.75pt]  [font=\small]  {$(0,2)$};
\draw (97.5,153.69) node [anchor=north west][inner sep=0.75pt]  [font=\small]  {$(0,0)$};
\draw (152.36,153.69) node [anchor=north west][inner sep=0.75pt]  [font=\small]  {$(0,0)$};
\draw (310,130.79) node [anchor=north west][inner sep=0.75pt]  [font=\small]  {$\xi$};
\draw (300,110.52) node [anchor=north west][inner sep=0.75pt]  [font=\small]  {$\gamma$};

\end{tikzpicture}
\caption{Pair of subdivided Newton polytopes of a polynomial map, together with the corresponding tropical curves. The cell $\xi$ is lateral and its dual polytopes satisfy $\delta_1(\xi) = (1,1)$, and $\delta_2(\xi)=(2,2)$. 
The $1$-cell $\gamma$ is essential to $\xi$.}\label{fig:essentiality_132}
\end{figure}

Now, we are ready to give a more explicit definition for the function $\Phi_F$ described in the introduction.

\begin{notation}\label{not:arrows-numbers}
Given $p\in\R^2$, the set $\Raw(p)\subset\R^2$ denotes the half line, emanating from $p$, and whose direction is $(-1,0)$, and $\Ras(p)$ will denote the half line, emanating from $p$, and whose directions is $(0,-1)$, respectively. We define $\Rasw(p):=\Ras(p)\cup\Raw(p)$.
\end{notation}

\begin{definition}[Tropical discriminant function]\label{def:main} Given a tropical polynomial map $F$ as above such that all vertices of $T_1$ and of $T_2$ are trivalent, and the subdivision $\Xi$ of $\R^2$, induced by $F$ is stable, we define
\[
\Phi_F:~\Xi\to\pow(\R^2),
\] sending a cell $\xi\in\Xi$ onto a subset in $\R^2$ according to the following situations (see Figure~\ref{fig:situations-table}):
\begin{enumerate}[topsep=0mm,itemsep=0mm,label=\textbf{\arabic*.},ref=\arabic*]
\setcounter{enumi}{-1}
	\item\label{it:0}  Assume that $\dim \xi = 0$.
	
	\begin{enumerate}[topsep=0mm,itemsep=0mm,label=\textbf{0.\arabic*.},ref=\textbf{0.\arabic*}]

		\item\label{it:01} $\xi=E_1\cap E_2$, where $E_i$ is an edge of $T_i$ ($i=1,2$).\\
		
			\begin{enumerate}[topsep=0mm,itemsep=0mm,label=\textbf{0.1.\arabic*.},ref=\textbf{0.1.\arabic*}]
			
				\item\label{it:011} 
				If there are two relevant cells positioned on 
				the same side of $E_1$ (resp. $E_2$) (e.g. $\color{royalblue}\bm e.\text{\textbf{ii}}$ or $\color{royalblue}\bm e.\text{\textbf{iii}}$), we set
				\[
				\Phi_F(\xi) := \Raw\big(F(\xi) \big) \text{ (resp. } \Phi_F(\xi) := \Ras\big(F(\xi) \big)\text{)}
				\]				
				\item\label{it:012} Otherwise (e.g. $\color{royalblue}\bm a.\text{\textbf{i}}$, $\color{royalblue}\bm b.\text{\textbf{i}}$ or $\color{royalblue}\bm c.\text{\textbf{i}}$), we set				
				\[
				\Phi_F(\xi) := \Rasw\big(F(\xi) \big).
				\]				
			\end{enumerate}					
		
		\item\label{it:02} $\xi$ is a vertex of $T_1$ (resp. of $T_2$). \\
		
		 \begin{enumerate}[topsep=0mm,itemsep=0mm,label=\textbf{0.2.\arabic*.},ref=\textbf{0.2.\arabic*}]
			
				\item\label{it:021} If $\xi$ is adjacent to at most one relevant cell (e.g. $\color{royalblue}\bm c.\text{\textbf{ii}}$, $\color{royalblue}\bm d.\text{\textbf{ii}}$, $\color{royalblue}\bm c.\text{\textbf{iii}}$ or $\color{royalblue}\bm d.\text{\textbf{iii}}$), we set
				\[
				\Phi_F(\xi) := \Raw\big(F(\xi) \big) \text{ (resp. } \Phi_F(\xi) := \Ras\big(F(\xi) \big)\text{)}
				\]						
				\item\label{it:022} Otherwise (e.g. $\color{royalblue}\bm d.\text{\textbf{iv}}$ or $\color{royalblue}\bm e.\text{\textbf{iv}}$), we set
				\[
				\Phi_F(\xi) = \emptyset
				\]								
			\end{enumerate}
	
	\end{enumerate}

\item\label{it:1} Assume that $\dim\xi = 1$. Then, it is directly adjacent to two $2$-cells $\sigma,\sigma'\in
\Xi$, and $\xi$ is contained in an edge $E_1$ of $T_1$ (resp. $E_2$ of $T_2$).\\ 	
	\begin{enumerate}[topsep=0mm,itemsep=0mm,label=\textbf{1.\arabic*.},ref=\textbf{1.\arabic*}]

		\item\label{it:11} If both $\sigma$ and $\sigma'$ are \emph{irr}elevant (e.g. $\color{royalblue}\bm d.\text{\textbf{v}}$ or $\color{royalblue}\bm e.\text{\textbf{v}}$), we set  		
				\[
				\Phi_F(\xi) := F(\xi)
				\]				
		\item\label{it:12} If $\sigma$ is relevant, and $\sigma'$ is irrelevant (e.g. $\color{royalblue}\bm b.\text{\textbf{iv}}$ $\color{royalblue}\bm c.\text{\textbf{iv}}$), we set  
				\[
				\Phi_F(\xi) := \emptyset
				\]				
		\item\label{it:13} Assume that both $\sigma$ and $\sigma'$ are relevant. \\
		
			\begin{enumerate}[topsep=0mm,itemsep=0mm,label=\textbf{1.3.\arabic*.},ref=\textbf{1.3.\arabic*}]
			
				\item\label{it:131} If, both $\sigma$ and $\sigma'$ are diagonal (e.g. $\color{royalblue}\bm a.\text{\textbf{ii}}$, $\color{royalblue}\bm b.\text{\textbf{ii}}$, $\color{royalblue}\bm a.\text{\textbf{iii}}$ or $\color{royalblue}\bm b.\text{\textbf{iii}}$), we set
				\[
				\Phi_F(\xi) := \Raw\big(F(\xi) \big) \text{ (resp. } \Phi_F(\xi) := \Ras\big(F(\xi) \big)\text{)}
				\]						
				\item\label{it:132}  Assume that both $\sigma$ and $\sigma'$ are lateral, let $\zeta$ denote  
				the endpoint of $\xi$, and let $\gamma,\gamma'\in\Xi$ be two $1$-cells adjacent to $\sigma$ and $\sigma'$, respectively, and having $\zeta$ as an endpoint (e.g.  $\color{royalblue}\bm{ a}. \text{\textbf{vi}}$, $\color{royalblue}\bm{ a}. \text{\textbf{vii}}$, $\color{royalblue}\bm{b}. \text{\textbf{vi}}$ or $\color{royalblue}\bm{b}.\text{\textbf{vii}}$).\\ 
					\begin{enumerate}[topsep=0mm,itemsep=0mm,label=\textbf{1.3.2.\arabic*.},ref=\textbf{1.3.2.\arabic*}]
							
						\item\label{it:1321} If $\gamma$ and $\gamma'$ are essential to $\sigma$ and $\sigma'$, respectively, we set 
							\[
							\Phi_F(\xi) := \Raw\big(F(\xi) \big) \text{ (resp. } \Phi_F(\xi) := \Ras\big(F(\xi) \big)\text{)}
							\]														
						\item\label{it:1322} If $\gamma$ or $\gamma'$ is not essential to $\sigma$, and to $\sigma'$ respectively, we set
							\[
								\Phi_F(\xi) := \Phi_F(\zeta)
							\]
				\end{enumerate}			
			\end{enumerate}					
		\end{enumerate}
	
	\item\label{it:2} Assume that $\dim\xi = 2$.\\ 

	\begin{enumerate}[topsep=0mm,itemsep=0mm,label=\textbf{2.\arabic*.},ref=\textbf{2.\arabic*}]
		
		\item\label{it:21} If $\xi$ is irrelevant (e.g. $\color{royalblue}\bm a.\text{\textbf{iv}}$), we set		
		\[
		\Phi_F(\xi) := \emptyset.
		\]		
		\item\label{it:22} Assume that $\xi$ is relevant.\\ 		
		
			\begin{enumerate}[topsep=0mm,itemsep=0mm,label=\textbf{2.2.\arabic*.},ref=\textbf{2.2.\arabic*}]
			
				\item\label{it:221} If $\xi$ is diagonal (e.g. $\color{royalblue}\bm a.\text{\textbf{v}}$, $\color{royalblue}\bm b.\text{\textbf{v}}$ or $\color{royalblue}\bm c.\text{\textbf{v}}$), we set		
				\[
				\Phi_F(\xi) := F(\xi).
				\]				
				\item\label{it:222} If  $\xi$ is lateral, we set 
				\[
				\Phi_F(\xi) := \bigcup_{\gamma\in\Xi} F(\gamma),
				\] where $\gamma$ runs over all $1$-cells in $\Xi$ that are essential to $\xi$ (c.f. Figure~\ref{fig:essentiality_132}).						
			\end{enumerate}						
	\end{enumerate}	
\end{enumerate}
\end{definition}

\begin{example}\label{ex:for_def_main}
Each situation in  Definition~\ref{def:main} has a representative in some Figure in this paper. Below are some examples that we mention for each case. 
 
\begin{itemize}

	\item[\ref{it:011}] Figure~\ref{fig:critical_set-disriminant}: The left-most intersection point $\mu$ of $\VT(f_1)\cap \VT(f_2)$.
	
	\item[\ref{it:012}] Figure~\ref{fig:critical_set-disriminant}: Intersection points $\beta$, $\gamma$, $\delta$ and $\mu$.
	
	\item[\ref{it:021}] Figure~\ref{fig:critical_set-disriminant}: Vertexes $\lambda$, $\alpha$,  $\kappa$ and the one to the left of $\delta$.

	\item[\ref{it:022}] Figure~\ref{fig:critical_set-disriminant}: Every lateral cell is adjacent to at least one vertex of type~\ref{it:022}.
	
	\item[\ref{it:11}] Figure~\ref{fig:critical_set-disriminant}:  Edges joining $\beta$ to $\lambda$ and $\beta$ to $\alpha$.
	
	\item[\ref{it:12}] Figure~\ref{fig:critical_set-disriminant}: Edges joining $\gamma$ to $\nu$ and $\beta$ to $\gamma$.
	
	\item[\ref{it:131}] Figure~\ref{fig:it131}: The $1$-cell $\tilde{\xi}$ is adjacent to two relevant diagonal $2$-cells.
	
	\item[\ref{it:1321}] Figure~\ref{fig:it131}: The $1$-cell $\xi$ is adjacent to two relevant lateral $2$-cells $\sigma$ and $\sigma'$.
	
	\item[\ref{it:1322}] Figure~\ref{fig:subdivisions}: Any vertical or horizontal ray is an example for \ref{it:1322}.
	
	\item[\ref{it:21}] Figure~\ref{fig:critical_set-disriminant}: Irrelevant cells are represented in white.
	
	\item[\ref{it:221}] Figure~\ref{fig:critical_set-disriminant}: Diagonal relevant cell is the only yellow bounded one.
	
	\item[\ref{it:222}] Figure~\ref{fig:essentiality_132}: $\gamma$ is the only essential cell to $\xi$.

\end{itemize} 
\end{example}

\begin{figure}

\tikzset{every picture/.style={line width=0.75pt}} 

\begin{tikzpicture}[x=0.75pt,y=0.75pt,yscale=-1,xscale=1]

\draw  [color={rgb, 255:red, 248; green, 231; blue, 28 }  ,draw opacity=0 ][fill={rgb, 255:red, 248; green, 231; blue, 28 }  ,fill opacity=0.3 ] (240.14,95.28) -- (274.42,104.86) -- (249.67,129.61) -- (240.14,120.03) -- cycle ;
\draw  [color={rgb, 255:red, 248; green, 231; blue, 28 }  ,draw opacity=0 ][fill={rgb, 255:red, 248; green, 231; blue, 28 }  ,fill opacity=0.3 ] (240.14,120.03) -- (249.67,129.61) -- (224.92,154.36) -- (215.39,120.03) -- cycle ;
\draw [color={rgb, 255:red, 0; green, 0; blue, 0 }  ,draw opacity=1 ][fill={rgb, 255:red, 74; green, 144; blue, 226 }  ,fill opacity=1 ][line width=0.75]    (173.76,154.5) -- (149.01,129.75) ;
\draw  [color={rgb, 255:red, 0; green, 0; blue, 0 }  ,draw opacity=1 ][fill={rgb, 255:red, 0; green, 0; blue, 0 }  ,fill opacity=1 ] (171.77,154.5) .. controls (171.77,153.4) and (172.66,152.51) .. (173.76,152.51) .. controls (174.86,152.51) and (175.75,153.4) .. (175.75,154.5) .. controls (175.75,155.6) and (174.86,156.49) .. (173.76,156.49) .. controls (172.66,156.49) and (171.77,155.6) .. (171.77,154.5) -- cycle ;
\draw  [color={rgb, 255:red, 0; green, 0; blue, 0 }  ,draw opacity=1 ][fill={rgb, 255:red, 255; green, 255; blue, 255 }  ,fill opacity=1 ] (122.27,154.5) .. controls (122.27,153.4) and (123.16,152.51) .. (124.26,152.51) .. controls (125.36,152.51) and (126.25,153.4) .. (126.25,154.5) .. controls (126.25,155.6) and (125.36,156.49) .. (124.26,156.49) .. controls (123.16,156.49) and (122.27,155.6) .. (122.27,154.5) -- cycle ;
\draw [color={rgb, 255:red, 74; green, 144; blue, 226 }  ,draw opacity=1 ][fill={rgb, 255:red, 74; green, 144; blue, 226 }  ,fill opacity=1 ][line width=0.75]    (73.98,129.75) -- (98.73,105) ;
\draw  [color={rgb, 255:red, 74; green, 144; blue, 226 }  ,draw opacity=1 ][fill={rgb, 255:red, 74; green, 144; blue, 226 }  ,fill opacity=1 ] (96.74,105) .. controls (96.74,103.9) and (97.63,103.01) .. (98.73,103.01) .. controls (99.82,103.01) and (100.71,103.9) .. (100.71,105) .. controls (100.71,106.1) and (99.82,106.99) .. (98.73,106.99) .. controls (97.63,106.99) and (96.74,106.1) .. (96.74,105) -- cycle ;
\draw  [color={rgb, 255:red, 74; green, 144; blue, 226 }  ,draw opacity=1 ][fill={rgb, 255:red, 255; green, 255; blue, 255 }  ,fill opacity=1 ] (47.24,154.5) .. controls (47.24,153.4) and (48.13,152.51) .. (49.23,152.51) .. controls (50.32,152.51) and (51.21,153.4) .. (51.21,154.5) .. controls (51.21,155.6) and (50.32,156.49) .. (49.23,156.49) .. controls (48.13,156.49) and (47.24,155.6) .. (47.24,154.5) -- cycle ;
\draw  [color={rgb, 255:red, 74; green, 144; blue, 226 }  ,draw opacity=1 ][fill={rgb, 255:red, 74; green, 144; blue, 226 }  ,fill opacity=1 ] (71.99,105) .. controls (71.99,103.9) and (72.88,103.01) .. (73.98,103.01) .. controls (75.07,103.01) and (75.96,103.9) .. (75.96,105) .. controls (75.96,106.1) and (75.07,106.99) .. (73.98,106.99) .. controls (72.88,106.99) and (71.99,106.1) .. (71.99,105) -- cycle ;
\draw  [color={rgb, 255:red, 74; green, 144; blue, 226 }  ,draw opacity=1 ][fill={rgb, 255:red, 74; green, 144; blue, 226 }  ,fill opacity=1 ] (71.99,129.75) .. controls (71.99,128.65) and (72.88,127.76) .. (73.98,127.76) .. controls (75.07,127.76) and (75.96,128.65) .. (75.96,129.75) .. controls (75.96,130.85) and (75.07,131.74) .. (73.98,131.74) .. controls (72.88,131.74) and (71.99,130.85) .. (71.99,129.75) -- cycle ;
\draw [color={rgb, 255:red, 74; green, 144; blue, 226 }  ,draw opacity=1 ][fill={rgb, 255:red, 74; green, 144; blue, 226 }  ,fill opacity=1 ][line width=0.75]    (73.98,129.75) -- (73.98,105) ;
\draw [color={rgb, 255:red, 74; green, 144; blue, 226 }  ,draw opacity=1 ][fill={rgb, 255:red, 74; green, 144; blue, 226 }  ,fill opacity=1 ][line width=0.75]    (98.73,105) -- (73.98,105) ;
\draw  [color={rgb, 255:red, 0; green, 0; blue, 0 }  ,draw opacity=1 ][fill={rgb, 255:red, 0; green, 0; blue, 0 }  ,fill opacity=1 ] (147.02,129.75) .. controls (147.02,128.65) and (147.91,127.76) .. (149.01,127.76) .. controls (150.11,127.76) and (151,128.65) .. (151,129.75) .. controls (151,130.85) and (150.11,131.74) .. (149.01,131.74) .. controls (147.91,131.74) and (147.02,130.85) .. (147.02,129.75) -- cycle ;
\draw [color={rgb, 255:red, 74; green, 144; blue, 226 }  ,draw opacity=1 ][fill={rgb, 255:red, 74; green, 144; blue, 226 }  ,fill opacity=1 ][line width=0.75]    (240.14,120.03) -- (264.89,144.78) ;
\draw [color={rgb, 255:red, 74; green, 144; blue, 226 }  ,draw opacity=1 ][fill={rgb, 255:red, 74; green, 144; blue, 226 }  ,fill opacity=1 ][line width=0.75]    (215.39,120.03) -- (240.14,120.03) ;
\draw [color={rgb, 255:red, 74; green, 144; blue, 226 }  ,draw opacity=1 ][fill={rgb, 255:red, 74; green, 144; blue, 226 }  ,fill opacity=1 ][line width=0.75]    (240.14,120.03) -- (240.14,95.28) ;
\draw [color={rgb, 255:red, 0; green, 0; blue, 0 }  ,draw opacity=1 ][fill={rgb, 255:red, 74; green, 144; blue, 226 }  ,fill opacity=1 ][line width=0.75]  [dash pattern={on 3pt off 0.75pt}]  (224.92,154.36) -- (274.42,104.86) ;
\draw  [color={rgb, 255:red, 248; green, 231; blue, 28 }  ,draw opacity=0 ][fill={rgb, 255:red, 248; green, 231; blue, 28 }  ,fill opacity=0.3 ] (514.34,118.39) -- (535.67,129.05) -- (535.67,153.8) -- (496.78,118.39) -- cycle ;
\draw  [color={rgb, 255:red, 248; green, 231; blue, 28 }  ,draw opacity=0 ][fill={rgb, 255:red, 248; green, 231; blue, 28 }  ,fill opacity=0.3 ] (535.67,129.05) -- (565.17,129.05) -- (565.17,153.34) -- (535.67,153.8) -- cycle ;
\draw  [color={rgb, 255:red, 0; green, 0; blue, 0 }  ,draw opacity=1 ][fill={rgb, 255:red, 0; green, 0; blue, 0 }  ,fill opacity=1 ] (434.02,154.5) .. controls (434.02,153.4) and (434.91,152.51) .. (436.01,152.51) .. controls (437.11,152.51) and (438,153.4) .. (438,154.5) .. controls (438,155.6) and (437.11,156.49) .. (436.01,156.49) .. controls (434.91,156.49) and (434.02,155.6) .. (434.02,154.5) -- cycle ;
\draw  [color={rgb, 255:red, 0; green, 0; blue, 0 }  ,draw opacity=1 ][fill={rgb, 255:red, 255; green, 255; blue, 255 }  ,fill opacity=1 ] (409.27,154.5) .. controls (409.27,153.4) and (410.16,152.51) .. (411.26,152.51) .. controls (412.36,152.51) and (413.25,153.4) .. (413.25,154.5) .. controls (413.25,155.6) and (412.36,156.49) .. (411.26,156.49) .. controls (410.16,156.49) and (409.27,155.6) .. (409.27,154.5) -- cycle ;
\draw [color={rgb, 255:red, 74; green, 144; blue, 226 }  ,draw opacity=1 ][fill={rgb, 255:red, 74; green, 144; blue, 226 }  ,fill opacity=1 ][line width=0.75]    (360.98,129.75) -- (385.73,105) ;
\draw  [color={rgb, 255:red, 74; green, 144; blue, 226 }  ,draw opacity=1 ][fill={rgb, 255:red, 74; green, 144; blue, 226 }  ,fill opacity=1 ] (383.74,105) .. controls (383.74,103.9) and (384.63,103.01) .. (385.73,103.01) .. controls (386.82,103.01) and (387.71,103.9) .. (387.71,105) .. controls (387.71,106.1) and (386.82,106.99) .. (385.73,106.99) .. controls (384.63,106.99) and (383.74,106.1) .. (383.74,105) -- cycle ;
\draw  [color={rgb, 255:red, 74; green, 144; blue, 226 }  ,draw opacity=1 ][fill={rgb, 255:red, 255; green, 255; blue, 255 }  ,fill opacity=1 ] (334.24,154.5) .. controls (334.24,153.4) and (335.13,152.51) .. (336.23,152.51) .. controls (337.32,152.51) and (338.21,153.4) .. (338.21,154.5) .. controls (338.21,155.6) and (337.32,156.49) .. (336.23,156.49) .. controls (335.13,156.49) and (334.24,155.6) .. (334.24,154.5) -- cycle ;
\draw  [color={rgb, 255:red, 74; green, 144; blue, 226 }  ,draw opacity=1 ][fill={rgb, 255:red, 74; green, 144; blue, 226 }  ,fill opacity=1 ] (358.99,154.5) .. controls (358.99,153.4) and (359.88,152.51) .. (360.98,152.51) .. controls (362.07,152.51) and (362.96,153.4) .. (362.96,154.5) .. controls (362.96,155.6) and (362.07,156.49) .. (360.98,156.49) .. controls (359.88,156.49) and (358.99,155.6) .. (358.99,154.5) -- cycle ;
\draw  [color={rgb, 255:red, 74; green, 144; blue, 226 }  ,draw opacity=1 ][fill={rgb, 255:red, 74; green, 144; blue, 226 }  ,fill opacity=1 ] (383.74,154.5) .. controls (383.74,153.4) and (384.63,152.51) .. (385.73,152.51) .. controls (386.82,152.51) and (387.71,153.4) .. (387.71,154.5) .. controls (387.71,155.6) and (386.82,156.49) .. (385.73,156.49) .. controls (384.63,156.49) and (383.74,155.6) .. (383.74,154.5) -- cycle ;
\draw [color={rgb, 255:red, 74; green, 144; blue, 226 }  ,draw opacity=1 ][fill={rgb, 255:red, 74; green, 144; blue, 226 }  ,fill opacity=1 ][line width=0.75]    (360.98,154.5) -- (360.98,129.75) ;
\draw [color={rgb, 255:red, 74; green, 144; blue, 226 }  ,draw opacity=1 ][fill={rgb, 255:red, 74; green, 144; blue, 226 }  ,fill opacity=1 ][line width=0.75]    (385.73,154.5) -- (360.98,154.5) ;
\draw [color={rgb, 255:red, 74; green, 144; blue, 226 }  ,draw opacity=1 ][fill={rgb, 255:red, 74; green, 144; blue, 226 }  ,fill opacity=1 ][line width=0.75]    (385.73,154.5) -- (385.73,105) ;
\draw  [color={rgb, 255:red, 74; green, 144; blue, 226 }  ,draw opacity=1 ][fill={rgb, 255:red, 74; green, 144; blue, 226 }  ,fill opacity=1 ] (358.99,129.75) .. controls (358.99,128.65) and (359.88,127.76) .. (360.98,127.76) .. controls (362.07,127.76) and (362.96,128.65) .. (362.96,129.75) .. controls (362.96,130.85) and (362.07,131.74) .. (360.98,131.74) .. controls (359.88,131.74) and (358.99,130.85) .. (358.99,129.75) -- cycle ;
\draw [color={rgb, 255:red, 74; green, 144; blue, 226 }  ,draw opacity=1 ][fill={rgb, 255:red, 74; green, 144; blue, 226 }  ,fill opacity=1 ][line width=0.75]    (360.98,154.5) -- (385.73,105) ;
\draw [color={rgb, 255:red, 0; green, 0; blue, 0 }  ,draw opacity=1 ][fill={rgb, 255:red, 74; green, 144; blue, 226 }  ,fill opacity=1 ][line width=0.75]    (460.76,154.5) -- (436.01,154.5) ;
\draw  [color={rgb, 255:red, 0; green, 0; blue, 0 }  ,draw opacity=1 ][fill={rgb, 255:red, 0; green, 0; blue, 0 }  ,fill opacity=1 ] (458.77,154.5) .. controls (458.77,153.4) and (459.66,152.51) .. (460.76,152.51) .. controls (461.86,152.51) and (462.75,153.4) .. (462.75,154.5) .. controls (462.75,155.6) and (461.86,156.49) .. (460.76,156.49) .. controls (459.66,156.49) and (458.77,155.6) .. (458.77,154.5) -- cycle ;
\draw [color={rgb, 255:red, 74; green, 144; blue, 226 }  ,draw opacity=1 ][fill={rgb, 255:red, 74; green, 144; blue, 226 }  ,fill opacity=1 ][line width=0.75]    (496.78,100.83) -- (514.34,118.39) ;
\draw [color={rgb, 255:red, 74; green, 144; blue, 226 }  ,draw opacity=1 ][fill={rgb, 255:red, 74; green, 144; blue, 226 }  ,fill opacity=1 ][line width=0.75]    (496.78,118.39) -- (514.34,118.39) ;
\draw [color={rgb, 255:red, 74; green, 144; blue, 226 }  ,draw opacity=1 ][fill={rgb, 255:red, 74; green, 144; blue, 226 }  ,fill opacity=1 ][line width=1.5]    (535.67,153.8) -- (535.67,129.05) ;
\draw [color={rgb, 255:red, 74; green, 144; blue, 226 }  ,draw opacity=1 ][fill={rgb, 255:red, 74; green, 144; blue, 226 }  ,fill opacity=1 ][line width=0.75]    (535.67,129.05) -- (565.17,129.05) ;
\draw [color={rgb, 255:red, 74; green, 144; blue, 226 }  ,draw opacity=1 ][fill={rgb, 255:red, 74; green, 144; blue, 226 }  ,fill opacity=1 ][line width=0.75]    (514.34,118.39) -- (535.67,129.05) ;
\draw [color={rgb, 255:red, 0; green, 0; blue, 0 }  ,draw opacity=1 ][fill={rgb, 255:red, 74; green, 144; blue, 226 }  ,fill opacity=1 ][line width=0.75]  [dash pattern={on 3pt off 0.75pt}]  (558.23,154.75) -- (558.23,100) ;
\draw  [color={rgb, 255:red, 248; green, 231; blue, 28 }  ,draw opacity=0 ][fill={rgb, 255:red, 74; green, 144; blue, 226 }  ,fill opacity=0.4 ] (385.73,154.5) -- (360.98,154.5) -- (360.98,129.75) -- (385.73,105) -- cycle ;
\draw  [color={rgb, 255:red, 248; green, 231; blue, 28 }  ,draw opacity=0 ][fill={rgb, 255:red, 74; green, 144; blue, 226 }  ,fill opacity=0.4 ] (98.73,105) -- (73.98,129.75) -- (73.98,105) -- cycle ;
\draw [color={rgb, 255:red, 74; green, 144; blue, 226 }  ,draw opacity=1 ][line width=1.5]    (240.14,120.03) -- (249.67,129.61) ;

\draw (246.14,111.63) node [anchor=north west][inner sep=0.75pt]  [font=\tiny]  {$\tilde{\xi} $};
\draw (533.27,157) node [anchor=north west][inner sep=0.75pt]  [font=\tiny]  {$\xi $};
\draw (523.46,134.67) node [anchor=north west][inner sep=0.75pt]  [font=\tiny]  {$\sigma $};
\draw (541.46,137.07) node [anchor=north west][inner sep=0.75pt]  [font=\tiny]  {$\sigma '$};
\draw (525.66,118.47) node [anchor=south west] [inner sep=0.75pt]  [font=\tiny]  {$\gamma $};
\draw (550.22,124.05) node [anchor=south] [inner sep=0.75pt]  [font=\tiny]  {$\gamma '$};

\end{tikzpicture}

\caption{Examples of a 1-cell as in Definition~\ref{def:main}:~\ref{it:131} and~\ref{it:1321} respectively.}\label{fig:it131}
\end{figure}
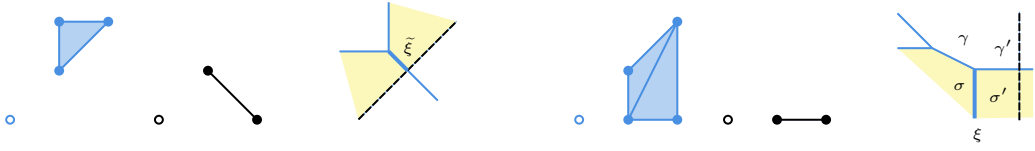

\begin{remark}\label{rem:prop-all-cases}
Lemma~\ref{lem:adjacent_diag_relevant} in \S\ref{sec:tropical_critical} shows that if $\xi\in\Xi$ is directly adjacent to two relevant cells, then $\xi$ is in one of the two situations \textbf{1.3.1} and \textbf{1.3.2} in Definition~\ref{def:main}. Therefore, Definition~\ref{def:main} covers all possible cases.
\end{remark}

\begin{figure}[htb]
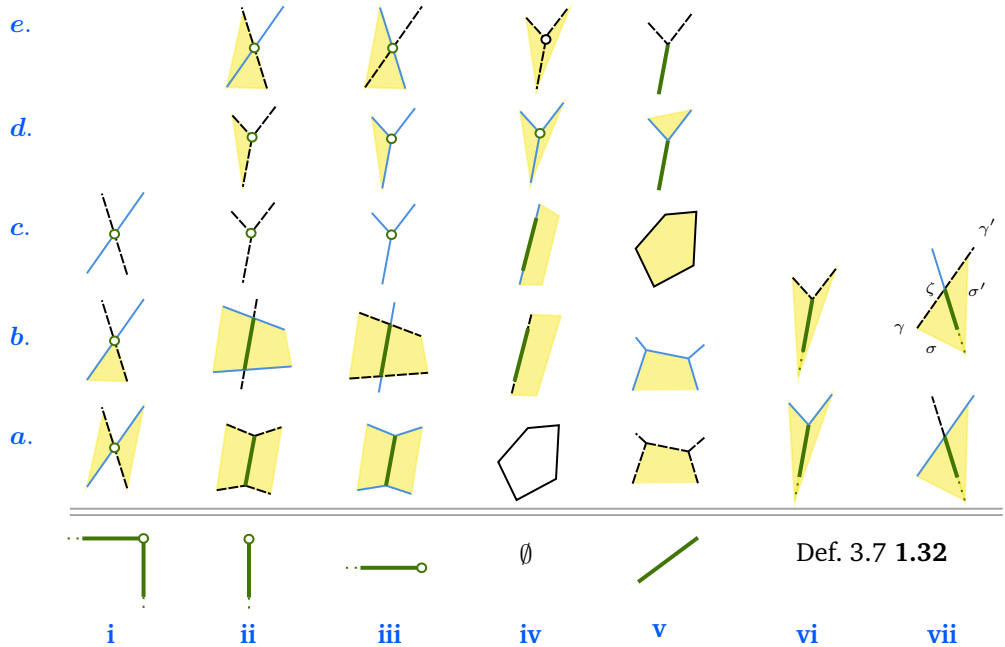


\tikzset{every picture/.style={line width=0.75pt}} 


\caption{Above the double line: representative examples of several possible situations for a cell $\xi\in\Xi$ (in green) to have. The yellow regions represent relevant cells, thin continuous and dashed lines represent sections of $T_1$ and $T_2$, respectively. Below the double line: the image $\Phi_F(\xi)\subset\R^2$ of the cell $\xi$ above it.}\label{fig:situations-table}
\end{figure}

\subsection{Stable subdivisions}\label{sub:valuation-space} We state two useful observations related to stable subdivisions. The following lemma will be useful for the proof of Theorem~\ref{thm:main}.
\begin{lemma}[Stable intersections]\label{lem:transversal}
Let $g_1,g_2\in \K[z_1,z_2]$ be two polynomials whose respective tropical curves $\VT(g_1)$ and $\VT( g_2)$ have a stable intersection point $x\in \R^2$ (c.f. Definition~\ref{def:Transv}). Then, all solutions to $g_1=g_2=0$ in $\TK$ with valuation $x$ are simple.
\end{lemma}

\begin{proof}
\cite[Lemma 5.2]{DiTa12} shows that $(y_1,y_2)\in\R^2$ is the valuation of a degenerate solution to $g_1 = g_2 = 0$ if and only if there exists $y_3\in\R$ so that $(y_1,y_2,y_3)$ is the valuation of a singular root to the polynomial $G:= g_1(z_1,z_2)+ z_3 g_2(z_1,z_2)\in\K[z_1,z_2,z_3]$. Let $\tau$ be the dual subdivision of $\cN(G)$, induced by the tropical hypersurface $\VT(G)\subset\R^3$. Then,~\cite[Lemma 3.1]{DiTa12} shows that for any vertex $\overline{p}\in\VT(G)$, the polyhedron $\delta(\overline{p})\in\tau$ is not a tetrahedron if $\overline{p}$ is the valuation of a singular root of $G$. 

For the stable intersection point $(x_1,x_2)$ of $\VT(g_1)$ and $\VT( g_2)$, one can check that there exists exactly one $x_3$ so that $\overline{x}:=(x_1,x_2,x_3)\in\VT(G)$. Furthermore, the point $\overline{x}$ is a vertex whose dual $\delta(\overline{x})$ is a tetrahedron in $\tau$. Indeed, when evaluated at $x$, exactly two terms in $g_i^{\trop}$ ($i=1,2$) reach their maximum. Therefore, $\overline{x}$ is not the valuation of a singular root of $\VT(G)$ (\cite[Lemma 3.1]{DiTa12}), and thus $x$ is not the valuation of any degenerate solution to $g_1 = g_2 = 0$ (\cite[Lemma 5.2]{DiTa12}).
\end{proof}


\begin{lemma}\label{lem:gen-valued_transversal}
There exists a Zariski open subset $\Omega_1\subset\KW$ for which any $f\in \Omega_1$ induces (through $f^{\trop}$) a stable subdivision $\Xi$ of $\R^2$ (see Definition~\ref{def:Transv}), where the dual subdivisions $\tau_1$ and $\tau_2$ are triangulations of the respective Newton polytopes.

\end{lemma}
\begin{proof}
Let $\bm{T}$ denote the full-dimensional torus in $\K[A]$. If a map $f\in \bm{T}$ satisfies the opposite of the two properties in the lemma, then the valuations of coefficients of $f$ must satisfy a particular linear combination that depends on $A$. This gives rise to a hyperplane in the Euclidean space $\Val(\bm{T})\cong \R^{|A_1|}\times\R^{|A_2|}$. Let $H$ be the set of all hyperplanes in $\Val(\bm{T})$ obtained in this way. Note that, with $A$ being fixed, there are finitely-many combinatorial types of unstable subdivisions $\Xi$ induced by maps $f\in\K[A]$. Hence, $H$ is a finite union of hyperplanes. In fact, as $A_1,A_2\subset\N^2$, the coefficients of the linear combinations giving rise to $K$ are rational numbers. Therefore $H$ is itself a tropical hypersurface in $\Val(\bm{T})$. Finally, Theorem~\ref{th:Kapranov} shows that $\Omega_1=\bm{T}\cap \Val^{-1}(\Val(\bm{T})\setminus H)$ is the complement of an algebraic variety in $\K[A]$.
\end{proof}

\subsection{Tropical critical points}\label{sub:tropical_critical} 
We will use $\cj_f$ and $C(f)$ to denote the determinant $\det(\Jac_{z}f)\in \K[z_1,z_2]$, and its zero locus $\VK(\cj_f)\in\TK$ respectively. The latter is called the set of \emph{critical points of $f$}. 

\begin{lemma}\label{lem:crit-points-invar}
There exists a Zariski open subset $\Omega_2\subset\KW$, in which any two maps $f,g\in \Omega_2$ satisfy
\[
f^{\trop} = g^{\trop}\Rightarrow\Val(C(f)) = \Val(C(g)).
\]
\end{lemma}

\begin{proof}
For any $f\in \bm{T}\subset\KW$, the coefficients of $\cj_f$ are polynomials $\phi\in \Z[c_1,\ldots,c_r]$, in the non-zero coefficients of $f_1$ and $f_2$. One can choose values $\tilde{c}_1,\ldots,\tilde{c}_r\in\K\setminus 0$ so that any two terms $\tilde{c}^\alpha,\tilde{c}^\beta$ appearing in $\phi(\tilde{c})$ satisfies $\val(\tilde{c}^\alpha)\neq \Val(\tilde{c}^\beta)$. We denote by $\Omega_2$ the collection of all such choices in $\bm{T}$. Then, for any $c\in\Omega_2$, the valuation of each polynomial $\phi(c)$ depends only on $\Val(c)$. Consequently, for any $c,d\in\Omega_2$ representing two maps $f$ and $g$, if $\Val(c)=\Val(d)$, then 
\[
\cj_f^{\trop} = \cj_g^{\trop}.
\] Finally, $\Omega_2$ being Zariski open follows from~\cite[Lemma 2.2.12]{MS15}.
\end{proof}

%


Both the below proposition, and the subsequent lemma will be proven in \S\ref{sec:tropical_critical}. 
\begin{proposition}\label{prop:t-crit_mixed-cells}
Let $A$ be a pair of finite subsets of $\N^2\setminus\{(0,0)\}$. There is a Zariski open $\Omega\subset \K[A]$, contained in the intersection $\Omega_1\cap\Omega_2$ (see Lemmas~\ref{lem:gen-valued_transversal} and~\ref{lem:crit-points-invar}), for which every polynomial map $f\in\Omega$ satisfies the following. If $F:=f^{\trop}$, and $\Xi$ is the subdivision of $\R^2$ induced by $F$, then for any $\xi\in\Xi$, it holds

\begin{enumerate}

	\item the set $\Phi_F(\xi)$ is empty if and only if $\xi\cap C(F)$ is empty,
	
	\item if $\xi$ satisfies Definition~\ref{def:main}:~\ref{it:1321}, then $\xi\cap C(F)$ is bounded, 
	
	\item if $\dim \xi =2$, then $y\in \Phi_F(\xi)$ if and only if
			\[
				F^{-1}(y)\cap\xi\cap C(F)\neq\emptyset.
			\]	

\end{enumerate}
\end{proposition}
\begin{remark}\label{rem:3->1}
Concerning the above Proposition, notice that for $2$-cells, it holds (3) $\Rightarrow$ (1). The converse may fail whenever $\xi$ is as in Definition~\ref{def:main}:~\ref{it:222} where $F(\xi)\neq\Phi_F(\xi)$ (see e.g., Figure~\ref{fig:lateral_cell}). 
\end{remark}

\subsection{Images of tropical critical points}\label{subs:images_tropical}  In the rest of this paper, we are interested in the set of polynomial maps $\Omega\subset\K[A]$ introduced in Proposition~\ref{prop:t-crit_mixed-cells}. Let $f$ be any such map, let $F:\R^2\to\R^2$ denote the tropical polynomial map $f^{\trop}$. Thanks to Lemma~\ref{lem:crit-points-invar} we will write $C(F)$ instead of $\Val(C(f))$. 

\begin{lemma}[Image of the critical set]\label{lem:val-F_F-val}
For any point $x\in C(F)$, there exists a point $z\in C(f)$ satisfying $\Val(z) = x$ and
\begin{equation}\label{eq:lem:val-F_F-val}
F(x)= \Val (f(z)).
\end{equation}
\end{lemma}

\begin{definition}[Super critical cells]\label{def:super-crit}
Let $\Xi$ denote the subdivision of $\R^2$ induced by $F$. A cell $\xi\in\Xi$ is called \emph{super critical for $T_1$} (resp. for $T_2$),  if $F(\xi)$ is a point, and
\[
\Raw(F(\xi)) \subset\Phi_F(\xi)\text{ (resp.} \Ras(F(\xi)) \subset\Phi_F(\xi)\text{)}. 
\] That is, in the notations of Definition~\ref{def:main}, the cell $\xi$ satisfies one of the conditions~\ref{it:01}, ~\ref{it:021}, ~\ref{it:131}, ~\ref{it:1321} or~\ref{it:1322} where in the latter case, the $0$-cell $\zeta$ and $T_1$ satisfy~\ref{it:011}.
\end{definition} 

\begin{notation}
For any value $c\in\K\setminus 0$, we use $T_i(\val(c))$ ($i=1,2$), to denote the tropical curve in $\R^2$ defined as $\Val \VT(f_i-c)$. For any pair $w:=(w_1,w_2)\in(\K\setminus 0)^2$, the notation $T(\Val(w))$ refers to the intersection $T_1(\val(w_1))\cap T_2(\val(w_2))$. Note that, since $f_i(0,0) = 0$ ($i=1,2$), the tropical curve $T_i(\val(w_i))$ depends only on $F_i$ and $\val(w_i)$. Hence $T_1(y_1)$ and $T_2(y_2)$ are well-defined for any $(y_1,y_2)\in\R^2$.
\end{notation}

A common feature about super-critical cells $\xi\in\Xi$ is that there are infinitely-many $y\in\R^2$ for which $T(y)$ has an unstable intersection at a fixed point in $\overline{\xi}$.

\begin{example}\label{ex:map_super-crit}
Let $f$ be the map $z\mapsto (z_1+z_2,~z_2 +z_1z_2)$. Then, its tropicalization $F$ is the map
\[
(x_1,x_2)\mapsto \big(\max(x_1,~x_2),~\max( x_2,~ x_1+x_2)\big),
\] where the subdivision $\Xi$ (see Figure~\ref{fig:supercrit}) contains the cell $\{(0,0)\}$ that is supercritical to each of $T_1$ and $T_2$. One can check that for any point $y$ in the union of the two half-lines $\{(\lambda,0)~|~\lambda<0\}\cup \{(0,\lambda)~|~\lambda<0\}\subset\R^2$, the intersection $T(y)$ is unstable and contains $(0,0)$.
\end{example}

\begin{figure}

\tikzset{every picture/.style={line width=0.75pt}} 

\begin{tikzpicture}[x=1.2pt,y=1.2pt,yscale=-1,xscale=1]

\draw [color={rgb, 255:red, 0; green, 0; blue, 0 }  ,draw opacity=0.3 ][fill={rgb, 255:red, 74; green, 144; blue, 226 }  ,fill opacity=1 ][line width=0.75]  [dash pattern={on 3pt off 0.75pt}]  (200.36,113.5) -- (200.36,88.5) ;
\draw [color={rgb, 255:red, 74; green, 144; blue, 226 }  ,draw opacity=0.3 ][fill={rgb, 255:red, 74; green, 144; blue, 226 }  ,fill opacity=1 ][line width=0.75]    (244.36,113.06) -- (269.36,88.31) ;
\draw [color={rgb, 255:red, 0; green, 0; blue, 0 }  ,draw opacity=1 ][fill={rgb, 255:red, 74; green, 144; blue, 226 }  ,fill opacity=1 ][line width=0.75]  [dash pattern={on 3pt off 0.75pt}]  (175.88,88.5) -- (200.36,88.5) ;
\draw [color={rgb, 255:red, 0; green, 0; blue, 0 }  ,draw opacity=1 ][fill={rgb, 255:red, 74; green, 144; blue, 226 }  ,fill opacity=1 ][line width=0.75]  [dash pattern={on 3pt off 0.75pt}]  (200.36,88.5) -- (217.84,105.98) ;
\draw [color={rgb, 255:red, 74; green, 144; blue, 226 }  ,draw opacity=1 ][fill={rgb, 255:red, 74; green, 144; blue, 226 }  ,fill opacity=1 ][line width=0.75]    (190.95,97.97) -- (225.36,63.5) ;
\draw [color={rgb, 255:red, 0; green, 0; blue, 0 }  ,draw opacity=1 ][fill={rgb, 255:red, 74; green, 144; blue, 226 }  ,fill opacity=1 ][line width=0.75]  [dash pattern={on 3pt off 0.75pt}]  (200.36,88.5) -- (200.36,63.5) ;
\draw  [color={rgb, 255:red, 0; green, 0; blue, 0 }  ,draw opacity=1 ][fill={rgb, 255:red, 255; green, 255; blue, 255 }  ,fill opacity=1 ] (198.37,88.5) .. controls (198.37,87.4) and (199.26,86.51) .. (200.36,86.51) .. controls (201.46,86.51) and (202.35,87.4) .. (202.35,88.5) .. controls (202.35,89.6) and (201.46,90.49) .. (200.36,90.49) .. controls (199.26,90.49) and (198.37,89.6) .. (198.37,88.5) -- cycle ;
\draw  [color={rgb, 255:red, 248; green, 231; blue, 28 }  ,draw opacity=0 ][fill={rgb, 255:red, 248; green, 231; blue, 28 }  ,fill opacity=0.3 ] (136.03,63.5) -- (135.78,88.75) -- (111.03,113.5) -- cycle ;
\draw [color={rgb, 255:red, 74; green, 144; blue, 226 }  ,draw opacity=1 ][fill={rgb, 255:red, 74; green, 144; blue, 226 }  ,fill opacity=1 ][line width=0.75]    (190.95,113.72) -- (190.95,97.97) ;
\draw [color={rgb, 255:red, 74; green, 144; blue, 226 }  ,draw opacity=1 ][fill={rgb, 255:red, 74; green, 144; blue, 226 }  ,fill opacity=1 ][line width=0.75]    (111.03,113.5) -- (161.03,63.5) ;
\draw [color={rgb, 255:red, 0; green, 0; blue, 0 }  ,draw opacity=1 ][fill={rgb, 255:red, 74; green, 144; blue, 226 }  ,fill opacity=1 ][line width=0.75]  [dash pattern={on 3pt off 0.75pt}]  (136.03,113.5) -- (136.03,63.5) ;
\draw  [color={rgb, 255:red, 0; green, 0; blue, 0 }  ,draw opacity=1 ][fill={rgb, 255:red, 255; green, 255; blue, 255 }  ,fill opacity=1 ] (134.04,88.5) .. controls (134.04,87.4) and (134.93,86.51) .. (136.03,86.51) .. controls (137.12,86.51) and (138.01,87.4) .. (138.01,88.5) .. controls (138.01,89.6) and (137.12,90.49) .. (136.03,90.49) .. controls (134.93,90.49) and (134.04,89.6) .. (134.04,88.5) -- cycle ;
\draw [color={rgb, 255:red, 74; green, 144; blue, 226 }  ,draw opacity=1 ][fill={rgb, 255:red, 74; green, 144; blue, 226 }  ,fill opacity=1 ][line width=0.75]    (175.2,97.97) -- (190.95,97.97) ;
\draw [color={rgb, 255:red, 0; green, 0; blue, 0 }  ,draw opacity=1 ][fill={rgb, 255:red, 74; green, 144; blue, 226 }  ,fill opacity=1 ][line width=0.75]  [dash pattern={on 3pt off 0.75pt}]  (244.88,101.43) -- (269.36,101.43) ;
\draw [color={rgb, 255:red, 0; green, 0; blue, 0 }  ,draw opacity=1 ][fill={rgb, 255:red, 74; green, 144; blue, 226 }  ,fill opacity=1 ][line width=0.75]  [dash pattern={on 3pt off 0.75pt}]  (269.36,101.43) -- (281.56,113.64) ;
\draw [color={rgb, 255:red, 74; green, 144; blue, 226 }  ,draw opacity=1 ][fill={rgb, 255:red, 74; green, 144; blue, 226 }  ,fill opacity=1 ][line width=0.75]    (269.36,88.31) -- (294.36,63.56) ;
\draw [color={rgb, 255:red, 0; green, 0; blue, 0 }  ,draw opacity=1 ][fill={rgb, 255:red, 74; green, 144; blue, 226 }  ,fill opacity=1 ][line width=0.75]  [dash pattern={on 3pt off 0.75pt}]  (269.36,88.56) -- (269.36,63.56) ;
\draw [color={rgb, 255:red, 74; green, 144; blue, 226 }  ,draw opacity=1 ][fill={rgb, 255:red, 74; green, 144; blue, 226 }  ,fill opacity=1 ][line width=0.75]    (269.36,114.31) -- (269.36,88.56) ;
\draw [color={rgb, 255:red, 74; green, 144; blue, 226 }  ,draw opacity=1 ][fill={rgb, 255:red, 74; green, 144; blue, 226 }  ,fill opacity=1 ][line width=0.75]    (243.61,88.31) -- (269.36,88.31) ;
\draw [color={rgb, 255:red, 74; green, 144; blue, 226 }  ,draw opacity=0.3 ][fill={rgb, 255:red, 74; green, 144; blue, 226 }  ,fill opacity=1 ][line width=0.75]    (175.36,113.5) -- (190.95,97.97) ;
\draw [color={rgb, 255:red, 65; green, 117; blue, 5 }  ,draw opacity=1 ]   (367.61,63.3) -- (342.6,88.31) ;
\draw [color={rgb, 255:red, 65; green, 117; blue, 5 }  ,draw opacity=1 ]   (342.6,88.31) -- (342.6,113.32) ;
\draw [color={rgb, 255:red, 65; green, 117; blue, 5 }  ,draw opacity=1 ]   (317.59,88.31) -- (342.6,88.31) ;
\draw  [color={rgb, 255:red, 0; green, 0; blue, 0 }  ,draw opacity=1 ][fill={rgb, 255:red, 255; green, 255; blue, 255 }  ,fill opacity=1 ] (340.61,88.31) .. controls (340.61,87.21) and (341.5,86.32) .. (342.6,86.32) .. controls (343.7,86.32) and (344.59,87.21) .. (344.59,88.31) .. controls (344.59,89.41) and (343.7,90.3) .. (342.6,90.3) .. controls (341.5,90.3) and (340.61,89.41) .. (340.61,88.31) -- cycle ;
\draw  [color={rgb, 255:red, 65; green, 117; blue, 5 }  ,draw opacity=1 ][fill={rgb, 255:red, 65; green, 117; blue, 5 }  ,fill opacity=1 ] (329.12,88.31) .. controls (329.12,87.77) and (329.56,87.33) .. (330.1,87.33) .. controls (330.63,87.33) and (331.07,87.77) .. (331.07,88.31) .. controls (331.07,88.85) and (330.63,89.29) .. (330.1,89.29) .. controls (329.56,89.29) and (329.12,88.85) .. (329.12,88.31) -- cycle ;
\draw  [color={rgb, 255:red, 65; green, 117; blue, 5 }  ,draw opacity=1 ][fill={rgb, 255:red, 65; green, 117; blue, 5 }  ,fill opacity=1 ] (341.62,100.81) .. controls (341.62,100.27) and (342.06,99.84) .. (342.6,99.84) .. controls (343.14,99.84) and (343.58,100.27) .. (343.58,100.81) .. controls (343.58,101.35) and (343.14,101.79) .. (342.6,101.79) .. controls (342.06,101.79) and (341.62,101.35) .. (341.62,100.81) -- cycle ;
\draw [color={rgb, 255:red, 139; green, 87; blue, 42 }  ,draw opacity=1 ][line width=1.5]    (269.36,88.31) -- (269.36,101.43) ;
\draw  [color={rgb, 255:red, 0; green, 0; blue, 0 }  ,draw opacity=1 ][fill={rgb, 255:red, 255; green, 255; blue, 255 }  ,fill opacity=1 ] (267.37,88.31) .. controls (267.37,87.21) and (268.26,86.32) .. (269.36,86.32) .. controls (270.46,86.32) and (271.35,87.21) .. (271.35,88.31) .. controls (271.35,89.41) and (270.46,90.3) .. (269.36,90.3) .. controls (268.26,90.3) and (267.37,89.41) .. (267.37,88.31) -- cycle ;

\draw (327.99,79.82) node [anchor=south west] [inner sep=0.75pt]  [font=\tiny]  {$p$};
\draw (347.56,101.59) node [anchor=south west] [inner sep=0.75pt]  [font=\tiny]  {$q$};
\draw (325.72,134.09) node [anchor=south west] [inner sep=0.75pt]  [font=\tiny]  {$\Val( D( f))$};
\draw (245.25,135.43) node [anchor=south west] [inner sep=0.75pt]  [font=\tiny]  {$T_{1}( q) \cup T_{2}( q)$};
\draw (176.94,135.37) node [anchor=south west] [inner sep=0.75pt]  [font=\tiny]  {$T_{1}( p) \cup T_{2}( p)$};
\draw (136.05,130.7) node [anchor=north] [inner sep=0.75pt]  [font=\tiny]  {$T_{1} \cup T_{2}$};
\draw (140.01,91.9) node [anchor=north west][inner sep=0.75pt]  [font=\tiny]  {$( 0,0)$};

\end{tikzpicture}
\caption{ The tropical curves corresponding to Example~\ref{ex:map_super-crit}}\label{fig:supercrit}
\end{figure}
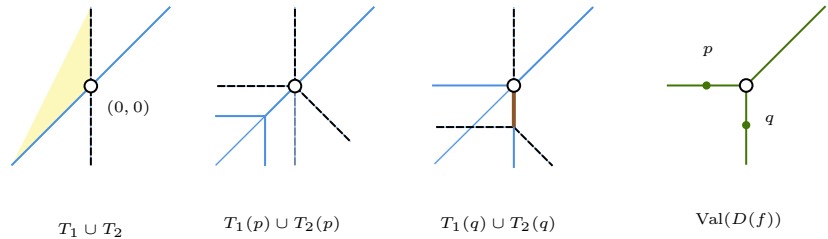

%
 
\begin{proposition}\label{prop:super-critical-mixed-cells}
We retain the notations of Proposition~\ref{prop:t-crit_mixed-cells}. Let $f\in\Omega$, and let $i,j\in\{1,2\}$ be be two distinct indexes. Then, for any $\xi\in\Xi$, the following statements hold true.
\begin{itemize}
	\item[\textbf{a)}] If $\xi$ is super-critical to $T_i$, then for any $(y_1,y_2)\in\R^2$ with $y_i\leq F_i(\xi)$, $y_j = F_j(\xi)$, there exists $(w_1,w_2)\in\TK$ satisfying $\Val(w_1,w_2)=(y_1,y_2)$, and the system 
 \begin{equation}\label{eq:sys:main-half}
  \left\{
  \begin{array}{@{}ccccc@{}}
   f_1 - w_1 & = & 0,\\
   f_2 - w_2 & = & 0,
  \end{array}
  \right.
\end{equation} has a degenerate solution in $\TK$ with valuation in $\overline{\xi}$.

	\item[\textbf{b)}] If $\xi$ is not super-critical with respect to $T_i$, then for any $x\in\xi$ and any $(w_1,w_2)\in\TK$ with $\val(w_i)\neq F_i(x)$, the system~\eqref{eq:sys:main-half} has no degenerate solutions with valuation $x$.
\end{itemize}  
\end{proposition}

\begin{example}\label{ex:super-crit-prop-a}
Given the map $f$ from Example~\ref{ex:map_super-crit}, we compute $\cj_f= -1+z_1 - z_2$. According to Definition~\ref{def:super-crit}, the cell $\{(0,0)\}$ is supercritical for $T_1$. We also have $F(0,0) = ( 0,0)$. As Proposition~\ref{prop:t-crit_mixed-cells} \textbf{a} predicts, for any $\ell>0$, and any $y_1=-\ell$, and $y_2=0$, we choose $w_1 := t^{\ell}$, and $w_2:= (-3+2t^{\ell}+t^{2\ell})/4$, so that $\val(w_1) = y_1<F_1(0,0)=0$, $\val(w_2) = y_2=F_2(0,0)=0$, and the point
\[
\big((t^{\ell}+1)/2,~(t^{\ell}-1)/2\big)
\] has valuation $(0,0)$, is a root to $\cj_f$, and a solution to the system~\eqref{eq:sys:main-half}. Analogously, the cell $\{(0,0)\}$ is supercritical for $T_2$: for any $\ell>0$, and any $y_1=0$, and $y_2=-\ell$, we choose $w_1 := -3 +t^{\ell}$, and $w_2:= t^{\ell}(-2+t^{\ell})$, so that $\val(w_1) = y_1=0$, $\val(w_2) = y_2<0$, and the point $(t^{\ell}-1,~t^{\ell}-2)$ has valuation $(0,0)$, is a root to $\cj_f$, and a solution to the system~\eqref{eq:sys:main-half}.
\end{example}

\begin{remark}\label{rem:prop_super-crit}
Two tropical curves having an unstable intersection, does not always imply that their lifted curves in $\TK$ have a degenerate intersection. Take the system $1+ 2z_1 +z_2 = 1- z_1z_2 = 0$, whose tropical intersection is an unstable isolated point at $(0,0)$ resulted from a vertex intersecting  the relative interior of an edge. This system has two simple solutions $(-1,1),(1/2,-2)\in \TK$ with valuation at $(0,0)$. 

Known results (see e.g.,~\cite[Prop. 5.8]{DiTa12}) assert that once an unstable tropical intersection $T_1\cap T_2$ occurs, there exists two curves $X_1,X_2\subset\TK$, tropicalizing to $T_1$ and $T_2$ respectively, and having a degenerate intersection point (in this example, if we replace the constant term $1$ in $1- z_1z_2$ by the value $1/8$, for instance, the above system will have a unique solution $(-1/4,-1/2)$). In contrast, Proposition~\ref{prop:super-critical-mixed-cells} \textbf{a.}  constitutes a more detailed description relating tropical intersections to classical ones. Namely, we describe some classes of tropical non-degenerate intersections for which any lifting $X_1$ and $X_2$, up to changing two coefficients in the classical polynomials defining $X_1$ and $X_2$, has a degenerate intersection. In the same vain, Proposition~\ref{prop:super-critical-mixed-cells} \textbf{b.} shows that some other unstable tropical intersections have no lifting with a degenerate intersection.
\end{remark}

\section{Proof of Theorem~\ref{thm:main}}\label{sec:proof-main-th} Consider two polynomials $f_1,f_2\in\K[z_1,z_2]$, and let $\Xi$ denote the subdivision of $\R^2$ induced by $f^{\trop}$. Recall that we use $D(f)$ to denote the discriminant of the map $f:\TK\to\TK$, which is defined as the subset $f(C(f))\subset\TK$.

\begin{definition}\label{def:tropical-cell-discriminant}
For any $\xi\in\Xi$, let $\tdf:\Xi\to \pow(\R^2)$ be the function taking $\xi$ to the set of points $\Val(w)\in\R^2$ for which there exists $z\in C(f)$, $w = f(z)$, and $\Val(z)\in\xi$. That is, $(z_1,z_2,w_1,w_2)\in(\K\setminus 0)^4$  is a solution to 
\begin{equation}\label{eq:sys:main}
\left\{
\begin{array}{@{}ccccc@{}}
   f_1(z) - w_1 & = & 0,\\
   f_2(z) - w_2 & = & 0,\\
\cj_f(z) & = & 0.
  \end{array}
  \right.
\end{equation} 
\end{definition} 
The following equality follows from the definitions
\[
\Val(D(f)) = \bigcup_{\xi\in\Xi} \tdf(\xi).
\]
Therefore, Theorem~\ref{thm:main} is a consequence of the following result.

\begin{theorem}\label{th:main-concise}
Let $A$ be a pair of finite subsets in $\N^2\setminus\{(0,0)\}$, consider the subset $\Omega\subset\K[A]$ defined in Proposition~\ref{prop:t-crit_mixed-cells}, let $f$ be a map in $\Omega$, let $F$ denotes the tropical map $f^{\trop}:\R^2\to\R^2$, and let $\Xi$ be the subdivision of $\R^2$ induced by $F$. Then, it holds
%
%
%
\begin{equation}\label{eq:th-main-concise}
\Phi_{F}(\xi) = \tdf(\xi).
\end{equation}
\end{theorem}

\begin{proof}
Let $\xi$ be a cell in $\Xi$. Then, it is in one of the situations of Definition~\ref{def:main}. In what follows, we will show that~\eqref{eq:th-main-concise} holds true for $\xi$ in each of those cases. 

\begin{itemize}

	\item \textbf{Double super-critical cells.} 
First, we consider cells $\xi$ that are super critical to both $T_1$ and to $T_2$. That is $\xi$ satisfies Definition~\ref{def:main}:~\ref{it:012}. First, note that if $\xi$ is an unstable intersection of $T_1(y_1)$ and $T_2(y_2)$, then the point $y:=(y_1,y_2)$ belongs to $\Rasw(F(\xi))$. Therefore, the inclusion $\tdf(\xi)\subset \Rasw(F(\xi))$ follows from Lemma~\ref{lem:transversal}. In what follows, we prove the inclusion 
\begin{equation}\label{eq:inclusion_rays}
\Rasw(F(\xi))\subset \tdf(\xi).
\end{equation} For any $y\in\Raw(F(\xi))$, it holds $y_1\leq F_1(\xi)$, and $y_2=F_2(\xi)$. Thanks to Proposition~\ref{prop:super-critical-mixed-cells}: \textbf{a)}, we find $w\in\TK$ satisfying $\val(w)=y$ and so that the system
 \begin{equation}\label{eq:sys:main-half2}
  \left\{
  \begin{array}{@{}ccccc@{}}
   f_1 - w_1 & = & 0,\\
   f_2 - w_2 & = & 0,
  \end{array}
  \right.
\end{equation} has a degenerate solution $\tilde{z}\in\TK$ with $\Val(\tilde{z})\in\overline{\xi}$. This shows that $\Raw(F(\xi))\subset \tdf(\xi)$. Similarly, we can show that $\Ras(F(\xi))\subset \tdf(\xi)$.\\

\item \textbf{Single super-critical cells.} 
 Let $i\in\{1,2\}$ be such that $\xi\subset T_i$. Without loss of generality, we may assume that $i=1$. We consider a cell $\xi$ that is super critical to $T_1$ but not super critical to $T_2$. That is, it holds $\Phi_F(\xi)=\Raw(F(\xi))$ following Definition~\ref{def:main}. Similarly to the above case, we have $\tdf(\xi)\subset\Rasw(F(\xi))$, and $\Raw(F(\xi))\subset \tdf(\xi)$. However, Proposition~\ref{prop:super-critical-mixed-cells}: \textbf{b)} shows that $\tdf(\xi)$ is contained in a horizontal line passing through $F(\xi)$. Therefore, we get $\Raw(F(\xi)) = \tdf(\xi)$.\\
 
\item \textbf{Non-super critical, lower-dimensional cells.}
We consider $0$-cells and $1$-cells $\xi$ that are not super-critical to $T_1$, nor super-critical to $T_2$.
That is, according to Definition~\ref{def:main}, either $\Phi_F(\xi)=F(\xi)$, or $\Phi_F(\xi)=\emptyset$.
Theorem~\ref{th:main-concise} follows from Proposition~\ref{prop:t-crit_mixed-cells}\textbf{(1)} whenever $\Phi_F(\xi)=\emptyset$. Otherwise, if $\Phi_F(\xi)=F(\xi)$, then Proposition~\ref{prop:super-critical-mixed-cells} \textbf{b.} applied to both $(\xi,T_1)$, and $(\xi,T_2)$ shows that for any $x\in\xi$, the system~\eqref{eq:sys:main} has no solutions with valuation $x$ if $\Val(w)\neq F(x)$. This shows that $\tdf(\xi)\subset F(\xi)$. The second inclusion follows from Lemma~\ref{lem:val-F_F-val}. \\

\item \textbf{Cells of dimension two.} 
Let $\xi$ be a $2$-cell in $\Xi$. Then, for any $x\in\xi$, we have $y=F(x)\Leftrightarrow x\in T_1(y_1)\cap T_2(y_2)$. This implies $\tdf(\xi)\subset F(C(F)\cap\xi)$. Similarly to above, Lemma~\ref{lem:val-F_F-val} implies that the second inclusion holds true as well. This yields
\[
\tdf(\xi)=F(C(F)\cap\xi).
\] Proposition~\ref{prop:t-crit_mixed-cells}\textbf{(3)} concludes Theorem~\ref{th:main-concise} if $\xi$ is a $2$-cell.
\end{itemize}
\end{proof}

\section{Proof of Theorem~\ref{th:Newton-polytope}}\label{sec:Newton-pol} 
Let $A:=(A_1,A_2)$ be a pair of supports in $\N^2\setminus\{(0,0)\}$. In this section, we consider generic polynomial maps $\TT\to\TT$ in $\C[A]$, and we will prove Theorem~\ref{th:Newton-polytope} by providing a method that uses Theorem~\ref{thm:main} to compute the Newton polytope of their discriminant, up to translations, without relying on elimination. We will apply this method for the polynomial map appearing in the below example.

\begin{example}\label{ex:computing_polytope}
Let $g:=(g_1,g_2):\C^2_{u,v}\to\C^2_{a,b}$ be the map defined as 
\[
(u,v)\mapsto (v+v^2 +uv+uv^2+u^2v^2,~2v+3u^2v +4u^2v^2).
\] Its discriminant is given by the polynomial $\cd_g\in\C[a,b]$ below. Its Newton polytope is depicted in Figure~\ref{fig:polytopes_ex}. We obtained $\cd_f$ by computing the elimination ideal~\eqref{eq:disc_elimination} corresponding to $f$ using the software \href{https://www.sagemath.org/}{\texttt{SAGE}}. 
{\scriptsize
\begin{multline*}\label{eq:discriminant_example}
1073741824 a^7 b^2 - 1476395008 a^6 b^3 + 843055104 a^5 b^4 - 255852544 a^4 b^5 + 43515904 a^3 b^6 - 3932160 a^2 b^7 + 147456 a b^8 \\[-3pt]- 2147483648 a^8 + 4093640704 a^7 b - 4794089472 a^6 b^2 + 2401763328 a^5 b^3 - 776110080 a^4 b^4 + 113057792 a^3 b^5 + 1681408 a^2 b^6 \\[-3pt] - 1296896 a b^7 + 18432 b^8 + 3142582272 a^7
 - 3726114816 a^6 b + 2258173952 a^5 b^2 - 583200768 a^4 b^3 + 120756992 a^3 b^4 + 98266432 a^2 b^5 \\[-3pt] +  1723152 a b^6 - 114496 b^7 - 701095936 a^6 + 408629248 a^5 b - 35657472 a^4 b^2 - 41773504 a^3 b^3 + 240054850 a^2 b^4 + 25297148 a b^5 \\[-3pt] + 288850 b^6  - 241209344 a^5 + 16084992 a^4 b - 219803072 a^3 b^2 + 212243632 a^2 b^3 + 50506468 a b^4 + 1668728 b^5 - 27366400 a^4 - 162048000 a^3 b \\[-3pt] + 61899400 a^2 b^2 + 42999450 a b^3 + 2187200 b^4 - 34512000 a^3 - 8960000 a^2 b + 17283500 a b^2 + 1168375 b^3 - 5475000 a^2 + 2737500 a b + 228125 b^2.
\end{multline*}
} 
\end{example}

\begin{figure}[h]
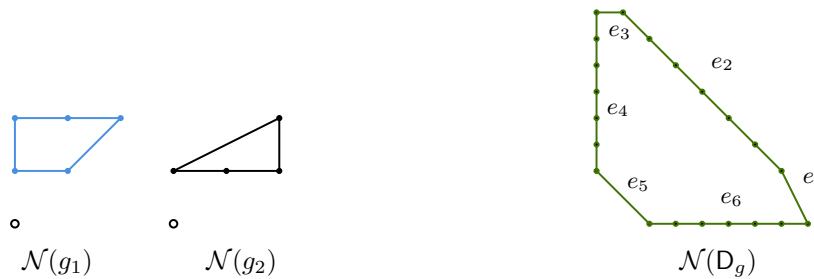


\tikzset{every picture/.style={line width=0.75pt}} 


\caption{The Newton polytopes of $g_1$, $g_2$ and $\cd_g$ from Example~\ref{ex:computing_polytope}}\label{fig:polytopes_ex}
\end{figure}


\begin{lemma}\label{lem:complex-and-Puiseux}
There exists a polytope $\Delta\subset\R^2$, and two Zariski open subsets $U\subset\KW$, and $V\subset\C[A]$ for which any $(f,g)\in U\times V$ satisfies
\[
\cN(\cd_f) = \cN(\cd_g) = \Delta.
\]
\end{lemma}

\begin{proof}
Let $K$ be a field of characteristic zero. Then, any point in $K[A]$ is obtained from two polynomials $h_1\in \Z[z_1,z_2,c_a:a\in A_1]$ and $h_2\in \Z[z_1,z_2,d_a:a\in A_2]$, with particular values for the coefficients $(c_a)_a$ and $(d_a)_a$. Accordingly, Equation~\eqref{eq:disc_elimination} shows that if $h:=(h_1,h_2)$, the polynomial $\cd_h$ is either a constant or is expressed as a finite sum
\[
\sum \phi_\alpha(c,d)w^\alpha,
\] where $\phi_\alpha\in\Z[c_a:a\in A_1,d_a:a\in A_2]$, and $w^{\alpha}:=w_1^{\alpha_1}w_2^{\alpha_2}$. The former case implies that any map $K^2\to K^2$, obtained from $K[A]$, is not dominant, and the latter case shows that, if $P$ is defined as the polynomial $\prod \phi_\alpha$, then for any $f,\tilde{f}\in K[A]\setminus\mathbb{V}_K(P)$, it holds $\phi_\alpha(c,d)=0$ $\Leftrightarrow$ $\phi_\alpha(\tilde{c},\tilde{d})=0$. Therefore, the Newton polytopes of $\cd_f$ and $\cd_{\tilde{f}}$ coincide in $\R^2$. 

To finish the proof, take $U := \KW\setminus\VK(P)$, and $V:= \C[A]\setminus\VC(P)$.
\end{proof}
Let $\mathcal{O}$ denote the Zariski open in $\K[A]$ given as the common intersection of the set $\Omega$ from Proposition~\ref{prop:t-crit_mixed-cells}, and $U$ from Lemma~\ref{lem:complex-and-Puiseux}. Let $g$ be a polynomial map $\TT\to\TT$ in the set $V\in\C[A]$ of Lemma~\ref{lem:complex-and-Puiseux}. In what follows, we compute $\Delta:=\cN(\cd_g)$.

\subsection{Dual fan of the polytope} First, we recover the dual fan of $\Delta$: It is enough to consider a tropical polynomial map $F:=(F_1,F_2):\R^2\to\R^2$, supported on $A$, and inducing a stable subdivision $\Xi$ of $\R^2$. Hence, any perturbation on the tropical coefficients appearing in $F_1$ and $F_2$ will not change the above stability condition (c.f. Definition~\ref{def:Transv}). Since $\cO$ is Zariski open in $\K[A]$, one can find $f\in\cO$ satisfying $f^{\trop}=F$. Lemma~\ref{lem:complex-and-Puiseux} shows that $\cN(\cd_f) = \Delta$, and Proposition~\ref{prop:mixed_subd} (4) determines the dual fan, $\cF(\cd_f)$, of $\cN(\cd_f)$ from the unbounded edges of $\Val(D(f))$. Hence, thanks to Theorem~\ref{th:main-concise}, we obtain $\cF(\Delta)$ from Definition~\ref{def:main} applied to $F$. 
 \begin{example}\label{ex:computing_polytope_cntd}
 Assume that $F:=(F_1,F_2):\R^2\to\R^2$ is the tropical polynomial map
 {\small
 \[
 (x_1,~x_2)\mapsto (\max(x_2,~2x_1+x_2,~2x_1+2x_2),~\max(x_2,~2x_2-5,~x_1+x_2-1,~x_1+2x_2-4,~2x_1+2x_2-4)).
 \] 
 } The subdivision of $\R^2$ induced by $F$ is stable (see Figure~\ref{fig:ex:computing_polytope_cntd}). The set $\Val(D(f))\subset\R^2$ determines $\cF(\Delta)$, both of which are represented in Figure~\ref{fig:ex:computing_polytope_cntd}.
 \end{example}

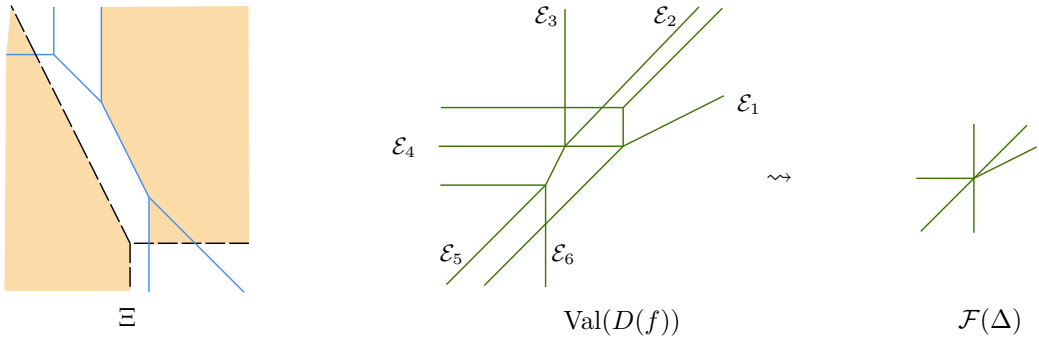
\begin{figure}[h]

\tikzset{every picture/.style={line width=0.55pt}} 

\begin{tikzpicture}[x=.9pt,y=.9pt,yscale=-1,xscale=1]

\draw  [color={rgb, 255:red, 255; green, 255; blue, 255 }  ,draw opacity=0 ][fill={rgb, 255:red, 245; green, 166; blue, 35 }  ,fill opacity=0.4 ] (111.57,151.18) -- (88.47,150.94) -- (69.5,131.8) -- (49.5,91.8) -- (49.5,51.8) -- (111.34,51.38) -- cycle ;
\draw  [color={rgb, 255:red, 255; green, 255; blue, 255 }  ,draw opacity=0 ][fill={rgb, 255:red, 245; green, 166; blue, 35 }  ,fill opacity=0.4 ] (88.47,150.94) -- (69.5,151.8) -- (69.56,132.38) -- cycle ;
\draw  [color={rgb, 255:red, 255; green, 255; blue, 255 }  ,draw opacity=0 ][fill={rgb, 255:red, 245; green, 166; blue, 35 }  ,fill opacity=0.4 ] (21.6,71.81) -- (9.5,71.8) -- (11.57,51.18) -- cycle ;
\draw  [color={rgb, 255:red, 255; green, 255; blue, 255 }  ,draw opacity=0 ][fill={rgb, 255:red, 245; green, 166; blue, 35 }  ,fill opacity=0.4 ] (61.57,151.18) -- (61.57,171.18) -- (8.78,171.19) -- (9.5,71.8) -- (21.6,71.81) -- cycle ;
\draw [color={rgb, 255:red, 65; green, 117; blue, 5 }  ,draw opacity=1 ]   (236.21,126.72) -- (236.21,169.59) ;
\draw [color={rgb, 255:red, 65; green, 117; blue, 5 }  ,draw opacity=1 ]   (236.21,126.72) -- (192,126.72) ;
\draw [color={rgb, 255:red, 65; green, 117; blue, 5 }  ,draw opacity=1 ]   (244.37,110.37) -- (236.21,126.72) ;
\draw [color={rgb, 255:red, 65; green, 117; blue, 5 }  ,draw opacity=1 ]   (268.82,110.37) -- (210.4,168.96) ;
\draw [color={rgb, 255:red, 65; green, 117; blue, 5 }  ,draw opacity=1 ]   (236.21,126.72) -- (194.53,168.53) ;
\draw [color={rgb, 255:red, 65; green, 117; blue, 5 }  ,draw opacity=1 ]   (268.82,110.37) -- (311.14,89.14) ;
\draw [color={rgb, 255:red, 65; green, 117; blue, 5 }  ,draw opacity=1 ]   (268.82,110.37) -- (244.37,110.37) ;
\draw [color={rgb, 255:red, 65; green, 117; blue, 5 }  ,draw opacity=1 ]   (244.37,110.37) -- (191.31,110.37) ;
\draw [color={rgb, 255:red, 65; green, 117; blue, 5 }  ,draw opacity=1 ]   (310.5,52.29) -- (268.82,94.1) ;
\draw [color={rgb, 255:red, 65; green, 117; blue, 5 }  ,draw opacity=1 ]   (268.82,94.1) -- (268.82,110.37) ;
\draw [color={rgb, 255:red, 65; green, 117; blue, 5 }  ,draw opacity=1 ]   (300.72,51.71) -- (244.37,110.37) ;
\draw [color={rgb, 255:red, 65; green, 117; blue, 5 }  ,draw opacity=1 ]   (268.82,94.1) -- (192.19,94.1) ;
\draw [color={rgb, 255:red, 65; green, 117; blue, 5 }  ,draw opacity=1 ]   (244.37,52.71) -- (244.37,110.37) ;
\draw [color={rgb, 255:red, 0; green, 0; blue, 0 }  ,draw opacity=1 ] [dash pattern={on 7.5pt off 1.5pt}]  (61.57,151.18) -- (61.57,171.18) ;
\draw [color={rgb, 255:red, 0; green, 0; blue, 0 }  ,draw opacity=1 ] [dash pattern={on 7.5pt off 1.5pt}]  (111.57,151.18) -- (61.57,151.18) ;
\draw [color={rgb, 255:red, 0; green, 0; blue, 0 }  ,draw opacity=1 ] [dash pattern={on 7.5pt off 1.5pt}]  (61.57,151.18) -- (41.57,111.18) ;
\draw [color={rgb, 255:red, 74; green, 144; blue, 226 }  ,draw opacity=1 ]   (69.5,131.8) -- (89.5,151.8) ;
\draw [color={rgb, 255:red, 74; green, 144; blue, 226 }  ,draw opacity=1 ]   (69.5,131.8) -- (69.5,171.8) ;
\draw [color={rgb, 255:red, 74; green, 144; blue, 226 }  ,draw opacity=1 ]   (89.5,151.8) -- (109.5,171.8) ;
\draw [color={rgb, 255:red, 74; green, 144; blue, 226 }  ,draw opacity=1 ]   (49.5,91.8) -- (69.5,131.8) ;
\draw [color={rgb, 255:red, 74; green, 144; blue, 226 }  ,draw opacity=1 ]   (49.5,51.8) -- (49.5,91.8) ;
\draw [color={rgb, 255:red, 74; green, 144; blue, 226 }  ,draw opacity=1 ]   (29.5,71.8) -- (49.5,91.8) ;
\draw [color={rgb, 255:red, 74; green, 144; blue, 226 }  ,draw opacity=1 ]   (29.5,51.8) -- (29.5,71.8) ;
\draw [color={rgb, 255:red, 74; green, 144; blue, 226 }  ,draw opacity=1 ]   (29.5,71.8) -- (9.5,71.8) ;
\draw [color={rgb, 255:red, 0; green, 0; blue, 0 }  ,draw opacity=1 ] [dash pattern={on 7.5pt off 1.5pt}]  (41.57,111.18) -- (21.57,71.18) ;
\draw [color={rgb, 255:red, 0; green, 0; blue, 0 }  ,draw opacity=1 ] [dash pattern={on 7.5pt off 1.5pt}]  (21.57,71.18) -- (11.57,51.18) ;
\draw [color={rgb, 255:red, 65; green, 117; blue, 5 }  ,draw opacity=1 ]   (416.21,123.86) -- (416.21,146.73) ;
\draw [color={rgb, 255:red, 65; green, 117; blue, 5 }  ,draw opacity=1 ]   (416.21,123.86) -- (392,123.86) ;
\draw [color={rgb, 255:red, 65; green, 117; blue, 5 }  ,draw opacity=1 ]   (416.21,123.86) -- (393.9,146.17) ;
\draw [color={rgb, 255:red, 65; green, 117; blue, 5 }  ,draw opacity=1 ]   (416.21,100.99) -- (416.21,123.86) ;
\draw [color={rgb, 255:red, 65; green, 117; blue, 5 }  ,draw opacity=1 ]   (438.53,101.55) -- (416.21,123.86) ;
\draw [color={rgb, 255:red, 65; green, 117; blue, 5 }  ,draw opacity=1 ]   (416.21,123.86) -- (442.62,110.66) ;

\draw (242.53,177.55) node [anchor=north west][inner sep=0.75pt]   {$\Val(D(f))$};
\draw (408.53,177.35) node [anchor=north west][inner sep=0.75pt]   {$\mathcal{F}( \Delta )$};
\draw (55.53,177.35) node [anchor=north west][inner sep=0.75pt] {$\Xi$};
\draw (327.73,120.75) node [anchor=north west][inner sep=0.75pt]  {$\rightsquigarrow$};
\draw (315.58,87.68) node [anchor=north west][inner sep=0.75pt]  [font=\small]  {$\mathcal{E}_{1}$};
\draw (280.58,50.68) node [anchor=north west][inner sep=0.75pt]  [font=\small]  {$\mathcal{E}_{2}$};
\draw (230.58,50.68) node [anchor=north west][inner sep=0.75pt]  [font=\small]  {$\mathcal{E}_{3}$};
\draw (237.58,150.68) node [anchor=north west][inner sep=0.75pt]  [font=\small]  {$\mathcal{E}_{6}$};
\draw (190.58,150.68) node [anchor=north west][inner sep=0.75pt]  [font=\small]  {$\mathcal{E}_{5}$};
\draw (170.58,105.68) node [anchor=north west][inner sep=0.75pt]  [font=\small]  {$\mathcal{E}_{4}$};

\end{tikzpicture}
\caption{The subdivision $\Xi$, the set $\Val(D(f))$ and $\cF(\Delta)$ corresponding to Example~\ref{ex:computing_polytope_cntd}}\label{fig:ex:computing_polytope_cntd}
\end{figure}

\subsection{Binomial curves and parallel lines}\label{sub:parallel_lines} Recall that the dual fan $\cF(\Delta)$ determines the relative arrangement of edges $e_1,\ldots,e_r\subset\Delta$ together with their slopes. Hence, in order to obtain $\Delta$ up to translation, it is enough to compute the \emph{integer lengths} $\ell_1,\ldots,\ell_r\in\N$ of its respective edges $e_1,\ldots,e_r\subset\Delta$. That is, $\ell_i: = |e_i\cap\N^2|-1$ ($i=1,\ldots,r$).

Let $\cE_1,\ldots,\cE_r\subset \R^2$ denote the collections in $\R^2$ of unbounded edges of $\Val(D(f))$, satisfying
\[
\delta(\cE_i) = e_i,~i=1,\ldots,r.
\] In Example~\ref{ex:computing_polytope_cntd}: $\cE_1$, $\cE_3$ and $\cE_6$ are three half-lines with directions $(1,0)$, $(0,1)$ and $(0,-1)$ respectively, and $\cE_2$, $\cE_4$ and $\cE_5$ are three sets of parallel half-lines with directions $(1,1)$, $(-1,0)$ and $(-1,-1)$ respectively.

Next, we consider collections of tropical curves in $\R^2$ of the form $\max(0,~a_1x_1+a_2x_2+\lambda)$ or\\ $\max(a_1x_1,~a_2x_2+\mu)$ for some $\lambda,\mu\in\R$, and intersect them with $\Val(D(f))$ at the edges $\cE_1,\ldots,\cE_r$. The realization of any such tropical curve is a line (in the classical sense) in $\R^2$ with rational slope. For any vector $\alpha\in\Q^2$, there exists a line $L\subset\R^2$, having direction $\alpha$, and satisfying
\begin{equation}\label{eq:lines-separating}
L\cap \Val(D(f)) =L\cap \bigcup_{i\in K}\cE_i,
\end{equation} for some $K\subset [r]$. Indeed, it is enough to choose $L$ on one side of the cluster of vertices of $\Val(D(f))$. In Figure~\ref{fig:ex:computing_polytope_cntd}, if $\alpha=(1,-1)$, then $K=\{1,2,3\}$.

%
We then consider the set $\cB$ of all binomials in $\K[w_1,w_2]$ with the smallest degree in which the Newton polytope, $\cN(B)$, of any $B\in\cB$ is orthogonal to $L$. 
 For example, if $L$ has direction $(1,-1)$, then $\cB=\{r+sw^{\bm 1}\in\K[w_1,w_2]:r,s\in\K\setminus 0\}$. Thanks to Proposition~\ref{prop:mixed_subd}, we can choose $B\in\cB$ so that $L=\VT(B)$. Then, from the description of $L$, all points in the intersection $\VT(B)\cap\Val(D(f)) $ are stable.

\subsection{Preimages of binomial curves} We will show how to use the set $\cB\subset\K[w_1,w_2]$ for obtaining a linear relation on the set of edge-lengths of $\Delta$. Let $B$ be a generic enough binomial in $\cB$. First, we compute the number of intersection points 
 \begin{equation}\label{eq:binomial-and-discr}
 \VK(B)\cap D(f),
 \end{equation} using the intersections in~\eqref{eq:lines-separating}. The \emph{mixed volume} of two polytopes $\Pi_1,\Pi_2\subset\R^2$ is 
\[
\MV(\Pi_1,\Pi_2):=\Vol(\Pi_1+\Pi_2) -\Vol(\Pi_1) - \Vol(\Pi_2),
\] where $\Vol(\cdot)$ is the usual volume of polytopes in $\R^2$. Bernstein's Theorem~\cite{Ber75} shows that, as $B$ is generic, the number of solutions in $\TK$, counted with multiplicities is equal to $\MV(\cN(B),~\Delta)$. As $\cN(B)$ is a segment with endpoints in $\N^2$, and of integer length one, it can be easily checked
 \begin{equation}\label{eq:mixed_volume=sums}
\MV(\cN(B),~\Delta) = \sum_{i\in K} \ell_i\cdot|\det(u_i,\nu)|,
\end{equation} where $u_i\in\Z^2$ is a primitive integer vector spanning $e_i$, and $\nu\in\Z^2$ spans $\cN(B)$. 

\begin{example}\label{ex:exmpl-pol-cont-cont}
For the polytope $\Delta$ from Example~\ref{ex:computing_polytope}, if $\sigma:=\cN(B)$ is the segment with endpoints $(0,0)$ and $(1,1)$, then $\MV(\Delta,\cN(B))=\Vol(\Delta+\cN(B))-\Vol(\Delta)$. The latter is equal to the sums of the volumes of the zonotopes $\sigma+e_1$, $\sigma+e_2$ and $\sigma+e_3$ (see e.g. Figure~\ref{fig:Minkowski)sum}),  which amounts to 
\[
\ell_1\cdot\det(u_1,\nu)+\ell_2\cdot\det(u_2,\nu)+ \ell_3\cdot\det(u_3,\nu)  = 1\cdot 3 + 6\cdot 2 + 1\cdot 1 = 16.
\] Notice also that it holds $\ell_4\cdot\det(u_4,\nu)+\ell_5\cdot\det(u_5,\nu)+ \ell_6\cdot\det(u_6,\nu)  = 6\cdot 1 + 2\cdot 2 + 6\cdot 1 = 16$. 
\end{example}

\begin{figure}[h]
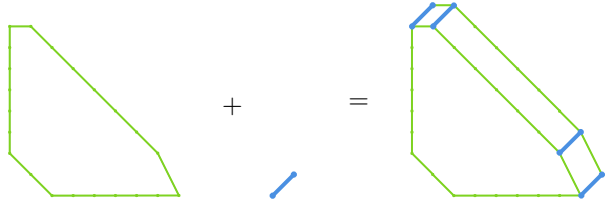


\tikzset{every picture/.style={line width=0.75pt}} 



\caption{The Minkowski sum of $\Delta$ and $\sigma$ from Example~\ref{ex:computing_polytope} .}\label{fig:Minkowski)sum}
\end{figure}

%
%
%
%
%
Let $M$ denote 
\[
\MV\big( \cN(\cj_f),~\cN(B\circ f)\big).
\]
\begin{claim}\label{clm:mixed-volume=tropical-intersection}
It holds $M = \MV(\cN(B),~\Delta)$.
\end{claim}

\begin{proof}
We follow similar steps to those in the proof of~\cite[Lemma 4.3]{farnik2020whitney}. Assume that all components of $D(f)$ are reduced. 
One can choose the coefficients $r,s\in\K\setminus 0$ of $B$ such that~\eqref{eq:binomial-and-discr} consists only of simple roots. Then, it holds
\begin{equation}\label{eq:preimage-enter-crit}
\VK(B)\cap D(f) = f^{-1}(\VK(B))\cap C(f).
\end{equation} Note that $f^{-1}(\VK(B))$ is a curve given by the polynomial $B\circ f$, and recall that $C(f)$ is given by the polynomial $\cj_f$. Therefore, Bernstein's Theorem~\cite{Ber75} shows that 
\[
f^{-1}(\VK(B))\cap C(f) = M
\] if $B\circ f,\cj_f\in\K[z_1,z_2]$ are generic enough. This choice for $f$ and $B$ is possible in $\cO$, and $\cB$ respectively.
\end{proof}

\begin{example}\label{ex:compt_polytope_cont3}
Thanks to Lemma~\ref{lem:crit-points-invar}, all polynomials $\cj_f$ corresponding to any polynomial $f\in\cO$ with supports $A$ from Example~\ref{ex:computing_polytope_cntd} share the same Newton polytope represented in Figure~\ref{fig:New-pol-Jac}. The same goes for the Newton polytope of the polynomial $B\circ f$, where $B$ is a generic binomial in $\K[w_1,w_2]$ with Newton polytope $\sigma$ from Example~\ref{ex:exmpl-pol-cont-cont}. The reader may check (c.f. Figure~\ref{fig:New-pol-Jac}) that $\MV\big( \cN(\cj_f),~\cN(B\circ f)\big)=16$.
\end{example}

\begin{figure}[h]

\tikzset{every picture/.style={line width=0.75pt}} 



\caption{The Minkowski sum of $\cN(\cj_f)$ and $\cN(B\circ f)$ from Example~\ref{ex:compt_polytope_cont3}.}\label{fig:New-pol-Jac}
\end{figure}

\begin{remark}\label{rem:New-pol-Jac}
Thanks to the generic choices for $f$ and $B$, the polytopes $\cN(\cj_f)$ and $\cN(B\circ f)$ can be determined from $A$ by purely combinatorial means. That is, computing $\cj_f$ and $B\circ f$ is not required. Indeed, wehenever coefficients are generic enough, operations on the polynomials such as derivation, summation, multiplication, and exponentiation have well-defined analogues in terms of the Newton polytopes.
\end{remark}

%
%

\subsection{Obtaining the edge-lengths}
In order to determine the vector of integer lengths $\ell:=(\ell_1,\ldots,\ell_r)$ of the polytope $\Delta$, we construct $r$ lines $L_1,\ldots,L_r$, with corresponding orthogonal vectors $\nu_1,\ldots,\nu_r$, and solve for $\ell$ a square system 
\begin{equation}\label{eq:lin-sys}
D\cdot\ell = M,
\end{equation} whose equations are of the form~\eqref{eq:mixed_volume=sums}. As we have seen in Remark~\ref{rem:New-pol-Jac}, the corresponding values $\MV(\cN(B_j),~\Delta)$ and $\big(\det(u_i,\nu_j)\big)_{i,j}$ depend only on the set of supports $A$. The below Claim~\ref{clm:D-non-sing} shows that such a system with non-singular matrix $D$ always exists.

\begin{example}\label{ex:polytope_comp_4}
Consider the six lines $L_1,\ldots,L_6$, where $v_1=v_5=(-1,3)$, $v_2 = v_6= (1,1)$, $v_3=(2,1)$ and $v_4 = (1,-1)$. We position the  lines so that the corresponding sets $K_1,\ldots,K_6$ satisfy $K_1 = \{1,2,3,4\}=[6]\setminus K_5$, $K_2 =K_3=\{1,2,3\}=[6]\setminus K_6$ and $K_4=\{1,6\}$. All together the six equations, each obtained from~\eqref{eq:mixed_volume=sums} applied to the line $L_i$ ($i=1,\ldots,6$), give rise to the linear system 
\begin{equation*}
\begin{pmatrix}
1 & 2 & 3 & 1 & 0 & 0 \\
3 & 2 & 1 & 0 & 0 & 0 \\
5 & 3 & 1 & 0 & 0 & 0 \\
1 & 0 & 0 & 0 & 0 & 1 \\
0 & 0 & 0 & 0 & 2 & 3 \\
0 & 0 & 0 & 1 & 2 & 1 \\
\end{pmatrix}
\cdot
\begin{pmatrix}
\ell_1 \\
\ell_2 \\
\ell_3 \\
\ell_4 \\
\ell_5 \\
\ell_6 \\
\end{pmatrix}
=
\begin{pmatrix}
22 \\
16 \\
24 \\
7 \\
22 \\
16 \\
\end{pmatrix}.
\end{equation*} Therefore, it holds $(\ell_1,~ \ell_2,~\ell_3,~\ell_4,~\ell_5,~\ell_6) = (1, 6, 1, 6, 2, 6)$.
\end{example}
To finish, we give a recipe on how to choose the above mentioned set of lines for any $\Delta$. Recall that for each line $L$ above, we obtain the set of indexes $K\subset [r]$ given by intersecting $L$ with the set $\cE_1\cup\cdots\cup\cE_r$. Assume that $K=\{1,\ldots,k\}$, and that the directional vector of $L$ is not collinear to any ray from the dual fan $\cF(\Delta)$. Then, one can choose $k$ lines $L_1,\ldots,L_k$ in $\R^2$ so that for each $i\in [k]$, the line $L_i$ satisfies~\eqref{eq:lines-separating}. Indeed, this can be done by choosing slightly different tilt for $L$ each time. For each such $L_i$, we consider the line $L'_i$, parallel to $L_i$ so that each $L'_i$ admits an analogous set $K'=[r]\setminus K$. Indeed, it is enough to choose $L'_i$ to be located far enough from $L_i$ and on the other side of the cluster of vertexes of $\Val (D(f))$. Assume without loss of generality that $k=\max(|K|,|K'|)$. 


%
%
%
%
%
%
%
%

\begin{claim}\label{clm:D-non-sing}
There exists a choice $(L_j,L_j')_{j\in [k]}$ satisfying~\eqref{eq:lines-separating}, and such that $D$ is a non-singular matrix. 
\end{claim}

\begin{proof}
Any line in $\R^2$ can be expressed of the form $\{y_2= \alpha y_1 + \beta\}$. Thus, we can identify the configuration space of lines $(L_j)_{j\in [k]}$ with $\R^{2\times k}$. If a pair $(L,L')$ satisfies~\eqref{eq:lines-separating}, then so does any of its perturbations. Then, there exists an open subset $U\subset\R^{2\times k}$ in which $(L_j,L'_j)_{j\in K}$ satisfies~\eqref{eq:lines-separating}.

Next, we find an open subset $V\subset U$ consisting of all lines $\{y_2= \alpha_j y_1 + \beta_j\}_{j\in K}$ so that $\det D\neq 0$. This is straight-forward as this determinant is locally a non-trivial polynomial in $\Z[\alpha_j,~j\in [k]]$. Hence we choose $V=U\cap (\R^{2\times k} \setminus \{\det D= 0\})$. 
\end{proof}

%
%

\section{Proofs of Proposition~\ref{prop:t-crit_mixed-cells} and Lemma~\ref{lem:val-F_F-val}}\label{sec:tropical_critical}
We start this section by stating a technical result, Lemma~\ref{lem:Jacobian_at_origin2}, then proceed by proving Proposition~\ref{prop:t-crit_mixed-cells} and Lemma~\ref{lem:val-F_F-val}. In what follows, we keep the notations in \S\ref{sec:proof-main-th}, but we assume that $f$ lies inside the sets $\Omega_1$ and $\Omega_2$ defined in Lemmas~\ref{lem:gen-valued_transversal} and~\ref{lem:crit-points-invar}. 

Recall from \S\ref{subsub:subdivisions}, if $\xi$ is a cell in $\Xi$, then there exists a unique pair $(\xi_1,\xi_2)\in\Xi_1\times \Xi_2$ such that $\xi=\xi_1\cap\xi_2$. Furthermore,  the dual polyhedron of $\xi_i$ in $\tau_i$  is denoted  by $\delta(\xi_i)$ ($i=1,2$), and the dual polyhedron of $\xi$ in $\tau$ is $\delta(\xi_1)+\delta(\xi_2)$, denoted by $\delta(\xi)$. Recall Notation~\ref{not:restrictions}.
\begin{notation}\label{not:restrictions2}
Given a pair $g:=(g_1,g_2)$ of polynomials in $\K[z_1,z_2]$, we use $\overline{g}$ to denote the pair $(\overline{g}_1,\overline{g}_2)$ of polynomials in $\C[z_1,z_2]$. 

\end{notation}

\begin{lemma}\label{lem:Jacobian_at_origin2}
Let $\xi$ be a cell in $\Xi$ containing $(0,0)$, and assume that $|\Jac_{z} \overline{f}|$ is not identically zero. Then, it holds 
\[
\overline{\cj}_f = |\Jac_{z} \overline{f}|.
\]
\end{lemma}

\begin{proof} For each $i\in\{1,2\}$, we express $f_i$ as $ \sum c^i_az_1^{a_1}z_2^{a_2}$, and let $\sigma(\xi_i)$ be the set of all $b\in\delta(\xi_i)$ satisfying
\[
\val(c^i_b)=\max\{\val(c^i_a)~|~a\in\delta(\xi_i)\} \text{,~ for } i=1,2.
\] Recall from \S\ref{subs:tropical_polynomials} that the tropical function $F_i:\R^2\to\R$ is linear precisely at the cells in $\Xi_i$. In particular, for any $x\in\xi$, and any $v\in\sigma(\xi_i)$, it holds
\[
F_i(x_1,x_2) =  v_1 x_1 + v_2 x_2 +\val(c^i_v).
\]
Furthermore, it holds
\begin{equation}\label{eq:condition_orders}
\val(c^i_\alpha) = \val(c^i_\beta)>\val(c^i_\gamma),\quad \alpha,\beta\in \sigma(\xi_i),~\text{and } \gamma\in A_i\setminus \sigma(\xi_i).
\end{equation} 

To finish the proof, note that one can write $\cj_f$ as the sum $ G + H$, where $G=\det(\Jac_z \overline{f})$, and $H\in \K[z_1,z_2]$. Then, the assumption $G\not\equiv 0$ and Condition~\eqref{eq:condition_orders} implies that all coefficients in $G$ have the same valuation, say, $\lambda$, and any coefficient $c_a$ in the polynomial $H$ satisfies $\val(c_a)<\lambda$. Therefore, the proof follows from $ \overline{G}= |\Jac_{z} \overline{f}|$.
\end{proof}

\begin{proof}[Proof of Proposition~\ref{prop:t-crit_mixed-cells}]
Let $\xi$ be a cell in $\Xi$. Then, it satisfies one of the conditions of Definition~\ref{def:main}. We may assume without loss of generality that $(0,0)\in\xi$. 
The proof requires computing the polynomial $\overline{\cj}_f$ in each of the situations \textbf{0} to \textbf{2} described in Definition~\ref{def:main} (these computations will also be useful later for the proof of Lemma~\ref{lem:val-F_F-val}). The polynomial $\overline{\cj}_f$ is made explicit in \S\ref{sec:appendix_some_proofs}, where we will use Lemma~\ref{lem:Jacobian_at_origin2} to count the number $N$ of monomial terms in $\overline{\cj}_f$ whenever $\overline{\cj}_\xi:=|\Jac_{z} \overline{f}_{\xi}|$ is not identically 
 zero. That is, thanks to Corollary~\ref{cor:lower-order}, we get 
\[
N>1~\Leftrightarrow C(F)\cap\xi\neq\emptyset. 
\] This will yield the first statement of the Proposition whenever $\xi$ satisfies conditions~\ref{it:0},~\ref{it:11},~\ref{it:12} or~\ref{it:21} of Definition~\ref{def:main}. As for the remaining cases, we will obtain $\overline{\cj}_\xi\equiv 0$, and thus a more detailed analysis is required for computing $\overline{\cj}_\xi$. This case-by-case analysis/computations is elaborated in \S\ref{sec:appendix_some_proofs}.
\end{proof}

\begin{proof}[Proof of Lemma~\ref{lem:val-F_F-val}]
Let $\xi$ denote the cell in $\Xi$ containing $x$, and assume without loss of generality that $x=(0,0)$. It is enough to show that there exists a solution $z\in C(f)$ with $\Val z = (0,0)$, and such that $\overline{z}$ is not a solution to $\overline{f_1} = 0$ nor to  $\overline{f_2} = 0$. Indeed, once this is the case, then the point $w:=f(z) $ would satisfy
$
\overline{w}=(\overline{f}(\overline{z}))\in (\C^*)^2,
$ and thus we obtain $\Val (w) =\Val (f(z)) = F(\Val(z)) = F(0,0)$.  

For each $\xi\in\Xi$ and each $i=1,2$, we show that either $\overline{f}_i$ is a monomial term or it does not divide the polynomial $\overline{G}$ from Lemma~\ref{lem:Jacobian_at_origin2}). Since both of $\overline{f}_1,\overline{f}_2$ are monomial terms if $\xi$ is a $2$-cell, we only consider the cases where $\xi$ is a $0$-cell or a $1$-cell.

 Assume without loss of generality that $\xi\subset T_1$. The proof proceeds by simply comparing the polynomial $\overline{f}_{i}$ ($i=1,2$) with $\overline{G}$ in all the situations for $\xi$. The latter cases are analyzed in detail, in \S\ref{sec:appendix_some_proofs}, of the proof of Proposition~\ref{prop:t-crit_mixed-cells}. 
\end{proof}

\section{Proof of Proposition~\ref{prop:super-critical-mixed-cells}}\label{sec:supercr} 
Let $A:=(A_1,A_2)$ be a pair of subsets in $\N^2\setminus\{(0,0)\}$, let $\Omega\subset\K[A]$ be the Zariski open from Proposition~\ref{prop:t-crit_mixed-cells}, and let $f\in\Omega$. We use $F:=(F_1,F_2):\R^2\to\R^2$ to denote the tropical polynomial map $f^{\trop}$. It induces a subdivision $\Xi$ of $\R^2$ as in Proposition~\ref{prop:mixed_subd}. 
\subsection{Proof of Proposition~\ref{prop:super-critical-mixed-cells}: \textbf{a.}}\label{sub:supercr} 
We need the following two results. 

\begin{lemma}\label{lem:Brugalle-Lopez-de-Medrano}
Let $X_1$ and $X_2$ be two algebraic curves in $(\K\setminus 0)^2$, intersecting in a finite number of points, and let $E\subset\R^2$ be a bounded component of the intersection of $\Val (X_1)$ and $\Val( X_2)$. Then, there exists a point $z\in X_1\cap X_2$ such that $\Val(z)\in E$.
\end{lemma}

\begin{proof}
This is a consequence of~\cite[Proposition 3.11]{Br-deMe11}. 
\end{proof}
\begin{lemma}\label{lem:adjacent_diag_relevant}
Let $\xi\in\Xi$ be a relevant diagonal $2$-cell, and let $\xi'\in\Xi$ be a $2$-cell directly adjacent to $\xi$. Then, either $\xi'$ is irrelevant or it is diagonal relevant and such that 
\begin{equation}\label{eq:aligned}
\left\{\delta(\xi_1),\delta(\xi_2),\delta(\xi'_1),\delta(\xi'_2)\right\}\subset L_0,
\end{equation} for some line $L_0\subset\R^2$ is a line passing through $0$.
\end{lemma}

\begin{proof}
Since both $\xi$ and $\xi'$ are $2$-cells, each polyhedron in~\eqref{eq:aligned} is a point in $\N^2$. We may assume that the $1$-cell common to $\xi$ and $\xi'$ belongs to $T_1$. 
Then, it holds
\begin{equation}\label{eq:directly-adjacent-cells_relations}
\delta(\xi_2) = \delta(\xi'_2)\text{ and }\delta(\xi_1) \neq \delta(\xi'_1).
\end{equation} Furthermore, since $\xi$ is relevant, there exists $r\in\Q$ such that
\begin{equation}\label{eq:relevant-cell_property}
\delta(\xi_1) = r \cdot\delta(\xi_2).
\end{equation} If $\xi'$ is relevant, then, there exists $s\in\Q$ such that $\delta(\xi'_1) = s\cdot \delta(\xi'_2)$. Hence, Equations~\eqref{eq:directly-adjacent-cells_relations} and~\eqref{eq:relevant-cell_property} yield~\eqref{eq:aligned}. Finally, the inclusion~\eqref{eq:aligned} implies that $\xi'$ is also diagonal.
\end{proof} Recall that supercritical cells have dimension at most one (see Definition~\ref{def:super-crit}). Assume that $\xi$ is super-critical to $T_1$. Proposition~\ref{prop:t-crit_mixed-cells} shows that $\xi\cap C(F)\neq\emptyset$. Without loss of generality, we assume that $(0,0)\in \xi\cap C(F)$. Let $a\in \K\setminus 0$ be a point such that $(0,0)\in T_1(\val (a))$, and $C(f)\cap\mathbb{V}(f-a)$ has only isolated solutions in $\TK$. This assumption can be realized by applying a perturbation of $a$ if necessary so that $\val (a)$ remains unchanged. 

If $(0,0)$ belongs to a bounded component of $C(F)\cap T_1(\val (a))$, then the proof of Proposition~\ref{prop:super-critical-mixed-cells} $\textbf{a.}$ follows from Lemma~\ref{lem:Brugalle-Lopez-de-Medrano}. This is exactly the case if either $\xi$ is a $0$-cell as in Figure~\ref{fig:situations-table} $\color{royalblue}\bm a.\text{\textbf{i}}$, $\color{royalblue}\bm e.\text{\textbf{iii}}$ or, thanks to the below claim, if $\xi$ is a $1$-cell.

\begin{claim}\label{clm:adj-cells}
Let $\xi$ be a cell of dimension $1$ that is super critical to $T_1$, and assume that none of the  endpoints of $\xi$ is a point in $T_1\cap T_2$. Then, for any $\lambda\in\R$, either $C(F)\cap T_1(\lambda)\cap \xi=\emptyset $ or $\xi$ contains a bounded component of $C(F)\cap T_1(\lambda)$. 
\end{claim}

\begin{proof} Since $\xi$ is supercritical to $T_1$, according to Definition~\ref{def:super-crit}, either $\xi$ is adjacent to two relevant diagonal cells or it satisfies Definition~\ref{def:main}:~\ref{it:132}, in which $\zeta$ is a vertex of $T_1$. Assume that $C(F)\cap T_1(\lambda)\cap \xi\neq\emptyset $. Then, $\xi$ does not satisfy~\ref{it:1322} since $\zeta$ and $T_1$ satisfy~\ref{it:022}. If $\xi$ satisfies~\ref{it:1321}, then the proof follows from Proposition~\ref{prop:t-crit_mixed-cells} $(2)$. Otherwise, $\xi$ is adjacent to two relevant diagonal cells, and thus the proof follows from Definition~\ref{def:lat-diag}.
\end{proof}

We now treat the remaining cases for super-critical cells. That is, we assume that $\xi$ is a $0$-cell in one of the situations $\color{royalblue}\bm c.\text{\textbf{i}}$, $\color{royalblue}\bm c.\text{\textbf{iii}}$, $\color{royalblue}\bm b.\text{\textbf{i}}$ or $\color{royalblue}\bm d.\text{\textbf{iii}}$, represented in Figure~\ref{fig:situations-table}. Consider now the system
\begin{equation}\label{eq:sys:main-half-3}
\left\{
  \begin{array}{@{}lcc@{}}
   f_1(z) - a & = & 0,\\
   \cj_f(z) & = & 0,
  \end{array}
\right.
\end{equation} and let $C$ denote the space-curve defined as the set $(z_1,z_2,a)\in (\K\setminus 0)^3 $ forming its zero-locus. The projection $\pi^{\K}:(\K\setminus 0)^3\to\TK$, $(z_1,z_2,a)\to (z_1,z_2)$ sends $C$ surjectively onto the curve $C(f)$. The tropicalization 
\[
\Gamma:=\Val(C)
\] is its image under the map $\Val:(\K\setminus 0)^3\to\R^3$, $(z_1,z_2,a)\to (\val(z_1),\val(z_2),\val (a))$. Similarly, the set $\Gamma$ is sent surjectively to the tropical curve $C(F)\subset\R^2$ under the projection $\pi:\R^3\to\R^2$, $(x_1,x_2,\lambda)\to (x_1,x_2)$. Here, we have $\pi\circ\Val = \Val\circ\pi^{\K}$. We say that $\Gamma$ is an \emph{embedded tropical space-curve} in $\R^3$ (see~\cite{mikhalkin2008tropical} ). That is, a realization in $\R^3$ of a metrized graph $(E,V)$ having integer slopes, together with a weight function $\omega:E\to\N$, where for any $v\in V$ the set $E_v$ of edges adjacent to $v$ it holds
\begin{equation}\label{eq:balancing}
\sum_{e\in E_v} \omega (e)\vec{e} = 0,
\end{equation} where $\vec{e}$ is the primitive integer vector in $\Z^3$ pointing away from $v$. We call~\eqref{eq:balancing} the \emph{balancing condition} on $v$.

Recall that $\xi = (0,0)$. On the one hand, since $\xi$ is a point in $ T_1$, there exists $\lambda_0\in\R$ satisfying: $\xi\in T_1(\val a)$ for any $a\in \K$ if and only if $\val(a)\leq \lambda_0$. On the other hand, each point $p\in\Gamma$ is the valuation of a solution $(z_1,z_2,a)\in(\K\setminus 0)^3$ to the system~\eqref{eq:sys:main-half-3}. Therefore, Proposition~\ref{prop:super-critical-mixed-cells} \textbf{a.} is true if $\Gamma$ contains all points $(0,0,\lambda)$ satisfying $\lambda\leq \lambda_0$. The following Claim yields the proof of Proposition \textbf{a.} (see Figure~\ref{fig:four-cases}).

\begin{claim}\label{clm:vertical_edge}
Assume that $\xi$ satisfies one of the situations $\color{royalblue}\bm c.\text{\textbf{i}}$, $\color{royalblue}\bm c.\text{\textbf{iii}}$, $\color{royalblue}\bm b.\text{\textbf{i}}$ or $\color{royalblue}\bm d.\text{\textbf{iii}}$, represented in Figure~\ref{fig:situations-table}. Then, $\Gamma\subset \R^3$ has an edge of direction $(0,0,-1)$ with an endpoint $(0,0,\lambda_0)$.
\end{claim}

\begin{proof}
The embedded tropical curve $\Gamma$ is contained in the set $S_{1\cap J}$ defined as intersection locus of the two tropical surfaces in $\R^2$ 
\[
\begin{matrix}
S_1:=\left\lbrace \Val(z_1,z_2,a)\in( \K\setminus 0)^3~\left|~ f(z) - a =0 \right\rbrace\right. & \text{and} & S_J:=\left\lbrace \Val(z_1,z_2,a)\in(\K\setminus 0)^3~\left|~\cj_f(z)=0 \right\rbrace\right.  .
\end{matrix}
\]
On the one hand, the map $\pi|_{S_1}$ is a bijection over $\R^2\setminus T_1$, whereas the preimage, under $\pi|_{S_1}$, of any point in $T_1$ is a half-ray with direction $(0,0,-1)$ (see~\cite[\S 4]{Br-deMe11}). On the other hand, 
the preimage, under $\pi|_{S_J}$, over any point $x$ is empty if $x\in\R^2\setminus S_J$, and is a vertical line otherwise. Therefore, the map 
\[
\pi|_{S_{1\cap J}}:S_{1\cap J}\to C(F)
\] is a bijection over points outside $T_1$, and is a half-ray over points in $T_1\cap C(F)$ (see e.g. Figure~\ref{fig:four-cases}).

 Now, we describe the disposition of $S_{1\cap J}$ locally above $(0,0)$. Let $V_0\subset\R^2$ be a neighborhood of $(0,0)$, and let $U_0$ denote $\pi^{-1}(V_0)$. Then, the intersection $S_{1\cap J}\cap U_0$ is a piecewise-linear complex, consisting of a set $\Theta$ of half-edges, and a set $\Sigma$ of $2$-dimensional polygons in $\R^3$. All half-edges in $\Theta$ have a common endpoint, $v$ where $\pi(v)=(0,0)$. All polygons in $\Sigma$ have a common vertex $v$, and a common $E$ adjacent to $v$, such that $E$ is a half-line with direction $(0,0,-1)$ (see Figure~\ref{fig:four-cases}). Furthermore, for each half-edge $e\subset C(F)\cap V_0\setminus T_1$, there exists $\theta\in\Theta$ such that $\pi^{-1}(e)=\theta$. Similarly, for each half-edge $e\subset C(F)\cap T_1\cap V_0$, there exists $\sigma\in\Sigma$ such that $\pi^{-1}(e)=\sigma$.

Next, we describe the disposition of $\Gamma$ locally above $(0,0)$. 
The intersection $\Gamma\cap U_0$ is a set of half-edges, denoted by $\Upsilon$, that satisfy the following conditions.

\begin{enumerate}

	\item \emph{Any $\mu\in\Upsilon$ has $v$ as an endpoint:} This follows from the description of $S_{1\cap J}\cap U_0$, since $\Gamma\subset S_{1\cap J}$.
	
	\item \emph{It holds $\Theta\subset\Upsilon$:} From the description of $ S_{1\cap J}$, any point in $C(F)\cap V_0\setminus T_1$ has exactly one preimage $q\in S_{1\cap J}$, under $\pi|_{S_{1\cap J}}$. Since $\pi (\Gamma) =C(F)$, we obtain $q\in\Gamma$.
	
	\item \emph{For each $\sigma\in\Sigma$, there exists one $\mu\in\Upsilon$ such that $\mu\subset\partial\sigma$, and $\mu$ is not vertical:} If $e$ is a half-edge of $C(F)\cap T_1\cap V_0$, then $e\subset\xi$, where $\xi\in\Xi$ is a $1$-cell directly adjacent to two irrelevant $2$-faces of $\Xi$. That is, $\xi$ is not super-critical. Then, Proposition~\ref{prop:super-critical-mixed-cells} \textbf{b.} shows that any $q\in\Gamma$ satisfying $\pi(q)\in\xi$, the third coordinate of $q$ is $\val(\pi(q))$. In other words, it holds that $q\in\partial\sigma$.

\end{enumerate}
 Since $F_1:\R^2\to\R$ is a convex function, the above three facts show that all non-vertical edges of $\Gamma$, emanating from $v$, are at the boundary of a polytope $\Delta$ in $\R^3$ (in Figure~\ref{fig:four-cases}, the faces of $\Delta$ are the blue non-vertical plane sections). Hence, the balancing condition applied to $\Gamma$ around $v$, implies that $\Gamma$ must have at least one edge $\theta$, with direction outwards of $\Delta$. Fact (3) above shows that $\theta$ must be vertical. This yields the claim.
\end{proof}

\begin{figure}[h]
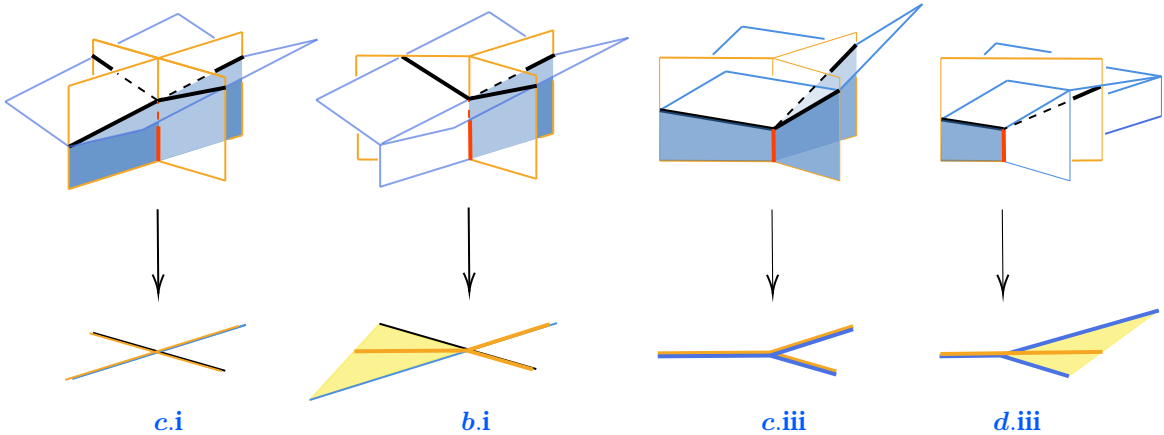



\caption{The set $S_{1\cap J}$ in each of the cases in Claim~\ref{clm:vertical_edge}. It projects onto $C(F)$, contains $\Gamma$, and is formed by the intersection of $S_1$ (in blue) with $S_J$ (in orange).}\label{fig:four-cases}
\end{figure}

\subsection{Proof of Proposition~\ref{prop:super-critical-mixed-cells} \textbf{b.}}\label{sub:non-supercr} Assume that $\xi$ is \emph{not} supercritical to $T_1$, and let $w:=(w_1,w_2)\in\TK$ be any point satisfying $\val(w_1)\neq F_1(x)$.
First, note that the equivalence 
\[
p\in T_1(\val (w_1))\Leftrightarrow \val(w_1) = F_1(x)
\] holds true if $\xi\cap T_1 = \emptyset$. This yields Proposition~\ref{prop:super-critical-mixed-cells} \textbf{b)} if $\xi\subset T_2\setminus T_1$ or if $\dim\xi = 2$. Therefore, we assume in what follows that $\xi\subset T_1$, and $\dim \xi\leq 1$.

\subsubsection{$\xi$ is a $0$-cell}\label{subsubsec:xi-0-cell-NOT-super} If $\xi$ is a vertex of $T_1$, Definition~\ref{def:super-crit} shows that $\xi$ is adjacent to two relevant $2$-cells (e.g. $\color{royalblue}\bm d.\text{\textbf{iv}}$ in Figure~\ref{fig:situations-table}). Proposition~\ref{prop:t-crit_mixed-cells} $(1)$ yields $\xi\cap C(F)=\emptyset$, and we deduce Proposition~\ref{prop:super-critical-mixed-cells} \textbf{b)}.

Assume now that $\xi\in T_1\cap T_2$. Then, locally around $\xi$, the situation is as $\color{royalblue}\bm e.\text{\textbf{ii}}$ in Figure~\ref{fig:situations-table}, where the continuous line is a part of $T_1$, and the dashed line is a part of $T_2$ (notice the difference in the disposition of curves between $\color{royalblue}\bm e.\text{\textbf{ii}}$ and $\color{royalblue}\bm e.\text{\textbf{iii}}$). Without loss of generality, we may assume that $x=(0,0)$.  Then, following Lemma~\ref{lem:Jacobian_at_origin2}, it holds 
\[
z^{\bm 1}\overline{\cj}_f = \overline{c}z^{v} \left|\begin{smallmatrix} u_1 &  v_1\\ u_2 &  v_2 \end{smallmatrix} \right|\cdot (\overline{a}\lambda z^{\lambda\cdot u} + \overline{b}\mu z^{\mu\cdot u}),
\] where 
\[
\overline{f} = (az^{\lambda\cdot u} + bz^{\mu\cdot u},~cz^{v} + dz^{\kappa\cdot u})
\]
for some values $\lambda,\mu,\kappa\in\Z$, $a,b,c,d\in\K\setminus 0$, and vectors $u,v\in\N^2$ such that $\lambda\neq\mu$. Since $\overline{\cj}_f = \overline{f_1} = 0$ has no solutions in $\TT$, the system~\eqref{eq:sys:main-half-3} has no solutions in $\TK$ with valuation $(0,0)$. This yields Proposition~\ref{prop:super-critical-mixed-cells} \textbf{b)} if $\xi$ is not a super-critical $0$-cell.

\subsubsection{$\xi$ is a $1$-cell}\label{subsec:xi-1-cell-NOT-super} If $\xi$ is an edge of $T_1$ directly adjacent to only one relevant $2$-cell of $\Xi$ (e.g. $\color{royalblue} \bm{e.}\text{\textbf{iv}}$ in Figure~\ref{fig:situations-table}), 
then Proposition~\ref{prop:t-crit_mixed-cells} yields $\xi\cap C(F)=\emptyset$, and we are done. The same can be said for the case where $\xi$ is directly adjacent to two relevant lateral cells having no essential  $1$-cells adjacent to them, and the 
endpoint of $\xi$ is a vertex of $T_1$ (e.g. $\color{royalblue} \bm{a.}\text{\textbf{vi}}$). 

Assume that $\xi$ is directly adjacent to two relevant lateral cells having no essential $1$-cells adjacent to them, and the endpoint of $\xi$ belongs to $T_1\cap T_2$ (e.g. $\color{royalblue} \bm{e.}\text{\textbf{vii}}$). Similarly as in the above \S\ref{subsubsec:xi-0-cell-NOT-super}, we assume that $(0,0)\in\xi$, and compute $\overline{\cj}_\xi$. Without loss of generality, we may assume that $\xi$ is vertical. Then, the polynomial $\overline{\cj}_\xi$ is a monomial times $k\overline{a}z_1^k + \overline{b}lz_1^l$, where $f_{1|\xi} = az_1^k + bz_1^l$ (\S\ref{sec:appendix_some_proofs} (5)). Since we have $k\neq l$, the system $\overline{\cj}_\xi = \overline{f_1} =0 $  has no solutions in $\TT$. This yields Proposition~\ref{prop:super-critical-mixed-cells} \textbf{b)} for such $\xi$.

Finally, if $\xi$ is directly adjacent to two irrelevant cells (e.g. $\color{royalblue} \bm{d.}\text{\textbf{v}}$), then the proof follows from Lemma~\ref{lem:transversal} as $ T_1(\val (w_1))\cap T_2(\val (w_2))$ has an unstable intersection at $x\in\xi$ if and only if $F(x)=\Val(w)$.

\section{Future work}\label{sec:future_work}
Two maps $f,g:\C^m\to \C^n$ are said to be \emph{topologically equivalent} if there are homeomorphisms $\phi:\C^m\to \C^m$ and $\psi:\C^n\to \C^n$ so that $g\circ \phi = \psi\circ f$. The classification of topological equivalence classes (or, topological types) is a difficult problem in affine geometry (c.f.~\cite{Fuk76,Aok80,Sab83,Nak84,jelonek2001topological}) that is solved only for quadratic polynomial maps $\C^2\to\C^n$ using a case-by-case analysis~\cite{FJ17,FJM18}. 
Such an approach, however, becomes too cumbersome when generalized to maps of higher degrees. 
As two topological types for polynomial maps are dissimilar if their respective discriminants are not homeomorphic, one application of Theorem~\ref{thm:main} is to classify (or, enumerate) homeomorphism classes of discriminants for maps having the same support. A classification of combinatorial types of tropical polynomial maps $\R^2\to\R^2$ can lead to lower bounds for homeomorphism classes of discriminants on the plane. 

A closely-related problem to the classification of maps is the classification of singularities of the corresponding discriminants. Farnik, Jelonek and Ruas~\cite{FJR19} used the theory of multi-jets~\cite{mather1973generic} to describe the topology of a generic polynomial map $f:=(f_1,f_2):X\to\C^2$ and to
compute the numbers $n(f)$ and $c(f)$ of nodes and cusps, respectively of $D(f)$ in terms of the degrees of $f_1$ and $f_2$. 
In light of Theorem~\ref{th:Newton-polytope}, we expect that the pair $(n(f),~c(f))$ can be obtained, using Theorem~\ref{thm:main}, in terms of invariants of the supports $A_1$ and $A_2$ of $f_1$ and $f_2$ if $f$ is generic enough in $\C[A]$.

An approach, similar to Theorem~\ref{thm:main} can be used to describe the tropical curve of the bifurcation set for polynomial maps on the plane. Doing so requires recovering the tropical curve of the \emph{non-properness set} (or non-finiteness set) $S(f)$ for a polynomial map $f:\K^2\to\K^2$. This is the set of all points $y\in\K^2$ at which the preimage has less isolated points (counted with multiplicity) than its topological degree~\cite{Jel93,jelonek2001topological}. Since $B(f) = D(f)\cup S(f)$ (see e.g.~\cite{Jel16}), in order to characterize the tropical bifurcation set, an analogue of Theorem~\ref{thm:main} for $S(f)$ remains to be proven. This is the subject of a work in progress~\cite{EH2022tropinon}, where the polynomial maps are defined over the space $\K^n$ of any dimension.

\appendix

\section{Rest of proof of Proposition~\ref{prop:t-crit_mixed-cells}}\label{sec:appendix_some_proofs}
In this section, we keep the notations of \S\ref{sec:tropical_critical}. In particular, we assume that $f\in\Omega_1\cap\Omega_2$ (see Lemmas~\ref{lem:gen-valued_transversal} and~\ref{lem:crit-points-invar}). We will compute the polynomial $\overline{\cj}_f$ for a cell $\xi\in\Xi$ containing the origin $(0,0)$. In what follows, we use $|M|$ instead of $\det (M)$ to denote the determinant of any square matrix $M$. Let $\overline{J}_\xi$ denote the polynomial $|\Jac_{z} \overline{f}|$, let $\bm 1$ denote the point $(1,1)$, and let $A:=(A_1,A_2)$ denote the supports of $f$. With some exceptions for the below situations of $\xi$ (see Item \textbf{(5)} below), it is enough to set $\Omega$ as $\Omega_1\cap\Omega_2$. 

\begin{enumerate}

	\item \textbf{Zero-dimensional cells from intersections.} 
Assume that $\xi\in T_1\cap T_2$. Lemma~\ref{lem:gen-valued_transversal} shows that there exists vectors $u,v,u',v'\in\N^2$ satisfying $\delta(\xi_1) =\{u,v\}$, $\delta(\xi_2) =\{u',v'\}$, and
\[
\overline{f} = (\overline{a_1}z^{u} + \overline{b_1}z^{v},~\overline{a_2}z^{u'} + \overline{b_2}z^{v'})
\] for some values $a_1,a_2,b_1,b_2\in \K\setminus 0$. The polynomial $z^{\bm 1}\cdot\overline{J}_\xi$ can be expressed as the sum of the four monomial terms
\[
\overline{a_1}\overline{a_2}z^{u+u'}\left| 
\begin{matrix}
u_1 & u_1'\\[4pt]
u_2 & u_2'
\end{matrix}
\right| 
+ \overline{a_1}\overline{b_2}z^{u+v'}
\left| 
\begin{matrix}
u_1 & v_1'\\[4pt]
u_2 & v_2'
\end{matrix}
\right| 
+ \overline{b_1}\overline{a_2}z^{v+u'}
\left| 
\begin{matrix}
v_1 & u_1'\\[4pt]
v_2 & u_2'
\end{matrix}
\right| 
+ \overline{b_1}\overline{b_2}z^{v+v'}
\left| 
\begin{matrix}
v_1 & v_1'\\[4pt]
v_2 & v_2'
\end{matrix}
\right|. 
\] Since $\xi$ is a stable intersection point of $T_1$ and $T_2$, one can check that $\overline{J}_\xi$ has at least two non-zero coefficients. This proves Proposition~\ref{prop:t-crit_mixed-cells} (1) when $\xi$ satisfies Definition~\ref{def:main}:~\ref{it:132}. \\
\item \textbf{Zero-dimensional cells from vertices.} 
Assume that $\xi$ is a vertex of $T_1$. As all the vertexes of $T_1$ are trivalent (see Lemma~\ref{lem:gen-valued_transversal}), there exists vectors $u,v,w,k\in\N^2$ satisfying $\delta(\xi_1) =\{u,v,w\}$, and $\delta(\xi_2) =\{k\}$, and thus 
\[
\overline{f} = (\overline{a}z^{u} + \overline{b}z^{v} + \overline{c}z^{w},~\overline{d}z^{k}),
\]
for some values $a,b,c,d\in\K\setminus 0$. The polynomial $z^{\bm 1}\cdot\overline{J}_\xi$ is then expressed as 
\[
\overline{d}z^{k}
\left(
\overline{a}z^{u}\left| 
\begin{matrix}
u_1 & k_1\\[4pt]
u_2 & k_2
\end{matrix}
\right| 
+ \overline{b}z^{v}\left| 
\begin{matrix}
v_1 & k_1\\[4pt]
v_2 & k_2
\end{matrix}
\right| 
+ \overline{c}z^{w}\left| 
\begin{matrix}
w_1 & k_1\\[4pt]
w_2 & k_2
\end{matrix}
\right|  
\right).
\] Since $\delta(\xi_1)$ is a set of three non-collinear points $u$, $v$, $w$, at least one of the monomials in $\overline{J}_\xi$ has a non-zero coefficient. Note that two coefficients are zero if and only if $\xi$ satisfies condition~\ref{it:022} of Definition~\ref{def:main}. This proves Proposition~\ref{prop:t-crit_mixed-cells} (1) when $\xi$ satisfies condition~\ref{it:02} of Definition~\ref{def:main}.\\

 \item \textbf{One-dimensional cell next to an irrelevant one.} 
Assume that $\xi$ is an edge of $T_1$ directly adjacent to an irrelevant cell. Lemma~\ref{lem:gen-valued_transversal} shows that there exists vectors $u,v,k\in\N^2$ satisfying $\delta(\xi_1) =\{u,v\}$, and $\delta(\xi_2) =\{k\}$, and thus 
\[
\overline{f} = (\overline{a}z^{u} + \overline{b}z^{v},~\overline{d}z^{k}),
\]
for some values $a,b,d\in\K\setminus 0$. The polynomial $z^{\bm 1}\cdot\overline{J}_\xi$ is then expressed as 
\[
\overline{d}z^{k}
\left(
\overline{a}z^{u}\left| 
\begin{matrix}
u_1 & k_1\\[4pt]
u_2 & k_2
\end{matrix}
\right| 
+ \overline{b}z^{v}\left| 
\begin{matrix}
v_1 & k_1\\[4pt]
v_2 & k_2
\end{matrix}
\right| 
\right).
\] If only one of the vectors $u$ or $v$ has same direction as $k$, then $\overline{J}_\xi$ has only one monomial term with non-zero coefficient. If none of them has the same direction as $k$, then $\overline{J}_\xi$ consists of exactly two monomial terms with non-zero coefficients. This proves Proposition~\ref{prop:t-crit_mixed-cells} (1) when $\xi$ satisfies condition~\ref{it:11} or~\ref{it:12} of Definition~\ref{def:main}.\\

 \item \textbf{One-dimensional cell next to two diagonal relevant ones.} 
Assume that $\xi$ is an edge of $T_1$, directly adjacent to two relevant diagonal cells $\sigma$ and $\sigma'$. Then, the exponent vectors $u,v,k\in\N^2$, computed in the previous case, now have the same direction, and thus $\overline{J}_\xi\equiv 0$. 
Lemma~\ref{lem:adjacent_diag_relevant}, and Proposition~\ref{prop:mixed_subd} (3) show that for any $q\in \R^2$, the set $F^{-1}(q)\cap(\sigma\cup\sigma')$ is either empty or a line segment with same direction as $\xi$. We conclude that, if Proposition~\ref{prop:t-crit_mixed-cells}:(1) is true for $\sigma$ and $\sigma'$, then it holds true for $\xi$ as well. This, together with Item (7) later, proves Proposition~\ref{prop:t-crit_mixed-cells} when $\xi$ satisfies Definition~\ref{def:main}:~\ref{it:131}.\\

 \item \textbf{One-dimensional cell next to two lateral cells.} 
Assume that $\xi$ is an edge of $T_1$ directly adjacent to two lateral cells $\sigma$ and $\sigma'$, 
and let $\gamma$ and $\gamma'$ be the two cells described in Definition~\ref{def:main}:~\ref{it:132}. We assume furthermore that $\xi$ is a vertical half-line inside an edge of $T_1$. The proof of the more cumbersome general case (where $\xi$  has finite slope) will be omitted as it follows closely this situation. Let $\zeta$ denote the $0$-cell in $\Xi$, adjacent to $\xi$. That is, $\zeta$ is the endpoint of $\xi$. For this case, we will drop the assumption that $(0,0)\in\xi$, and assume instead that $\zeta=(0,0)$. Then, we have 
\[
f = (f_{1}|_\zeta +  H_1,~f_{2}|_\zeta +  H_2),
\] where $H_i\in \K[z_1,z_2]$ collects monomial terms of $f_i$, whose coefficients have smaller valuations.

Assume first that $\zeta$ is a vertex of $T_1$. 
As all the vertexes of $T_1$ are trivalent (see Lemma~\ref{lem:gen-valued_transversal}), there exists $k,l,m,r\in\N$ and $a,b,c,d\in \K\setminus 0$ satisfying
\[
\overline{f} = \big(\overline{a}z_1^k +\overline{b}z_1^{l} + \overline{c}z^{(m,n)},~\overline{d}z_1^r\big),
\]
yielding 
\begin{equation}\label{eq:loc-Jac-vertex}
z^{\bm 1}\cdot\overline{J}_\zeta = \left|
\begin{matrix}
\overline{a}kz_1^{k} + \overline{b}lz_1^{l} + \overline{c}mz^{(m,n)}  & \overline{d}rz_1^{r} \\[4pt]

\overline{c}nz^{(m,n)} & 0
\end{matrix}
 \right| = \overline{c}\overline{d}rnz^{(l+m,n)}. 
\end{equation} Since $\xi$ is a vertical half-line situated below its endpoint $(0,0)$, computing 
\[
\overline{\cj}^{\lambda}_f:=\overline{|\Jac_{(z_1,t^\lambda z_2)} f|}
\] for some $\lambda>0$ provides the local equations for $C(F)\cap \xi$. Equation~\eqref{eq:loc-Jac-vertex} shows that the support of $\overline{J}_\zeta$ is $(l+m-1,n-1)\in\N^2$. Furthermore, one can check from Definition~\ref{def:essential} that 
$\gamma$ is essential to $\sigma$ if and only if the integer $n$ (appearing as in the exponent vector $(m,n)$ if $f_{1,\zeta}$) coincides with
\begin{equation}\label{eq:n_is_min}
M_0:=\min(w_2~|~(w_1,w_2)\in A_1\cup A_2).
\end{equation} If $n=M_0$, then for any $\lambda\in\R_+$ above, the value $\val(cdlz^{(l+m,\lambda n)}) =-n\lambda$ will always be strictly larger than the valuation of any other term, in $|\Jac_{(z_1,t^\lambda z_2)} f|$. 
Hence, the polynomial $\overline{\cj}^{\lambda}_f$ is a monomial $\overline{c}\overline{d}lz^{(l+m,n)}$.
 Consequently, Corollary~\ref{cor:lower-order} implies $C(F)\cap\xi=\emptyset$. Otherwise, if $n>M_0$, there exists $\lambda_0>0$ so that  $\val(cdlz^{(l+m,\lambda_0 n)}) =-n\lambda_0 = -M_0\lambda_0 +\val(c'd') = \val(c'd'lz^{(m',\lambda_0 M_0)})$, for some term $c'd'lz^{(u,\lambda_0 M_0)}$ in $\Jac_{(z_1,t^{\lambda_0} z_2)} f$. Hence, 
 $\overline{\cj}^{\lambda_0}_f$ has at least two monomial terms 
\[
\overline{c}\overline{d}rnz^{(l+m,n)}+\overline{c'}\overline{d'}z^{(u,M_0)}.
\] Similarly, Corollary~\ref{cor:lower-order} shows that $C(F)\cap\xi\neq\emptyset$. 

As for the statement (2) of Proposition~\ref{prop:t-crit_mixed-cells}, the discussion above shows that if the intersection $\xi\cap C(F)$ is non-empty, it can be locally expressed as
\[
\VT(az_1^k + bz_1^l +cz^{(m,n)})\cap\VT(cdrnz^{(l+m,n)} + c'd'z^{(u,M_0)}).
\] If this intersection is unbounded, then
\[
k=l~\text{and}~\val (c d) - \val(c'd')= \val (a) - \val (b).
\] We will define $\Omega_3$ from Proposition~\ref{prop:t-crit_mixed-cells} to exclude those (finitely-many) linear conditions. Using the same arguments as in Lemma~\ref{lem:crit-points-invar}, we can prove that $\Omega_3$ is Zariski open in $\K[A]$.  Therefore, Corollary~\ref{cor:lower-order} yields Proposition~\ref{prop:t-crit_mixed-cells} (1) and (2) whenever $\zeta$ is a vertex of $T_i$ satisfying Definition~\ref{def:main}:~\ref{it:132}.\\
 
Assume now that $\zeta\in T_1\cap T_2$. Lemma~\ref{lem:gen-valued_transversal} shows that there exists $k,l,m,n\in\N$ and $a,b,c,d\in\K\setminus 0$ such that  
\[
z^{\bm 1}\cdot \overline{J}_\zeta = 
\left|
\begin{matrix}
\overline{a}kz_1^k + \overline{b}lz_1^l  & \overline{c}rz_1^r + \overline{d}mz^{(m,n)}\\[4pt]

0 & \overline{d}mz^{(m,n)}
\end{matrix}
 \right| = \overline{d}mz^{(m,n)}(k\overline{a}z_1^k + \overline{b}lz_1^l).
\] Similarly as for the case above (where $p$ was a vertex of $T_1$), we will compute $\overline{\cj}_f^{\lambda}$ for all $\lambda>0$. Recall Equality~\eqref{eq:n_is_min}. If $n=M_0$, then both $\gamma$ and $\gamma'$ are essential to $\sigma$ and $\sigma'$ respectively. Thus, for any $\lambda\in\R_+$ above, the polynomial $\overline{\cj}_f^{\lambda}$ is the sum of exactly two monomials as above. Corollary~\ref{cor:lower-order} shows that $C(F)\cap\xi$ is a half-line. Otherwise, if $n>M_0$, then none of $\gamma$ or $\gamma'$ is essential to $\sigma$ or $\sigma'$ (see Definition~\ref{def:essential}). Hence, there exists $\lambda_0>0$ such that $\overline{\cj}_f^{\lambda_0}$ has at least three monomials 
\[
\overline{d}mz_1^{n-1}(k\overline{a}z_1^{k-1} + \overline{b}lz_1^{l-1}) + c_0z^{(u,M_0-1)}
\] for some $c_0\in \C^*$ and $u\in\N$. Corollary~\ref{cor:lower-order} implies that $C(F)$ has a vertex in $\xi$. Finally, to ensure that proving that $C(F)$ has no unbounded components in $\xi$, the coordinates of $\Val(f)$ have to avoid finitely-many linear combinations so that $\overline{\cj}_f^{\lambda_0}$ has no more than three monomials. We then choose the above $\Omega_3$ so that any $f\in\Omega_3$ avoids these equalities. This, proves Proposition~\ref{prop:t-crit_mixed-cells} (1) and (2) whenever $\zeta\in T_1\cap T_2$ and $\xi$  satisfies Definition~\ref{def:main}:~\ref{it:132}.\\

 \item \textbf{Irrelevant cells.}
From Lemma~\ref{lem:gen-valued_transversal}, if, for any $(y_1,y_2)\in\R^2$, the curves $T_1(y_1),T_2(y_2)$ intersect at a point $x$ in an irrelevant cell $\xi$, then $x$ is a stable intersection. Thus Lemma~\ref{lem:transversal} yields the proof of Proposition~\ref{prop:t-crit_mixed-cells} (3) whenever $\xi$ satisfies Definition~\ref{def:main}:~\ref{it:21}.\\

\item \textbf{Diagonal relevant cells.} 
Let $\xi$ be a diagonal relevant cell in $\Xi$. Then, for any $y\in F(\xi)$, the set 
\[
L_y:=F^{-1}(y)\cap \xi
\] is a line segment orthogonal to the line $L_0$ in Lemma~\ref{lem:adjacent_diag_relevant}. Let $x,x'$ be the two endpoints of $L_y$, that are at the boundary of $\xi$, so that the second coordinate of $x$ is bigger than that of $x'$ (see Figure~\ref{fig:it-7}). Then, we have $x\in \gamma$, and $x'\in\gamma'$, where $\gamma,\gamma'\in\Xi$ are $1$-cells directly adjacent to $\xi$. Lemma~\ref{lem:adjacent_diag_relevant} shows that each of $\gamma$ and $\gamma'$ is adjacent to only one relevant $2$-cell. In fact, this $2$-cell is $\xi$ (see Figure~\ref{fig:it-7}). 

Proposition~\ref{prop:t-crit_mixed-cells}, applied to each of $\gamma$ and $\gamma'$ shows that $x,x'\not\in C(F)$, and thus the tropical polynomial $\cj_f^{\trop}$ reaches its maximum at only one tropical monomial $\max(\langle x, m\rangle+ \val (c_m) $ when evaluated at $x$, and it does so at a monomial $\max(\langle x', n\rangle+ \val (d_n) $ when evaluated at $x'$. In what follows, we will show that those two monomials are different. By continuity of tropical polynomial functions, it will follow that the segment $L_x$ intersects $C(F)$.

Let us first compute $c_m z^m$. Assume without loss of generality that $\gamma\subset T_1$, and that $x=(0,0)$. This computation has been essentially done in Item (3) above, where $\delta(\xi_1) = (k_1,k_2)$ has the same direction as $\delta(\xi_2) = (v_1,v_2)$. Hence, we have $m$ is the vector sum of $k$, and $u$.

The same analysis holds true for the vector $n$, where the latter is the sum of the pair of vectors $k$, and $w$  or $l$ and $v$ (depending whether $x'$ lies $ T_1$ or in $T_2$), are dual to the cell $\xi''$ adjacent to $\sigma'$ (see Figure~\ref{fig:it-7}). As $\xi$ is diagonal, Definition~\ref{def:lat-diag} shows that  both $2$-cells $\xi'$, and $\xi''$ are irrelevant. Since additionally, $\sigma$, and $\sigma'$ lie on different sides of $\xi$, one concludes that $m:= k+u$, and $n:= k+w$ (or $l+v$), are two vectors having different directions. This proves Proposition~\ref{prop:t-crit_mixed-cells} when $\xi$ satisfies Definition~\ref{def:main}~\ref{it:221}.\\

\begin{figure}
\center

\tikzset{every picture/.style={line width=0.75pt}} 

\begin{tikzpicture}[x=1.25pt,y=1.25pt,yscale=-1,xscale=1]

\draw  [color={rgb, 255:red, 0; green, 0; blue, 0 }  ,draw opacity=0 ][fill={rgb, 255:red, 248; green, 231; blue, 28 }  ,fill opacity=0.48 ] (134.2,207.02) -- (142.2,231.02) -- (63.55,243.27) -- (56.31,229.95) -- cycle ;
\draw [color={rgb, 255:red, 74; green, 144; blue, 226 }  ,draw opacity=1 ]   (56.31,229.95) -- (134.2,207.02) ;
\draw [color={rgb, 255:red, 79; green, 126; blue, 24 }  ,draw opacity=1 ][fill={rgb, 255:red, 65; green, 117; blue, 5 }  ,fill opacity=1 ][line width=1.5]    (117.46,235.13) -- (108.68,214.48) ;
\draw [color={rgb, 255:red, 74; green, 144; blue, 226 }  ,draw opacity=1 ]   (63.55,243.27) -- (142.2,231.02) ;
\draw  [color={rgb, 255:red, 0; green, 0; blue, 0 }  ,draw opacity=1 ][fill={rgb, 255:red, 0; green, 0; blue, 0 }  ,fill opacity=1 ] (116.46,235.13) .. controls (116.46,234.58) and (116.9,234.13) .. (117.46,234.13) .. controls (118.01,234.13) and (118.46,234.58) .. (118.46,235.13) .. controls (118.46,235.68) and (118.01,236.13) .. (117.46,236.13) .. controls (116.9,236.13) and (116.46,235.68) .. (116.46,235.13) -- cycle ;
\draw  [color={rgb, 255:red, 0; green, 0; blue, 0 }  ,draw opacity=1 ][fill={rgb, 255:red, 0; green, 0; blue, 0 }  ,fill opacity=1 ] (107.68,214.48) .. controls (107.68,213.93) and (108.13,213.48) .. (108.68,213.48) .. controls (109.23,213.48) and (109.68,213.93) .. (109.68,214.48) .. controls (109.68,215.04) and (109.23,215.48) .. (108.68,215.48) .. controls (108.13,215.48) and (107.68,215.04) .. (107.68,214.48) -- cycle ;
\draw  [color={rgb, 255:red, 0; green, 0; blue, 0 }  ,draw opacity=0 ][fill={rgb, 255:red, 248; green, 231; blue, 28 }  ,fill opacity=0.48 ] (234.2,206.57) -- (242.2,230.57) -- (163.55,242.82) -- (156.31,229.51) -- cycle ;
\draw [color={rgb, 255:red, 74; green, 144; blue, 226 }  ,draw opacity=1 ]   (156.31,229.51) -- (234.2,206.57) ;
\draw [color={rgb, 255:red, 79; green, 126; blue, 24 }  ,draw opacity=1 ][fill={rgb, 255:red, 65; green, 117; blue, 5 }  ,fill opacity=1 ][line width=1.5]    (217.46,234.69) -- (208.68,214.04) ;
\draw [color={rgb, 255:red, 0; green, 0; blue, 0 }  ,draw opacity=1 ] [dash pattern={on 3.75pt off 0.75pt}]  (163.55,242.82) -- (242.2,230.57) ;
\draw  [color={rgb, 255:red, 0; green, 0; blue, 0 }  ,draw opacity=1 ][fill={rgb, 255:red, 0; green, 0; blue, 0 }  ,fill opacity=1 ] (216.46,234.69) .. controls (216.46,234.13) and (216.9,233.69) .. (217.46,233.69) .. controls (218.01,233.69) and (218.46,234.13) .. (218.46,234.69) .. controls (218.46,235.24) and (218.01,235.69) .. (217.46,235.69) .. controls (216.9,235.69) and (216.46,235.24) .. (216.46,234.69) -- cycle ;
\draw  [color={rgb, 255:red, 0; green, 0; blue, 0 }  ,draw opacity=1 ][fill={rgb, 255:red, 0; green, 0; blue, 0 }  ,fill opacity=1 ] (207.68,214.04) .. controls (207.68,213.49) and (208.13,213.04) .. (208.68,213.04) .. controls (209.23,213.04) and (209.68,213.49) .. (209.68,214.04) .. controls (209.68,214.59) and (209.23,215.04) .. (208.68,215.04) .. controls (208.13,215.04) and (207.68,214.59) .. (207.68,214.04) -- cycle ;

\draw (88.11,250.79) node [anchor=north west][inner sep=0.75pt]  [font=\small]  {$( w,k)$};
\draw (80.68,226.35) node [anchor=north west][inner sep=0.75pt]  [font=\small]  {$( v,k)$};
\draw (81.54,200.06) node [anchor=north west][inner sep=0.75pt]  [font=\small]  {$( u,k)$};
\draw (59.79,218.1) node [anchor=north west][inner sep=0.75pt]  [font=\small]  {$\gamma $};
\draw (110.86,202.81) node [anchor=north west][inner sep=0.75pt]  [font=\small]  {$p$};
\draw (119.46,237.53) node [anchor=north west][inner sep=0.75pt]  [font=\small]  {$p'$};
\draw (189.86,251.1) node [anchor=north west][inner sep=0.75pt]  [font=\small]  {$( v,l)$};
\draw (182.43,226.66) node [anchor=north west][inner sep=0.75pt]  [font=\small]  {$( v,k)$};
\draw (183.29,200.37) node [anchor=north west][inner sep=0.75pt]  [font=\small]  {$( u,k)$};
\draw (65.55,246.67) node [anchor=north west][inner sep=0.75pt]  [font=\small]  {$\gamma '$};

\end{tikzpicture}

\caption{the intersection of the preimage $F^{-1}(q)$ with a relevant diagonal cell}\label{fig:it-7}

\end{figure}
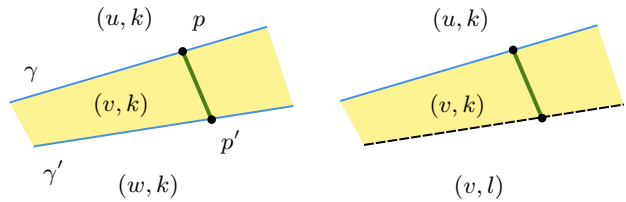
 \item \textbf{Lateral cells.} 
In this part of the proof, we suppose that $\xi$ is a lateral $2$-cell. For simplicity, we proceed similarly to Item (5) by assuming that $\delta(\xi_1)+\delta(\xi_2)$ is a vector whose second coordinate is zero. The cases where $\delta(\xi_1)+\delta(\xi_2)$ is an arbitrary vector in $\N^2$ follow the same steps as in our assumption, albeit with more cumbersome notations, and thus we omit it here.

Let $n\in\N$ be so that $\delta(\xi_1)+\delta(\xi_2)=(n,0)$. For any point $y\in F(\xi)$, we have $F^{-1}(y)$ is a vertical half-line $L_y$, with an endpoint $\zeta\in \gamma$, where $\gamma$ is a $1$-cell in $\Xi$, directly adjacent to $\xi$ (see Figure~\ref{fig:lateral_cell}). 
For this case, we will drop the assumption that $(0,0)\in\xi$, and assume instead that $\zeta=(0,0)$. If $\gamma\subset T_1$, there exists $k,m,n,r\in\N$ and $a,c,d\in \K\setminus 0$ satisfying
\[
\overline{f} = (\overline{a}z_1^k +\overline{c}z^{(m,n)} ,~\overline{d}z_1^r).
\] Hence, we obtain
\begin{equation}\label{eq:loc-Jac-vertex2}
z^{\bm 1}\cdot\overline{J}_p = \left|
\begin{matrix}
\overline{a}kz_1^{k} + \overline{c}mz^{(m,n)}  & \overline{d}rz_1^{r} \\[4pt]

\overline{c}nz^{(m,n)} & 0
\end{matrix}
 \right| = \overline{c}\overline{d}rnz^{(l+m,n)}. 
\end{equation} The proof follows closely that of Item (5), where the intersection $ L_y\cap C(F)$ is non-empty if and only if the essentiality condition is met. This proves Proposition~\ref{prop:t-crit_mixed-cells} (3) when $\xi$ satisfies Definition~\ref{def:main}:~\ref{it:222}.

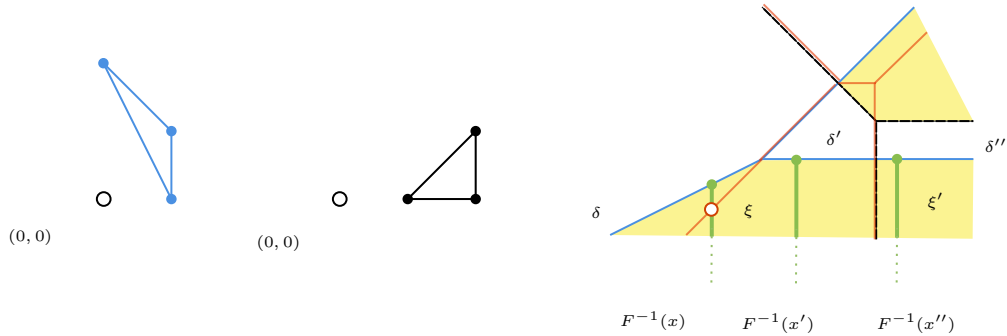
\begin{figure}

\tikzset{every picture/.style={line width=0.75pt}} 

\begin{tikzpicture}[x=1.6pt,y=1.6pt,yscale=-1,xscale=1]

\draw  [color={rgb, 255:red, 0; green, 0; blue, 0 }  ,draw opacity=0 ][fill={rgb, 255:red, 248; green, 231; blue, 28 }  ,fill opacity=0.48 ] (223.63,72.52) -- (274.2,72.22) -- (274.02,91.07) -- (188.33,90.19) -- cycle ;
\draw  [color={rgb, 255:red, 0; green, 0; blue, 0 }  ,draw opacity=0 ][fill={rgb, 255:red, 248; green, 231; blue, 28 }  ,fill opacity=0.48 ] (260.2,36.29) -- (274.18,63.24) -- (251.22,63.24) -- (242.23,54.25) -- cycle ;
\draw [color={rgb, 255:red, 74; green, 144; blue, 226 }  ,draw opacity=1 ]   (68.53,49.54) -- (84.6,65.61) ;
\draw [color={rgb, 255:red, 74; green, 144; blue, 226 }  ,draw opacity=1 ]   (84.6,65.61) -- (84.6,81.68) ;
\draw [color={rgb, 255:red, 74; green, 144; blue, 226 }  ,draw opacity=1 ]   (68.53,49.54) -- (84.6,81.68) ;
\draw [color={rgb, 255:red, 0; green, 0; blue, 0 }  ,draw opacity=1 ][fill={rgb, 255:red, 0; green, 0; blue, 0 }  ,fill opacity=1 ]   (140.47,81.68) -- (156.53,81.68) ;
\draw [color={rgb, 255:red, 0; green, 0; blue, 0 }  ,draw opacity=1 ][fill={rgb, 255:red, 0; green, 0; blue, 0 }  ,fill opacity=1 ]   (140.47,81.68) -- (156.53,65.61) ;
\draw  [color={rgb, 255:red, 74; green, 144; blue, 226 }  ,draw opacity=1 ][fill={rgb, 255:red, 74; green, 144; blue, 226 }  ,fill opacity=1 ] (83.6,81.68) .. controls (83.6,81.13) and (84.05,80.68) .. (84.6,80.68) .. controls (85.15,80.68) and (85.6,81.13) .. (85.6,81.68) .. controls (85.6,82.23) and (85.15,82.68) .. (84.6,82.68) .. controls (84.05,82.68) and (83.6,82.23) .. (83.6,81.68) -- cycle ;
\draw  [color={rgb, 255:red, 0; green, 0; blue, 0 }  ,draw opacity=1 ][fill={rgb, 255:red, 0; green, 0; blue, 0 }  ,fill opacity=1 ] (155.53,81.68) .. controls (155.53,81.13) and (155.98,80.68) .. (156.53,80.68) .. controls (157.09,80.68) and (157.53,81.13) .. (157.53,81.68) .. controls (157.53,82.23) and (157.09,82.68) .. (156.53,82.68) .. controls (155.98,82.68) and (155.53,82.23) .. (155.53,81.68) -- cycle ;
\draw  [color={rgb, 255:red, 0; green, 0; blue, 0 }  ,draw opacity=1 ][fill={rgb, 255:red, 0; green, 0; blue, 0 }  ,fill opacity=1 ] (139.47,81.68) .. controls (139.47,81.13) and (139.91,80.68) .. (140.47,80.68) .. controls (141.02,80.68) and (141.47,81.13) .. (141.47,81.68) .. controls (141.47,82.23) and (141.02,82.68) .. (140.47,82.68) .. controls (139.91,82.68) and (139.47,82.23) .. (139.47,81.68) -- cycle ;
\draw  [color={rgb, 255:red, 0; green, 0; blue, 0 }  ,draw opacity=1 ][fill={rgb, 255:red, 255; green, 255; blue, 255 }  ,fill opacity=1 ] (66.98,81.68) .. controls (66.98,80.78) and (67.71,80.06) .. (68.6,80.06) .. controls (69.49,80.06) and (70.22,80.78) .. (70.22,81.68) .. controls (70.22,82.57) and (69.49,83.29) .. (68.6,83.29) .. controls (67.71,83.29) and (66.98,82.57) .. (66.98,81.68) -- cycle ;
\draw [color={rgb, 255:red, 0; green, 0; blue, 0 }  ,draw opacity=1 ][fill={rgb, 255:red, 0; green, 0; blue, 0 }  ,fill opacity=1 ]   (156.53,81.68) -- (156.53,65.61) ;
\draw  [color={rgb, 255:red, 0; green, 0; blue, 0 }  ,draw opacity=1 ][fill={rgb, 255:red, 0; green, 0; blue, 0 }  ,fill opacity=1 ] (155.53,65.61) .. controls (155.53,65.06) and (155.98,64.61) .. (156.53,64.61) .. controls (157.09,64.61) and (157.53,65.06) .. (157.53,65.61) .. controls (157.53,66.16) and (157.09,66.61) .. (156.53,66.61) .. controls (155.98,66.61) and (155.53,66.16) .. (155.53,65.61) -- cycle ;
\draw  [color={rgb, 255:red, 74; green, 144; blue, 226 }  ,draw opacity=1 ][fill={rgb, 255:red, 74; green, 144; blue, 226 }  ,fill opacity=1 ] (83.6,65.61) .. controls (83.6,65.06) and (84.05,64.61) .. (84.6,64.61) .. controls (85.15,64.61) and (85.6,65.06) .. (85.6,65.61) .. controls (85.6,66.16) and (85.15,66.61) .. (84.6,66.61) .. controls (84.05,66.61) and (83.6,66.16) .. (83.6,65.61) -- cycle ;
\draw  [color={rgb, 255:red, 74; green, 144; blue, 226 }  ,draw opacity=1 ][fill={rgb, 255:red, 74; green, 144; blue, 226 }  ,fill opacity=1 ] (67.53,49.54) .. controls (67.53,48.99) and (67.98,48.54) .. (68.53,48.54) .. controls (69.09,48.54) and (69.53,48.99) .. (69.53,49.54) .. controls (69.53,50.1) and (69.09,50.54) .. (68.53,50.54) .. controls (67.98,50.54) and (67.53,50.1) .. (67.53,49.54) -- cycle ;
\draw  [color={rgb, 255:red, 0; green, 0; blue, 0 }  ,draw opacity=1 ][fill={rgb, 255:red, 255; green, 255; blue, 255 }  ,fill opacity=1 ] (122.78,81.68) .. controls (122.78,80.78) and (123.51,80.06) .. (124.4,80.06) .. controls (125.29,80.06) and (126.02,80.78) .. (126.02,81.68) .. controls (126.02,82.57) and (125.29,83.29) .. (124.4,83.29) .. controls (123.51,83.29) and (122.78,82.57) .. (122.78,81.68) -- cycle ;
\draw [color={rgb, 255:red, 74; green, 144; blue, 226 }  ,draw opacity=1 ]   (188.33,90.19) -- (224.27,72.22) ;
\draw [color={rgb, 255:red, 74; green, 144; blue, 226 }  ,draw opacity=1 ]   (224.27,72.22) -- (260.2,36.29) ;
\draw [color={rgb, 255:red, 74; green, 144; blue, 226 }  ,draw opacity=1 ]   (274.2,72.22) -- (224.27,72.22) ;
\draw [color={rgb, 255:red, 228; green, 55; blue, 4 }  ,draw opacity=0.6 ]   (206.3,90.19) -- (224.27,72.22) ;
\draw [color={rgb, 255:red, 228; green, 55; blue, 4 }  ,draw opacity=0.6 ]   (223.63,72.52) -- (241.6,54.55) ;
\draw [color={rgb, 255:red, 0; green, 0; blue, 0 }  ,draw opacity=1 ][fill={rgb, 255:red, 0; green, 0; blue, 0 }  ,fill opacity=1 ] [dash pattern={on 3.75pt off 0.75pt}]  (251.22,63.24) -- (242.23,54.25) ;
\draw [color={rgb, 255:red, 0; green, 0; blue, 0 }  ,draw opacity=1 ][fill={rgb, 255:red, 0; green, 0; blue, 0 }  ,fill opacity=1 ] [dash pattern={on 3.75pt off 0.75pt}]  (242.23,54.25) -- (224.27,36.29) ;
\draw [color={rgb, 255:red, 228; green, 55; blue, 4 }  ,draw opacity=0.6 ]   (224.6,35.62) -- (242.57,53.59) ;
\draw [color={rgb, 255:red, 228; green, 55; blue, 4 }  ,draw opacity=0.6 ]   (242.23,54.25) -- (251.2,54.25) ;
\draw [color={rgb, 255:red, 0; green, 0; blue, 0 }  ,draw opacity=1 ][fill={rgb, 255:red, 0; green, 0; blue, 0 }  ,fill opacity=1 ] [dash pattern={on 3.75pt off 0.75pt}]  (251.22,63.24) -- (274.18,63.24) ;
\draw [color={rgb, 255:red, 0; green, 0; blue, 0 }  ,draw opacity=1 ][fill={rgb, 255:red, 0; green, 0; blue, 0 }  ,fill opacity=1 ] [dash pattern={on 3.75pt off 0.75pt}]  (251.22,91.2) -- (251.22,63.24) ;
\draw [color={rgb, 255:red, 228; green, 55; blue, 4 }  ,draw opacity=0.6 ]   (250.87,91.22) -- (250.87,54.25) ;
\draw [color={rgb, 255:red, 228; green, 55; blue, 4 }  ,draw opacity=0.6 ]   (250.87,54.25) -- (263.23,42.22) ;
\draw [color={rgb, 255:red, 137; green, 189; blue, 79 }  ,draw opacity=1 ][line width=1.5]    (212.35,90.55) -- (212.35,78.24) ;
\draw [color={rgb, 255:red, 150; green, 191; blue, 104 }  ,draw opacity=1 ][line width=0.75]  [dash pattern={on 0.84pt off 2.51pt}]  (212.26,101.44) -- (212.35,88.55) ;
\draw  [color={rgb, 255:red, 137; green, 189; blue, 79 }  ,draw opacity=1 ][fill={rgb, 255:red, 137; green, 189; blue, 79 }  ,fill opacity=1 ] (211.35,78.24) .. controls (211.35,77.69) and (211.8,77.24) .. (212.35,77.24) .. controls (212.9,77.24) and (213.35,77.69) .. (213.35,78.24) .. controls (213.35,78.79) and (212.9,79.24) .. (212.35,79.24) .. controls (211.8,79.24) and (211.35,78.79) .. (211.35,78.24) -- cycle ;
\draw [color={rgb, 255:red, 137; green, 189; blue, 79 }  ,draw opacity=1 ][line width=1.5]    (232.35,90.74) -- (232.35,72.43) ;
\draw [color={rgb, 255:red, 150; green, 191; blue, 104 }  ,draw opacity=1 ][line width=0.75]  [dash pattern={on 0.84pt off 2.51pt}]  (232.26,101.63) -- (232.35,88.74) ;
\draw  [color={rgb, 255:red, 137; green, 189; blue, 79 }  ,draw opacity=1 ][fill={rgb, 255:red, 137; green, 189; blue, 79 }  ,fill opacity=1 ] (231.35,72.43) .. controls (231.35,71.88) and (231.8,71.43) .. (232.35,71.43) .. controls (232.9,71.43) and (233.35,71.88) .. (233.35,72.43) .. controls (233.35,72.98) and (232.9,73.43) .. (232.35,73.43) .. controls (231.8,73.43) and (231.35,72.98) .. (231.35,72.43) -- cycle ;
\draw [color={rgb, 255:red, 137; green, 189; blue, 79 }  ,draw opacity=1 ][line width=1.5]    (256.15,90.81) -- (256.15,72.23) ;
\draw [color={rgb, 255:red, 150; green, 191; blue, 104 }  ,draw opacity=1 ][line width=0.75]  [dash pattern={on 0.84pt off 2.51pt}]  (256.06,101.43) -- (256.15,88.54) ;
\draw  [color={rgb, 255:red, 137; green, 189; blue, 79 }  ,draw opacity=1 ][fill={rgb, 255:red, 137; green, 189; blue, 79 }  ,fill opacity=1 ] (255.15,72.23) .. controls (255.15,71.68) and (255.6,71.23) .. (256.15,71.23) .. controls (256.7,71.23) and (257.15,71.68) .. (257.15,72.23) .. controls (257.15,72.78) and (256.7,73.23) .. (256.15,73.23) .. controls (255.6,73.23) and (255.15,72.78) .. (255.15,72.23) -- cycle ;
\draw  [color={rgb, 255:red, 208; green, 65; blue, 2 }  ,draw opacity=1 ][fill={rgb, 255:red, 255; green, 255; blue, 255 }  ,fill opacity=1 ] (210.75,84.14) .. controls (210.75,83.31) and (211.42,82.64) .. (212.24,82.64) .. controls (213.07,82.64) and (213.74,83.31) .. (213.74,84.14) .. controls (213.74,84.96) and (213.07,85.63) .. (212.24,85.63) .. controls (211.42,85.63) and (210.75,84.96) .. (210.75,84.14) -- cycle ;

\draw (45.33,88.07) node [anchor=north west][inner sep=0.75pt]  [font=\tiny]  {$( 0,0)$};
\draw (104,89.4) node [anchor=north west][inner sep=0.75pt]  [font=\tiny]  {$( 0,0)$};
\draw (190,107.4) node [anchor=north west][inner sep=0.75pt]  [font=\tiny]  {$F^{-1}( x)$};
\draw (218.67,108.07) node [anchor=north west][inner sep=0.75pt]  [font=\tiny]  {$F^{-1}( x')$};
\draw (250.67,108.07) node [anchor=north west][inner sep=0.75pt]  [font=\tiny]  {$F^{-1}( x'')$};
\draw (219.24,82.07) node [anchor=north west][inner sep=0.75pt]  [font=\tiny]  {$\xi $};
\draw (262.67,79.48) node [anchor=north west][inner sep=0.75pt]  [font=\tiny]  {$\xi '$};
\draw (183.24,82.89) node [anchor=north west][inner sep=0.75pt]  [font=\tiny]  {$\delta $};
\draw (238.67,64.89) node [anchor=north west][inner sep=0.75pt]  [font=\tiny]  {$\delta '$};
\draw (276.18,66.64) node [anchor=north west][inner sep=0.75pt]  [font=\tiny]  {$\delta ''$};

\end{tikzpicture}

\caption{The $1$-cell $\delta$ is essential to the lateral cell $\xi$. The $1$-cells $\gamma'$ and $\gamma''$ are not essential to $\xi$ and $\xi'$ respectively. The vertical green half-lines are preimages of three different points $x,x'\in F(\xi)$, and $x''\in F(\xi')$. The tropical curve in red is the set $C_F$.}\label{fig:lateral_cell}
\end{figure}

\end{enumerate}

\subsection*{Contact}
  Boulos El Hilany,\\
  Institut f\"ur Analysis und Algebra, \\
TU Braunschweig, Universit\"atsplatz 2. 38106 Braunschweig, Germany.\\
 \href{mailto:b.el-hilany@tu-braunschweig.de}{b.el-hilany@tu-braunschweig.de},\\
 \href{https://boulos-elhilany.com}{boulos-elhilany.com}

\bibliographystyle{alpha}					   
\def\cprime{$'$}

\end{document}